\documentclass{amsart}

\usepackage{latexsym,amsmath,amstext,mathrsfs,amsthm,amscd,amsopn,verbatim,amsrefs,tikz}
\tikzstyle{n}=[circle, draw, fill, minimum size=6, inner sep=0]
\tikzstyle{ext}=[circle, draw,  minimum size=3, inner sep=0]
\tikzstyle{int}=[circle, draw, fill, minimum size=3, inner sep=0]
\tikzset{diagram/.style={matrix of math nodes, row sep=3em, column sep=2.5em, text height=1.5ex, text depth=0.25ex}}
\tikzset{diagram2/.style={matrix of math nodes, row sep=0.5em, column sep=0.5em, text height=1.5ex, text depth=0.25ex}}
\usetikzlibrary{matrix,arrows}

\usepackage{hyperref}

\usepackage{graphicx}
\usepackage[all]{xy}
\usepackage{amscd}
\usepackage{amssymb}
\usepackage{MnSymbol}

\newtheorem{Thm}{Theorem}[section]
\newtheorem{Prop}[Thm]{Proposition}
\newtheorem{Lem}[Thm]{Lemma}
\newtheorem{Cor}[Thm]{Corollary}      
\newtheorem{Conj}[Thm]{Conjecture} 

\theoremstyle{remark}
\newtheorem{Rem}[Thm]{Remark}

\newcommand{\alg}[1]{\mathfrak{#1}}

\theoremstyle{definition}

\newtheorem{Def}[Thm]{Definition}

\newcommand{\K}{\mathbb C}
\newcommand{\Z}{\mathbb Z}

\newcommand{\Hom}{\mathrm{Hom}}

\newcommand{\dgra}{\mathrm{dgra}}
\newcommand{\dGra}{\mathsf{dGra}}

\newcommand{\Gra}{\mathsf{Gra}}

\newcommand{\Graphs}{\mathsf{Graphs}}
\newcommand{\KGra}{\mathsf{KGra}}
\newcommand{\kgra}{\mathrm{kgra}}

\newcommand{\Defo}{\mathsf{Def}}

\newcommand{\Dpoly}{D_{\rm{poly}}}

\newcommand{\OC}{\mathsf{OC}}
\newcommand{\dfGC}{\mathsf{dfGC}}
\newcommand{\fGC}{\mathsf{fGC}}
\newcommand{\fdGC}{\mathsf{dfGC}}
\newcommand{\ICG}{\mathsf{ICG}}
\newcommand{\GC}{\mathsf{GC}}
\newcommand{\SNbr}{\Gamma_{\bullet\!-\!\!\bullet}}
\newcommand{\grcup}{\Gamma_{\circ\!\phantom{-}\!\!\circ}}
\newcommand{\grHKR}{\Gamma_\mathrm{HKR}}

\newcommand{\tder}{\mathfrak{tder}}
\newcommand{\sder}{\mathfrak{sder}}
\newcommand{\TAut}{\mathsf{TAut}}
\newcommand{\SAut}{\mathsf{SAut}}
\newcommand{\grt}{\mathfrak{grt}}
\newcommand{\GRT}{\mathsf{GRT}}
\newcommand{\Lie}{\mathfrak{Lie}}
\newcommand{\Conf}{{\mathrm{Conf}}}

\newcommand{\p}{{\partial}}
\newcommand{\C}{{\rm C}}

\newcommand{\mU}{{\mathcal U}}

\newcommand{\op}[1]{{\mathsf{#1}}}

\newcommand{\psgn}[1]{{(-1)^{o(#1)}}}

\newcommand{\Exp}{{\mathrm{Exp}}}

\newcommand{\bbC}{{\mathbb{C}}}
\newcommand{\R}{{\mathbb{R}}}
\newcommand{\bbR}{{\mathbb{R}}}
\newcommand{\bbN}{{\mathbb{N}}}

\newcommand{\Li}{{\mathrm{Li}}}

\newcommand{\ba}{\begin{align*}}
\newcommand{\ea}{\end{align*}}
\newcommand{\be}[1]{\begin{equation}\label{#1}}
\newcommand{\ee}{\end{equation}}

\newcommand{\LaLie}{\Lambda\Lie}

\newcommand{\Reg}{{\mathrm{Reg}}}

\begin{document}

\title[P.~Etingof's conjecture about Drinfel{\cprime}d associators]{P.~Etingof's conjecture about Drinfel{\cprime}d associators}
\author{Carlo~A. Rossi}

\author{Thomas Willwacher}


\begin{abstract}
We construct a family of Drinfel{\cprime}d associators interpolating between the Knizhnik--Zamolodchikov associator, the Alekseev--Torossian associator and the anti-Knizhnik--Zamolodchikov associator.
We give explicit integral formul\ae\ for the family of elements of the Grothendieck-Teichm\"uller Lie algebra tangent to the family of associators. As an application, we settle a conjecture of Pavel Etingof about the Alekseev--Torossian associator.

Furthermore, we give explicit integral formul\ae\ for the family of stable formality morphisms corresponding (in a precise way) to the above family of associators, and for the family of graph cohomology classes corresponding to the above family of elements of the Grothendieck-Teichm\"uller Lie algebra. 
It follows in particular that the ``logarithmic'' Kontsevich formality morphism corresponds to the Knizhnik--Zamolodchikov associator.
\end{abstract}

\maketitle


\tableofcontents

\section{Introduction}\label{intro}
Drinfel{\cprime}d associators \cite{Dr} are algebraic objects which play a central {\em r\^ole} in many constructions in algebra. 
The set of Drinfel{\cprime}d associators is an infinite dimensional pro-algebraic variety, and a torsor for a pro-unipotent group, the Grothendieck-Teichm\"uller group $\GRT_1$.
Despite being widely studied, the set of Drinfel{\cprime}d associators and the group $\GRT_1$ are not yet fully understood. 
In particular, explicit constructions of elements are rare and all use inherently transcendental methods ({\em i.e.} integration). 
The following facts are however known:
\begin{enumerate}
\item[$i)$] there are explicit constructions of three Drinfel{\cprime}d associators, the Knizhnik--Zamolodchikov associator $\Phi_\mathrm{KZ}$, the anti-Knizhnik--Zamolodchikov associator $\Phi_{\overline{\mathrm{KZ}}}$ (see \cite{Dr}) and the Alekseev--Torossian associator $\Phi_\mathrm{AT}$ (see \cites{AT-1, SW}).
\item[$ii)$] Since the set of Drinfel{\cprime}d associators is a $\GRT_1$-torsor, there is a unique element $g\in \GRT_1$, whose action sends $\Phi_\mathrm{KZ}$ to $\Phi_{\overline{\mathrm{KZ}}}$.
\item[$iii)$] Because of the pro-unipotence of $\GRT_1$, there is a unique element $\psi$ in the Lie algebra $\grt_1$ of $\GRT_1$ obeying 
\[
g = \Exp(\psi),
\] 
where $\Exp(\bullet)$ denotes the exponential map from $\grt_1$ to $\GRT_1$.
\item[$iv)$] The Lie algebra $\grt_1$ is $\Z_{\geq 3}$-graded. 
\item[$v)$] Let $\{\sigma_3,\sigma_5,\dots\}$ be the list of the components of $\psi$ of odd degrees $3$, $5$ {\em etc.}: one can check that none of these elements vanishes.
\end{enumerate}
There is the following well-known and hard conjecture.
\begin{Conj}[Deligne-Drinfel{\cprime}d-Ihara conjecture]\label{conj:IDD}
The pro-nilpotent Lie algebra $\grt_1$ is isomorphic to the degree completion of the free Lie algebra $\Lie(\sigma_3,\sigma_5,\dots)$ generated by $\{\sigma_3,\sigma_5,\dots\}$.
\end{Conj}
One half of this conjecture, namely that 
\[
\Lie(\sigma_3,\sigma_5,\dots)\subset \grt_1
\]
has recently been proved by F. Brown \cite{Br-2}.
\begin{Rem}
Observe that by constructing elements of $\grt_1$ using $\{\sigma_3,\sigma_5,\dots\}$ and exponentiating them, we obtain elements of $\GRT_1$, and by acting on $\Phi_\mathrm{KZ}$ we obtain many Drinfel{\cprime}d associators. 
In fact, Conjecture~\ref{conj:IDD} states that any Drinfel{\cprime}d associator may be obtained (uniquely) by this recipe.
\end{Rem}

Note that so far we have used only two of the above three explicitly known Drinfel{\cprime}d associators. 
One might consider the unique element $a$ of $\GRT_1$, whose action maps  $\Phi_\mathrm{KZ}$ to $\Phi_\mathrm{AT}$, and repeat the above analysis using $a$ in place of $g$. 
However, P.~Etingof's conjecture states that essentially ``nothing new'' is obtained in this way.
\begin{Conj}[P. Etingof's conjecture]\label{conj:etingof}
Let $a$, $g$ in $\GRT_1$ and $\{\sigma_3,\sigma_5,\dots\}$ in $\grt_1$ be as above.
\begin{itemize}
\item (Weak form) $\log(a)\in \Lie(\sigma_3,\sigma_5,\dots)\subset \grt_1$.
\item (Strong form) $g=a^2$.
\end{itemize}
\end{Conj}

The first main result of this paper is the following Theorem, which in particular settles Conjecture \ref{conj:etingof}.

\begin{Thm}\label{thm:etingof}
There is a family of Drinfel{\cprime}d associators $\Phi^t$ over $\mathbb R$ and elements $\{\tau_3,\tau_5,\dots\}$ of $\grt_1$ of odd degrees, such that
\begin{align*}
\Phi^0 &=\Phi_\mathrm{KZ} &\Phi^{\frac 1 2} &=\Phi_\mathrm{AT}  & \Phi^1 &=\Phi_{\overline{\mathrm{KZ}}}
\end{align*}
and 
\begin{equation}
\label{equ:Phithor}
\p_t \Phi^t =\tau^t \cdot \Phi^t
\end{equation}
where 
\be{equ:tautdef}
\tau^t := \sum_{j=1}^\infty (t(1-t))^{2j} \tau_{2j+1}\in\grt_1.
\ee
In particular, it follows that the weak form of P. Etingof's Conjecture \ref{conj:etingof} is correct, while the strong form is incorrect.  
\end{Thm}

Furthermore, the Drinfel{\cprime}d associators $\Phi^t$ and the elements $\{\tau_3,\tau_5,\dots\}$ are given by explicit integral expressions.
This may be helpful, as no explicit formul\ae\ for the elements $\{\sigma_3,\sigma_5,\dots\}$ are known.

\begin{Rem}
Note that equations \eqref{equ:Phithor} and \eqref{equ:tautdef} may be used to compute all coefficients occurring in the Alekseev-Torossian associator combinatorially from the known formula for $\Phi_\mathrm{KZ}$ in terms of multiple zeta values \cite{Le-Mu}.
In particular it follows that all coefficients occurring in $\Phi_\mathrm{AT}$ are rational polynomials in multiple zeta values and $\frac 1 {\pi i}$.
\end{Rem}

The second main result of this paper is a theorem similar to Theorem~\ref{thm:etingof}, but in the realm of deformation quantization. 
In deformation quantization, there are similar algebraic structures, which resemble $\GRT_1$, $\grt_1$ and the torsor of Drinfel{\cprime}d associators. 
Concretely, the analogue of the torsor of Drinfel{\cprime}d associators is the set of~\emph{stable formality morphisms} introduced in \cite{Dol}. 
They are acted upon by closed elements of degree $0$ of M.~Kontsevich's graph complex $\GC$, which is a pro-nilpotent differential graded (dg for short) Lie algebra. 
The definitions of these objects will be recalled in more detail in Section~\ref{s-2} below. 
In fact, V. Dolgushev has proven the following Theorem.
\begin{Thm}[V. Dolgushev~\cite{Dol}]
\label{thm:dlgstable}
The exponential group of the $0$-th graph cohomology $\Exp(H^0(\GC))$ acts freely and transitively on the set of homotopy classes of stable formality morphisms.
\end{Thm}

In fact, $H^0(\GC)\cong \grt_1$, as shown in \cite{Will}, whence $\Exp(H^0(\GC))\cong \GRT_1$. 
Furthermore, the torsor of Drinfel{\cprime}d associators may be identified with the torsor formed by homotopy classes of stable formality morphisms. More concretely, the ``correct'' version of this identification is the unique one identifying the homotopy class of M.~Kontsevich's formality isomorphism with $\Phi_\mathrm{AT}$, and such that it is equivariant with respect to the action of $H^0(\GC)\cong \grt_1$.

Our second main result is the following.

\begin{Thm}\label{thm:stablefamily}
There is a family of stable formality morphisms $\mU^t$ over $\mathbb R$ and cocycles $\{x_3,x_5,\dots\}$ of degree $0$ in $\GC$, such that:
\begin{itemize}
\item[$i)$] The homotopy class of stable formality morphisms of  $\mU^t$  corresponds to the Drinfel{\cprime}d associator $\Phi^t$ from Theorem \ref{thm:etingof} for all $t$.

\item[$ii)$] The graph cohomology class represented by $x_{2j+1}$ corresponds to  $\tau_{2j+1}$ in $\grt_1$ (see Theorem~\ref{thm:etingof}) under the identification $\grt_1\cong H^0(\GC)$ for all $j=1,2,\dots$.

\item[$iii)$] $\mU^{\frac 1 2}$, resp.\ $\mU^0$, is the stable formality morphism constructed by means of the standard angular propagator, see~\cite{K}, resp.\ the logarithmic propagator, see \cite{ARTW}. 
In particular, this means that the homotopy class of the latter stable formality morphism corresponds to the Knizhnik--Zamolodchikov associator $\Phi_\mathrm{KZ}$.

\item[$iv)$] $x_{2j+1}$ is divergence-free for all $j\geq 1$, {\em i.e.} it commutes with the graph $\Gamma_{\lcirclearrowdown}$ with one vertex and one edge, see~\cite{Will}*{Section 2}.
\end{itemize}

\end{Thm}

Again, $\mU^t$ and $x_{2j+1}$ are given by explicit integral expressions. 
In particular, the formul\ae\ for the $x_{2j+1}$ (see~\eqref{eq-ctgammasimpl}) are the first explicit expressions (at least, to our knowledge) for graph cocycles representing the graph cohomology classes corresponding to the (conjectural) generators of $\grt_1$.

\begin{Rem}
One may also check that the stable formality morphisms $\mU^t$ can be globalized to formality morphisms for arbitrary smooth manifolds $M$. In other words they satisfy M. Kontsevich's conditions P1)-P5) \cite{K} sufficient for globalization. 
\end{Rem}

\subsection*{Structure of the paper}
The paper is roughly divided into two largely independent parts. 
The first part (Sections~\ref{s-5} and~\ref{s-6}) is devoted to the technical aspects of the proof of Theorem~\ref{thm:stablefamily}, and the second part (Sections~\ref{s-7} and~\ref{s-8}) is devoted to those of to the proof of Theorem~\ref{thm:etingof}.
Each part starts with a recollection of some preliminaries (Section~\ref{s-2} and Subsections~\ref{ss-7-1},~\ref{ss-7-2}). 
Finally, Theorem~\ref{thm:stablefamily} is proven in Subsection~\ref{ss-9-2}, while Subsection~\ref{ss-9-1} contains the proof of Theorem~\ref{thm:etingof}.

\subsection*{Acknowledgements}
We are grateful to P.~Etingof for sharing his insights about the relationship between the KZ-, anti-KZ- and AT-associator with the second author while he was in Harvard.
The first author thanks A.~Alekseev for having first indicated that the strong P.~Etingof conjecture might not be true, and for many enlightening discussions on singular propagators and related issues.
He also thanks J.~L\"offler for many discussions on the same subject at the Max Planck Institut f\"ur Mathematik.
We also thank heartily G.~Felder for his interest in the subject, for many useful discussions and questions on the subject of the paper, which surely helped to give it a better shape. 

\subsection*{Notation}
Unless otherwise stated we work over the ground field $\bbC$, i.~e., all vector spaces or differential graded vector spaces we consider are $\bbC$-vector spaces.
For a graded vector space $V$ we denote by $V[k]$ the $k$-fold desuspension of $V$. For example, if $V$ is concentrated in degree $0$, then $V[k]$ is concentrated in degree $-k$. 
We generally use cohomological conventions, so the differentials in complexes are of degree $+1$. 
The phrase \emph{differential graded} will be abbreviated to \emph{dg}.

We will denote the set of numbers from $1$ to $n$ by $[n]:=\{1,2,\dots, n\}$. The symmetric groups will be denoted by $\mathfrak S_n$.

We often use the language of operads and colored operads. A good introduction can be found in the textbook \cite{LV} by J.-L. Loday and B. Vallette. For an operad $\op P$ we will denote by $\Lambda\op P$ the desuspension of $\op P$, defined such that a $\Lambda\op P$-algebra structure on $V$ is the same as a $\op P$,-algebra structure on $V[1]$.
For a quadratic operad $\op P$ we denote by $\op P^\vee$ its Koszul dual coopered. For a coaugmented cooperad $\op C$ we denote by $\Omega(\op C)$ its operadic cobar construction.
The Lie operad is denoted by $\Lie$, and its minimal cofibrant resolution by $\Lie_\infty=\Omega(\Lie^\vee)$.

For handling signs it will be convenient to introduce the following notation. Let $n_1,\dots, n_k$ be pairwise distinct natural numbers. Then we define the expression $\psgn{n_1,\dots, n_k}$ to be the sign of the permutation 
\[
1,2,3,4,\dots \mapsto n_1,\dots, n_k, 1,2, \dots.
\]
In other words, one has the following recursion
\begin{align}\label{equ:odef}
\psgn{n_1} &:= (-1)^{n_1-1}
&
\psgn{n_1,\dots,n_k} &= \psgn{n_1,\dots,n_{k-1}} (-1)^{n_k-1-|\{j \mid n_j<n_k\}|} .
\end{align}
Note that the expression $\psgn{n_1,\dots, n_k}$ is antisymmetric in its arguments.

\section{Recollections about M. Kontsevich's graph complex and stable formality morphisms}\label{s-2}
\subsection{Kontsevich's graphs}\label{ss-2-1}
The notion of stable formality morphism puts in a more conceptual framework M. Kontsevich's seminal construction of the formality $L_\infty$-quasi-isomorphism $\mathcal U$ from~\cite{K}*{Subsection 6.3}. 
Since graphs will play a fundamental {\em r\^ole} in the present construction, let us briefly recall the relevant graph-theoretical objects and discuss their main features. 
In particular, we will make use of the graph operads $\dGra$ and $\KGra$, introduced (implicitly) by Kontsevich in~\cite{K2}*{Subsubsection 3.3.3} and~\cite{K}*{Subsection 6.1} and described in more detail in~\cite{Will}*{Section 2} or~\cite{Dol}*{Subsections 3.2-3.4}. 

Consider the set of directed graphs $\dgra_{n,k}$ with vertex set $n$ and edge set $k$. We can build an operad $\dGra$ such that the space of $n$-ary operations is
\[
 \dGra(n) = \bigoplus_{k\geq 0} \left( \mathrm{span}( \dgra_{n,k}) \otimes (\bbC[1])^{\otimes k} \right)_{S_k}.
\]
Here the symmetric group acts on graphs by permuting the edge labels and on the $\bbC[1]$ factors by permutations, with appropriate Koszul signs. 
%
%
Observe that, if $\Gamma$ in $\dGra(n)$ contains a multiple edge, {\em i.~e.} if there are at least two edges with the same direction between two distinct vertices of $\Gamma$, then $\Gamma$ is trivial.

Let us specify the operadic structure on $\dGra$. The composition $\Gamma_1\circ_i \Gamma_2$ of graph $\Gamma_1$, $\Gamma_2$ is defined by replacing the $i$-th vertex of $\Gamma_1$ by $\Gamma_2$, and summing over all ways of reconnecting the edges incident at vertex $i$. The ordering of the edges in $\Gamma_1\circ_i \Gamma_2$ is such that the edges of $\Gamma_1$ stand before those of $\Gamma_2$. 
Finally, there is an obvious right action of the symmetric group $\mathfrak S_n$ on $\dGra(n)$ by permuting the labels of the vertices.

Similarly, consider the set of \emph{admissible} graphs $\kgra_{n,m,k}$ (Kontsevich's graphs) with vertex set $[n]\sqcup [m]$ and edge set $k$, where admissible means that no edge starts at one of the vertices in $[m]$, cf. \cite{K}*{Subsection 6.1}. Following loc. cit. we will call the vertics in the set $[n]$ \emph{type I vertices} and vertices in the set $[m]$ \emph{type II vertices}.

Out of $\kgra_{n,m,k}$ we may build a two-colored operad $\KGra$ of graded vector spaces as follows.
We denote by $\mathfrak o$, $\mathfrak c$ the two colors of $\KGra$ ($\mathfrak o$, $\mathfrak c$ stays for open, closed respectively). We denote the spaces of operations with $n$ inputs of color $\mathfrak c$, $m$ inputs of color $\mathfrak o$ and with output in color $\mathfrak o$ or $\mathfrak c$ by $\KGra(n,m)^\mathfrak o$ and $\KGra(n,m)^\mathfrak c$, respectively.
We define 
\begin{align*}
\KGra(n,m)^\mathfrak c&=\begin{cases}
\dGra(n),& m=0\\
\{0\},& m\geq 1
\end{cases},
&
\KGra(n,m)^\mathfrak o
&=
\bigoplus_{k\geq 0} \left( \mathrm{span}( \kgra_{n,m,k}) \otimes (\bbC[1])^{\otimes k} \right)_{S_k}.
\end{align*}
In the terminology adopted by Kontsevich in~\cite{K}*{Subsection 6.1}, elements of $\KGra(n,m)^\mathfrak o$ are linear combinations of admissible graphs of type $(n,m)$.

To specify the operadic structure on $\KGra$, we resort again to partial insertions of directed, labeled graphs.
Partial insertions on $\KGra^\mathfrak c$ are exactly as for $\dGra$.
The guiding principle (plug-in and reconnect the vertices) remains the same as for $\dGra$, but we require no partial insertions from $\KGra^\mathfrak c\otimes\KGra^\mathfrak o$ either to $\KGra^\mathfrak c$ or $\KGra^\mathfrak o$.
Otherwise, there are partial insertions from $\KGra^\mathfrak o\otimes\KGra^\mathfrak c$ to $\KGra^\mathfrak o$ at either a vertex of color $\mathfrak c$ or $\mathfrak o$, and partial insertions from $\KGra^\mathfrak o\otimes\KGra^\mathfrak o$ to $\KGra^\mathfrak o$ at a vertex of color $\mathfrak o$.
Again, there is a natural re-labeling of vertices and an induced total order on the edges of the composite graphs.

There are three special elements of $\KGra$ which deserve separate attention.
First, there are two elements of $\KGra(2,0)^\mathfrak c$ of degree $-1$, and it will be convenient to consider their sum, which we denote by $\SNbr$. It is clear that $\SNbr$ is invariant with respect to the action of $\mathfrak S_2$ on the labels of its vertices.
Second, observe that an element of $\KGra(0,m)^\mathfrak o$, for $m\geq 1$, of strictly negative degree is necessarily trivial (because there would exist an edge departing from a vertex of color $\mathfrak o$). The unique element of $\KGra(0,2)^\mathfrak o$ of degree $0$ is denoted by $\grcup$.
Finally, we consider elements of $\KGra(1,m)^\mathfrak o$, for $m\geq 1$, of degree $-m$: the conditions on $\KGra(1,m)^\mathfrak o$ imply that there is exactly one such element of $\KGra(1,m)^\mathfrak o$, which we denote by $\grHKR^m$.
The special graphs $\SNbr$, $\grcup$ and $\grHKR^m$ are depicted in Figure~\ref{fig-HKR}.
\begin{figure}
\centering
\begin{tikzpicture}[>=latex]
\tikzstyle{k-int}=[draw,fill=gray!40,circle,inner sep=0pt,minimum size=2mm]
\tikzstyle{n-int}=[draw,fill=black,circle,inner sep=0pt,minimum size=2mm]
\tikzstyle{ext}=[draw,fill=white,circle,inner sep=0pt,minimum size=2mm]
\tikzstyle{spec}=[draw,rectangle,inner sep=0pt,minimum size=2mm]

\begin{scope}[scale=0.75]
\node[n-int,label={$\scriptstyle 1$}] (v1) at (-1.3,0) {};
\node[n-int,label={$\scriptstyle 2$}] (v2) at (-.4,0) {};
\node[n-int,label={$\scriptstyle 1$}] (v3) at (.4,0) {};
\node[n-int,label={$\scriptstyle 2$}] (v4) at (1.3,0) {};
\draw[thick,->] (v1) to (v2);
\draw[thick,->] (v4) to (v3);
\node at (0,-.75) {$\SNbr$};
\node at (0,0) {$+$};
\end{scope}

\begin{scope}[scale=0.75,shift={(3,0)}]
\node[ext,label={$\scriptstyle 1$}] (1) at (-.5,0) {};
\node[ext,label={$\scriptstyle 2$}] (2) at (.5,0) {};
\node at (0,-.75) {$\grcup$};
\end{scope}

\begin{scope}[scale=0.75,shift={(7,0)}]
\node[n-int] (v) at (0,1.5) {};
\node[ext,label={$\scriptstyle 1$}] (1) at (-1.5,0) {};
\node[ext,label=0:{$\scriptstyle 2$}] (2) at (-0.5,0) {};
\node[ext,label={$\scriptstyle m$}] (m) at (1.5,0) {};
\draw[thick,->] (v) to (1);
\draw[thick,->] (v) to (2);
\draw[thick,->] (v) to (m);
\node at (0.25,0.5) {$\dots$};
\node at (0,-0.75) {$\grHKR^m$};
\end{scope}

\end{tikzpicture}
\caption{\label{fig-HKR} The three special graphs in $\KGra$.}
\end{figure}
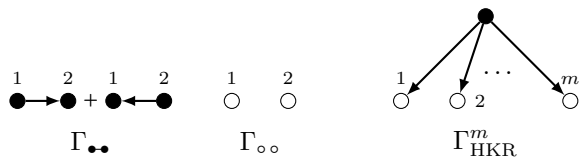

It is well-known that the pair $(T_\mathrm{poly}(A),A)$, $X=\mathbb K^d$ and $A=\mathbb K[X]$, forms a representation of $\KGra$. Here $T_\mathrm{poly}(A)=\bigwedge_A  \mathrm{Der}(A)$ is the set of multiderivations, or multivector fields. 
Let us first choose a set of global coordinates $\{x_i\}$ of $X$, so that $A=\mathbb K[x_1,\dots,x_d]$. We may then identify $T_\mathrm{poly}(A)$ with the graded $A$-module $(A[\theta_1,\dots,\theta_d])$, where $\{\theta_i\}$ denotes a set of graded variables of degree $1$, dual to $\{x_i\}$, which commute with $A$ and anticommute with each other. One should think of $\theta_i$ as $\frac \partial {\partial x_i}$ with shifted degree.
Observe that there are a natural injection $A\hookrightarrow T_\mathrm{poly}(A)$ and projection $T_\mathrm{poly}(A)\twoheadrightarrow A$.

We further consider the well-defined linear endomorphism $\tau$ of $T_\mathrm{poly}(X)^{\otimes 2}$ of degree $-1$ specified {\em via} 
\[
\tau=\partial_{\theta_i}\otimes\partial_{x_i},
\]
where summation over repeated indices is implicit.
With a directed edge $e=(i,j)$ of a graph $\Gamma$ in $\KGra(n,m)^\mathfrak c$ or $\KGra(n,m)^\mathfrak o$ we may associate a graded endomorphism $\tau_e$ of $T_\mathrm{poly}(A)^{\otimes (n+m)}$ by letting $\tau$ act on the $i$-th and $j$-th entry only: the order on $E(\Gamma)$ makes the assignment $\Gamma\mapsto \tau_\Gamma=\prod_{e\in E(\Gamma)}\tau_e$ a well-defined map from $\KGra(n,m)^\mathfrak c$ or $\KGra(n,m)^\mathfrak o$ to the endomorphisms of $T_\mathrm{poly}(A)^{\otimes (n+m)}$ of degree $\mathrm{deg}(\Gamma)$.  

Finally, with a graph $\Gamma$ in $\KGra(n,0)^\mathfrak c$, resp.\ $\KGra(n,m)^\mathfrak o$, we associate a multidifferential operator on $T_\mathrm{poly}(A)$ of arity $n$, resp.\ a multidifferential operator on $T_\mathrm{poly}(A)$ of arity $n$ with values in the multidifferential operators on $A$ of arity $m$, {\em via} the composite morphisms
\[
\begin{aligned}
&\xymatrix{& T_\mathrm{poly}(A)^{\otimes n}\ar^-{\tau_\Gamma}[r] & T_\mathrm{poly}(A)^{\otimes n}\ar^-{\mu_n}[r] & T_\mathrm{poly}(A),}\ \text{resp.}\\
&\xymatrix{T_\mathrm{poly}(A)^{\otimes n}\otimes A^{\otimes m}\ar@{^{(}->}[r] & T_\mathrm{poly}(A)^{\otimes (n+m)}\ar^-{\tau_\Gamma}[r] &  T_\mathrm{poly}(A)^{\otimes (n+m)} \ar^-{\mu_{n+m}}[r] & T_\mathrm{poly}(A)\ar@{->>}[r] & A},
\end{aligned}
\]
where $\mu_n$ denotes the $n$-th iterated multiplication morphism on a (graded) associative algebra.
Observe that we may also consider a graded, finite-dimensional linear manifold $X$, and the associated pair $(T_\mathrm{poly}(A),A)$, for $A=\mathbb K[x_1,\dots,x_d]$ with $\{x_i\}$ graded, is again a representation of $\KGra$.
By abuse of notation, we also denote by $\Gamma$ the above operators associated with the graph $\Gamma$.

\subsection{Stable formality morphism}\label{ss-2-2}
Let us consider the $2$-colored operad $\OC$ of dg vector spaces. Briefly, an algebra over $\OC$ is a triple $(\mathfrak g,A,\mathcal F)$, such that $(i)$ $\mathfrak g[1]$ is an $L_\infty$-algebra, $(ii)$ $A$ is an $A_\infty$-algebra and $(iii)$ $\mathcal F$ is an $L_\infty$-morphism from $\mathfrak g[1]$ to the Hochschild cochain complex $C^\bullet(A)$ of $A$ with values in itself and shifted degree, regarded as a dg Lie algebra with the $A_\infty$-Hochschild differential and the Gerstenhaber bracket.
It is well-known that $\OC$ is generated by three types of coroll\ae\ $\mathsf t^\mathfrak c_n$, $n\geq 2$, $\mathsf t_m^\mathfrak o$, $m\geq 2$ and $\mathsf t_{n,m}^\mathfrak o$, $n\geq 1$, $m\geq 0$, of degree $3-2n$, $2-m$ and $2-2n-m$ parametrizing the Taylor components of the pre-$L_\infty$- and pre-$A_\infty$-algebra structures and of a morphism of pre-$L_\infty$-algebras respectively; furthermore, the dg operad structure on $\OC$ is uniquely specified by the action of the differential $d_\OC$ on the generating coroll\ae, and encodes the fact that $\mathsf t_n^\mathfrak c$ define indeed an $L_\infty$-structure, $\mathsf t_m^\mathfrak o$ an $A_\infty$-structure and $\mathsf t_{n,m}^\mathfrak o$ a morphism of $L_\infty$-algebras.
We refer to~\cite{Dol}*{Subsection 4.1} for an explicit graphical representation of $d_\OC$, which will be useful later on.

Therefore, a morphism $\mathcal F$ of $2$-colored dg operads from $\OC$ to $\KGra$ is uniquely determined by the images of the generating coroll\ae\ of $\OC$, which are in turn linear combinations of graphs.
By further composing $\mathcal F$ with the morphism of $2$-colored dg operads $\KGra\to \mathrm{End}(T_\mathrm{poly}(A),A)$, for $A$ as above, the final result consists of $(i)$ an $L_\infty$-structure on $T_\mathrm{poly}(A)[1]$, $(ii)$ an $A_\infty$-structure on $A$ and $(iii)$ an $L_\infty$-morphism from $T_\mathrm{poly}(A)[1]$ to the Hochschild cochain complex of $A$ with shifted degree, the three of them parametrized by elements of $\KGra$.

For the pair $(T_\mathrm{poly}(A),A)$, there is a natural $L_\infty$-algebra structure on $T_\mathrm{poly}(A)[1]$ and a natural $A_\infty$-algebra structure on $A$, specified by the Schouten--Nijenhuis bracket and the (graded) commutative product: they admit the graphical counterparts $\SNbr$ and $\grcup$ in $\KGra$. 
Furthermore, the Hochschild--Kostant--Rosenberg quasi-isomorphism from $T_\mathrm{poly}(A)[1]$ to $C^\bullet(A)[1]$ can be considered as the first Taylor component of some $L_\infty$-morphism: it also admits the graphical interpretation {\em via} $\grHKR$ in $\KGra$.
\begin{Def}\label{d-stable}
A stable formality morphism $\mathcal F$ is a morphism of $2$-colored dg operads 
\[
\mathcal F:\OC\to \KGra, 
\]
whose induced representation on pairs $(T_\mathrm{poly}(A),A)$ coincides with the Schouten--Nijenhuis dg Lie algebra structure on $T_\mathrm{poly}(A)$, with the standard, (graded) commutative $A_\infty$-algebra structure on $A$, and such that the first Taylor component of the corresponding $L_\infty$-quasi-isomorphism coincides with the Hochschild--Kostant--Rosenberg quasi-isomorphism.   
\end{Def}

Stable formality morphisms may be identified with series of graphs
\[
\mathcal F(\mathsf t_{n,m}^\mathfrak o)=\sum_{\Gamma}\alpha_\Gamma\ \Gamma,\ \alpha_\Gamma\in\mathbb K,\ n\geq 1,\ m\geq 0,
\]
where the sum runs over a set of graphs $\Gamma\in \kgra_{n,m,k}$ forming a basis of the degree $2-2n-m$-subspace of $\KGra(n,m)$,
with the ``boundary conditions''
\begin{align*}
\mathcal F(\mathsf t_2^\mathfrak o)&=\grcup,
&
\mathcal F(\mathsf t_{1,m}^\mathfrak o)&=\grHKR^m,\ m\geq 0.
\end{align*}
Observe that the degree of $\mathsf t_m^\mathfrak o$ is strictly negative for $m\geq 3$, thus $\mathcal F(\mathsf t_m^\mathfrak o)$ is an element of $\KGra(0,m)^\mathfrak o$ of strictly negative degree and is automatically trivial by degree arguments.

In the language of $2$-colored operads, $\OC$ is the Cobar construction of a $2$-colored cooperad 
$\mathfrak{oc}^\vee$: thus, by the arguments of~\cite{Dol}*{Subsection 2.5}, a stable formality morphism $F$ as above is a Maurer--Cartan element in the convolution dg Lie algebra 
$$
\mathrm{Conv}(\mathfrak{oc}^\vee,\KGra) \cong \prod_{n\geq 1} \left(\KGra(n,0)^{\mathfrak c}\right)^{\mathfrak S_n} \oplus \prod_{n,m\geq 0} \left(\KGra(n,m)^{\mathfrak o}\right)^{\mathfrak S_n}.
$$
Then, the Maurer--Cartan equation for $F$ as above translates into the infinite family of quadratic relations between weights, which has been first explored in~\cite{K}*{Subsection 6.4}.

{\bf Notation:} Suppose we are given a collection of maps $\kgra_{n,m,k}\to \bbC$ (respectively $\dgra_{n,k}\to \bbC$), say $\Gamma\mapsto \alpha_\Gamma$, that are invariant under the $S_n$ and $S_k$ actions, with appropriate signs. The we can build an element 
\be{equ:sumgammadef}
\sum_\Gamma  \alpha_\Gamma \, \Gamma \in  \mathrm{Conv}(\mathfrak{oc}^\vee,\KGra) 
\ee
where the sum runs over all $n$ (and $m$) and over a set of graphs in $\cup_k\kgra_{n,m,k}$ (respectively in $\cup_k\dgra_{n,k}$) forming a basis of $\KGra(n,m)^\mathfrak o$ (respectively of $\KGra(n,0)^\mathfrak c$).

\subsection{The graph complex \texorpdfstring{$\GC$}{GC}}\label{ss-2-3}
Recall from the previous section that there is a map of operads $\LaLie_\infty\to \dGra$, factoring through the degree shifted Lie operad $\LaLie$. 

The full directed graph complex $\fdGC$ is just the deformation complex of that operad map, {\em i.e.}
\[
\fdGC := \Defo(\LaLie_\infty\to \dGra).
\]   
Concretely, the object on the right-hand side is the operadic convolution complex $\Hom_{\mathbb{S}}(\LaLie^\vee,\dGra)$, twisted by the Maurer--Cartan element corresponding to the operad map $\LaLie_\infty\to \dGra$.
Consequently,
\[
\fdGC = \prod_{n\geq 1} \dGra(n)^{\mathbb{S}_n}[2-2n]
\]
as vector spaces.
Note that it follows from the definition as a deformation complex that $\fdGC$ carries a natural dg Lie algebra structure.
Furthermore, by the action of $\dGra$ on $T_\mathrm{poly}(A)$ (for $A$ as above) we obtain a map of dg Lie algebras from $\fdGC$ to the Chevalley complex of $T_\mathrm{poly}(A)$. 
In particular, the sub-dg Lie algebra of closed degree zero elements of $\fdGC$ acts on $T_\mathrm{poly}(A)$ by $L_\infty$ derivations. 
Hence it acts also on all formality isomorphisms 
\[
T_\mathrm{poly}(A)\to \Dpoly(A)
\]
and it is not hard to check that this action factors through an action on all stable formality morphisms.

It was shown by V.~Dolgushev~\cite{Dol} that the induced action of $H^0(\fdGC)$ on the set of homotopy classes of stable formality morphisms is free and transitive.
This is almost the statement of Theorem~\ref{thm:dlgstable}, except that one replaces the complex $\fdGC$ by a much smaller subcomplex $\GC\subset\fdGC$ such that $H^0(\GC)=H^0(\fdGC)$. 
Concretely, $\dGra$ contains a sub-operad $\Gra$ of un-directed graphs, where the embedding $\Gra\to \dGra$ assigns to an un-directed graph the sum over all graphs obtained by assigning some orientations to the edges. So, pictorially,
\[
\begin{tikzpicture}[baseline=-0.65ex]
\draw (0,0)--(.5,0);
\end{tikzpicture}
\mapsto
\begin{tikzpicture}[baseline=-0.65ex,>=latex]
\draw[->] (0,0)--(.5,0);
\end{tikzpicture}
+
\begin{tikzpicture}[baseline=-0.65ex,>=latex]
\draw[<-] (0,0)--(.5,0);
\end{tikzpicture}.
\]

One may define an un-directed version of the graph complex 
\[
\fGC := \Defo(\LaLie_\infty\to \Gra)\cong \prod_{n\geq 1} \Gra(n)^{\mathbb{S}_n}[2-2n].
\]
Finally the subcomplex $\GC \subset \fGC$ is the subcomplex spanned by connected graphs, all vertices of which have valence at least $3$.
It was shown partly by~M.~Kontsevich and partly in~\cite{Will} that $H^\bullet(\fGC)=H^\bullet(\dfGC)$ and that this cohomology may be expressed through $H^\bullet(\GC)$. 
In particular, it follows that $H^0(\GC)=H^0(\fdGC)$.

\section{Compactified Configuration spaces {\em \`a la} Kontsevich}\label{s-3}

\subsection{Recollection: Stokes' Theorem for singular differential forms}\label{sec:regstokes}
One of the central tools in the proof of the main results of the present paper is a version of Stokes' Theorem for smooth differential forms with singularities on the boundary introduced in \cite{ARTW}. This section is devoted to briefly recalling the statement of this Theorem.

Following \cite{ARTW}, we consider an $n$-dimensional compact manifold with corners $K$. We assume that $K$ is covered by a system of charts $\{U_i\}_{i\in I}$ where $I$ is a partially ordered set.
We assume that the following two conditions hold:
\begin{itemize}
 \item $U_i\cap U_j=\emptyset$ unless $i\geq j$ or $i\leq j$.
 \item Each $U_i$ carries a free action of a torus $T_i$ preserving the boundary. For $i>j$ one has a natural injective group homomorphism $T_i \hookrightarrow T_j$ such that the inclusions $ U_i \cap U_j \hookrightarrow U_i, U_i \cap U_j \hookrightarrow U_j$ are $T_i$-equivariant. 
 \item There is a partition of unity $\{\rho_i\}_{i\in I}$ subordinate to the chosen covering, such that each $\rho_i$ is $T_i$-invariant.
\end{itemize}

Let the torus action of $T_i$ be generated by vector fields $v_{i,a}$, $a = 1, \dots , {\rm dim}(T_i)$, and define the multivector field
\[
 \xi_i = \bigwedge_{a=1}^{\mathrm{dim} T_i} v_{i,a} .
\]
A differential form $\omega$ is called \emph{$\xi$-basic}, for $\xi$ a multi-vector field, if $\iota_\xi \omega =0$ and $\iota_\xi d\omega =0$.
Let $\omega$ now be a smooth differential form on the interior $\K^\circ\subset K$ of top-1-degree. Then one defines $\omega$ as \emph{regularizable} if for each $i\in I$ there is a $\xi_i$-basic differential form $\alpha_i$ (the counterterm) such that $\omega-\alpha_i$ extends to the boundary $\p K\cap U_i$.
One defines the regularization $\Reg(\omega)$ of $\omega$ to be the top degree differential form on $\p K$ such that
\[
\Reg(\omega)\mid_{\p K\cap U_j}=(\omega-\alpha_j)\mid_{\p K\cap U_j}.
\]
It is shown in \cite{ARTW}*{Proposition 1} that the regularization is well-defined.
The regularized Stokes' Theorem can then be formulated as follows.
\begin{Thm}[Regularized Stokes' Theorem, \cite{ARTW}]\label{thm:regstokes}\label{t-stokes-reg-log}
Let $\omega$ be a regularizable top-1 degree form on $K$. Then, the differential form $d\omega$ is regular on $K$ and
\begin{equation}\label{equ:stokes}
 \int_K d\omega = \int_{\p K} \Reg(\omega).
\end{equation}
\end{Thm}


\subsection{Configuration spaces}\label{ss-3-2}
Let $\mathbb H^+$ denote the complex upper half-plane, $[n]$ the set $\{1,\dots,n\}$, let $A$ be a finite set and $B$ a finite set endowed with a total order.
The configuration space
\[
\mathrm{Conf}_A = \left\{p\in\mathbb C^A\,|\,p(a)\neq p(a')\textrm{ if }a\neq a'\right\}/G_3,
\]
is the space of configurations of $|A|$ points in $\bbC$. The group $G_3=\mathbb R_+\ltimes \mathbb C$ acts (diagonally) on $\mathbb C^A$ by rescalings and complex translations.
We denote the quotient by
\[
C_A=\mathrm{Conf}_A/G_3.
\]
Since $G_3$ is a real Lie group of dimension $3$ acting freely on the smooth, oriented manifold $\mathrm{Conf}_A$, whenever $2|A|-3\geq 0$, $C_A$ is a smooth oriented manifold of dimension $2|A|-3$.

For $A$ and $B$ as above, we similarly define 
\[
\mathrm{Conf}_{A,B}^+=\left\{(p,q)\in (\mathbb H^+)^A\times \mathbb R^B\,|\,p(a)\neq p(a')\textrm{ if }a\neq a'\,,\,
q(b)< q(b')\textrm{ if }b<b'\right\}
\]
to be the space of configurations of $|A|$ points in the upper halfplane and $|B|$ points on the real axis, respecting the order on $B$. The group $G_2=\mathbb R_+\ltimes \mathbb R$ acts diagonally on $(\mathbb H^+)^A\times \mathbb R^B$ {\em via} rescalings and real translations.
We denote the quotient space by
\[
C_{A,B}^+=\mathrm{Conf}_{A,B}^+/G_2.
\]
The action of the $2$-dimensional Lie group $G_2$ is free whenever $2|A|+|B|-2\geq 0$, in which case $C_{A,B}^+$ is a smooth real oriented manifold of dimension $2|A|+|B|-2$. 
For $A=[n]$ and $B=[m]$ with its natural total order, we use the simpler notation $C_n$ and $C_{n,m}^+$.

It is easy to verify by direct computations that $\mathrm{Conf}_A$ and $\mathrm{Conf}_{A,B}^+$ are trivial principal bundles over $C_A$ and $C_{A,B}^+$.
Let us illustrate this fact by an example which will be quite useful in subsequent computations. A global section of $C_n$ is defined by
\[
C_n\ni [(z_1,\dots,z_n)]\mapsto \left(0,\frac{z_2-z_1}{|z_2-z_1|},\frac{z_3-z_1}{|z_2-z_1|},\dots,\frac{z_n-z_1}{|z_2-z_1|}\right)\in \mathrm{Conf}_n.
\]
We hence get the identification
\[
C_n\cong S^1\times\mathrm{Conf}_{n-2}(\mathbb C\smallsetminus \{0,1\}).
\]

Observe that the standard $S^1$-action on $\mathbb C$ yields a diagonal action of $S^1$ on $\mathrm{Conf}_n$, which commutes with the $G_3$-action, thus it descends to an action on $C_n$, with respect to which the projection $\mathrm{Conf}_n(\mathbb C)\to C_n$ is equivariant.
The previous computations imply that $C_n$ is a trivial $S^1$-bundle over $\mathrm{Conf}_{n-2}(\mathbb C\smallsetminus\{0,1\})$.
This fact will also be used throughout the whole paper.


Given a subset $A'\subset A$ and a subset $B'\subset B$ with the induced total order, there are naturally defined projections $\mathrm{Conf}_A\to \mathrm{Conf}_{A'}$ and $\mathrm{Conf}_{A,B}^+\to \mathrm{Conf}_{A',B'}^+$, which obviously descend to projections $C_A\to C_{A'}$ and $C_{A,B}^+\to C_{A',B'}^+$.

\subsection{The compactified configuration spaces \texorpdfstring{$\overline C_A$ and $\overline C_{A,B}^+$}{CA and CAB}}\label{ss-3-3}
Kontsevich has introduced in~\cite{K}*{Section 5} compactifications {\em \`a la} Fulton--MacPherson of $C_A$ and $C_{A,B}^+$, which are endowed with the structure of manifold with corners.
We do not present here the explicit construction of the compactified configuration spaces $\overline C_A$ and $\overline C^+_{A,B}$, for which we refer instead to~\cite{BCKT}*{Appendice A}.

We concentrate mainly on the boundary stratification of $\overline C_A$ and $\overline C_{A,B}^+$ and on local coordinates near a given boundary stratum. The material presented here is a {\em r\'esum\'e} of the discussions in~\cite{ARTW}, to which we refer for more details.

The boundary strata of codimension $1\leq p\leq |A|-2$ of $\overline C_A$ are in one-to-one correspondence with nested families $\{A_1,\dots,A_p\}$ of subsets of $A$ of cardinality $2\leq |A_i|\leq |A|-1$, $i=1,\dots,p$, {\em i.~e.} either $A_i\cap A_j=\emptyset$ or $A_i\subset A_j$ or $A_i\supset A_j$. 
We set $A_0:=A$ for notational convenience.
A nested family determines a rooted tree with vertices corresponding to $\{A_0,A_1,\dots,A_p\}\cup \bigcup_{a\in A}\{a\}$ and with the descendants of vertex $A'$ in the tree being the vertices $A''$ such that $A''\subset A'$.
The star of an element $A_i$ of a nested family as above, denoted  $\mathrm{star}(A_i)$, is the set of direct children of $A_i$ in the tree.


The stratum $\partial_{A_1,\dots,A_p}\overline C_A$ associated with the nested family $\{A_1,\dots,A_p\}$ is isomorphic to the product of configuration spaces
\begin{equation}\label{eq-strat-1}
\partial_{A_1,\dots,A_p}\overline C_A\cong \prod_{j=0}^p \overline C_{\mathrm{star}(A_i)}.
\end{equation}

Similarly, we consider families $\{A_1,\dots, A_p, C_1,\dots,C_q\}$ with $1\leq p+q\leq |A|+|B|-1$ and with each $A_j\subset A$ and $C_j\subset A\sqcup B$. 
We call such a family nested if the following conditions hold.
\begin{itemize}
\item[$(i)$] $2\leq |A_j|\leq |A|$ for all $j$.
\item[$(ii)$] If $C_i=A_i'\sqcup B_i'\subset A\sqcup B$, then $B_i'$ consists of consecutive elements of $B$. Either $A_i'$ or $B_i'$ may be empty. If $A_i'$ is empty we require that $2\leq |B_i'|\leq |B|$, and if $B_i'$ is empty we require $1\leq |A_i'|\leq |A|$.
\item[$(iii)$] If $C_i\cap C_j\neq \emptyset$ then $C_i\subset C_j$ or $C_j\subset C_i$. If $A_i\cap A_j\neq \emptyset$ then $A_i\subset A_j$ or $A_j\subset A_i$. If $A_i\cap C_j\neq \emptyset$, then $A_i\subset A_j$.
\end{itemize}
We will call the $A_j$'s \emph{the type I elements} and the $C_j$'s the \emph{type II elements} of the nested family.
We set $C_0=A\sqcup B$.
Again, the nested family determines a rooted tree with vertices corresponding to $\{C_0,A_1,\dots,C_p\}\sqcup \bigcup_{c\in A\sqcup B}\{c\}$.
We call the vertices corresponding to the $A_j$ and to the one-element sets $\{a\}$ the \emph{type I vertices} and the remainder the \emph{type II vertices}.
The edges in the tree are defined such that the induced partial order agrees with that of inclusion, where we impose that by definition a type I subset does not include any type II subset, though the converse is possible.
The star of an element $D$ of a nested family as above, denoted  $\mathrm{star}(D)$, is the set of direct children of $D$ in the tree.
Moreover, for a type II vertex $C$ we define $\mathrm{star}_I(C)$ to be the set of type I children and $\mathrm{star}_{II}(C)$ the set of type II children.

To a nested family as above one may associate a boundary stratum $\partial_{A_1,\dots,A_p,C_1,\dots,C_q}\overline C_{A,B}^+$ of codimension $p+q$.  It is isomorphic to a product of configuration spaces 
\begin{equation}\label{eq-strat-2}
\partial_{A_1,\dots,A_p,C_1,\dots,C_q} \overline C_{A,B}^+ \cong \prod_{i=1}^p  \overline C_{\mathrm{star}(A_j)} \prod_{j=0}^q \overline C^+_{\mathrm{star}_I(C_j), \mathrm{star}_{II}(C_j)}.
\end{equation}


\subsection{Coordinates and local torus actions}\label{sec:3-4}
It has been shown in \cite{ARTW} that the compactified configuration spaces $\overline C(A)$ and $\overline C_{A,B}^+$ fit into the framework of the regularized Stokes Theorem from section \ref{sec:regstokes}. Concretely, one can construct the following data.
There is a covering by charts $U_i$ with a free action of a torus $T_i$.
Here $i$ runs over the set of nested families of subsets of $A$ (respectively, of $A\sqcup B$) as considered in the previous subsection.
Note that the set of such $i$ is naturally partially ordered.

For $\overline C(A)$ the torus action is defined as follows. For each element $A'$ of the nested family $i$ we define the center of mass
\[
\zeta_{A'} = \frac 1 {|A'|} \sum_{a\in A'} z_a.
\]
Then one defines an $S^1$ action by rotating the points in $A'$ around their center of mass.
The $S^1$ actions assigned to different $A'$ commute and assemble into an action of a torus $T_i$ of dimension $|i|$.
We denote the vector field generating the $S^1$ action by $v_{A'}$ and define the multivector field
\[
 \xi_i := \bigwedge_{A'\in i} v_{A'}.
\]

For $\overline C_{A,B}^+$ the construction is similar, except that one assigns $S^1$ actions only to the type I subsets in the nested family.

On each $U_i$ one may furthermore define local coordinates. We consider the case of $\overline C_A$.
One assigns to each $A_j\in i$ a parameter $r_{A_j}\geq 0$, and a configuration $\{z_{A'}^{(j)} \mid A'\in \mathrm{star}(A_j) \}$, normalized such that 
\begin{align*}
\sum_{A'} |A'| z_{A'}^{(j)} &= 0
&
\sum_{A'} |A'| |z_{A'}^{(j)}|^2 &= 1.
\end{align*} 

Suppose that $a\in A$ and that we have a maximal chain of subsets $a\in A_k \subset A_{k-1} \subset \cdots A_1 \subset A$, where all $A_j$ are members of the nested family $i$. 
Then locally one has a parameterization
\[
z_a = \zeta_{B_1} + r_{B_1}( z_{B_2}^{(1)} + r_{B_2}( \cdots ( z_{B_k}^{k-1} + r_{B_k} z_a^{(k)}   ) \cdots  ).
\]

For more details we refer the reader to \cite{ARTW}*{section 3}.

\begin{Rem}
 Following loc. cit. we will denote nested families by the letter $i$, slightly sub-optimally for the risk of confusion with a natural number. Furthermore, we will often think of $i$ as the tree determined by the nested family, and identify the the elements of $i$ with vertices of the tree. 
\end{Rem}

\section{Propagators and weight forms}

\subsection{A family of singular propagators}\label{s-4}
Let us first discuss in detail the compactified configuration space $\overline C_{2,0}^+$, the ``Eye''.
It has three boundary strata of codimension $1$, one of them is a copy of $\overline C_2=S^1$, the other two are isomorphic to closed intervals, glued at their boundary points, which in turn are the two boundary strata of codimension $2$.
See Figure~\ref{fig-eye} for a pictorial representation of $\overline C_{2,0}^+$ and its boundary stratification.
\begin{figure}
\centering
\begin{tikzpicture}[>=latex]
[scale=0.8]
\coordinate (v1) at (0,0);
\draw[fill=gray!20!white] (v1) arc [start angle=130, end angle=50,radius=5];
\draw[fill=gray!20!white] (v1) arc [start angle=230, end angle=310,radius=5];
\draw[fill=white] (3.17,0) circle [radius=0.58];
\node (p) at (75:2) {$\overline C_2$};
\node (1) at (-1,0) {$\overline C_{0,2}^+$};
\node (2) at (7.4,0) {$\overline C_{0,2}^+$};
\node (3) at (3.17,2) {$\overline C_{1,1}^+$};
\node (4) at (3.17,-2) {$\overline C_{1,1}^+$};
\node (q) at (2.61,0) {};
\draw[->] (p) to (q);
\node (5) at (0,0) {};
\draw[->] (1) to (5);
\node (6) at (6.4,0) {};
\draw[->] (2) to (6);
\node (7) at (3.17,1.1) {};
\draw[->] (3) to (7);
\node (8) at (3.17,-1.1) {};
\draw[->] (4) to (8);
\end{tikzpicture}
\caption{\label{fig-eye} The Eye $\overline C_{2,0}^+$ and its boundary stratification.}
\end{figure}
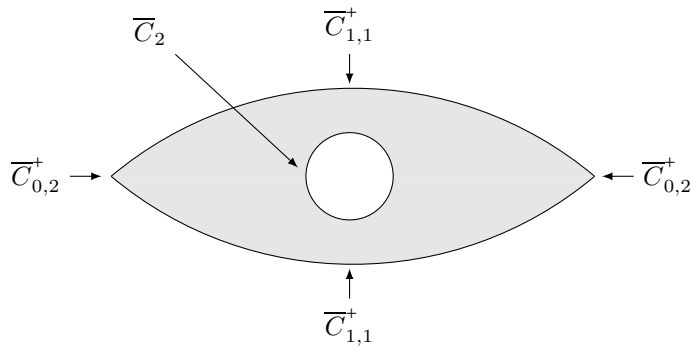

We now define on $\mathrm{Conf}_{2,0}^+$ the polynomial family over $\mathbb R\ni t$ of multi-valued functions
\[
\mathrm{Conf}_{2,0}^+\ni (z_1,z_2)\mapsto \phi^t(z_1,z_2)=\frac{1-t}{2\pi i}\log\!\left(\frac{z_1-z_2}{\overline z_1-z_2}\right)-\frac{t}{2\pi i}\log\!\left(\frac{\overline z_1-\overline z_2}{z_1-\overline z_2}\right)\in\mathbb C,
\]
for some choice of a complex logarithm $\log(\bullet)$.

Observe that the terms $\log(\overline z_1-z_2)$ and $\log(z_1-\overline z_2)$ are smooth and well-defined on $\mathrm{Conf}_{2,0}^+$, while $\log(z_1-z_2)$ and $\log(\overline z_1-\overline z_2)$ are multi-valued and have a logarithmic singularity at the  $z_1=z_2$ of $\mathbb H^+\times\mathbb H^+$.

It is clear that $\phi^t$ descends to a multi-valued function on $C_{2,0}^+$, and its exterior derivative yields a polynomial family $\omega^t$ over $\mathbb R$ of complex-valued, real analytic closed $1$-form on $C_{2,0}^+$.

By using local coordinates for $\overline C_{2,0}^+$, when the first, respectively the second, argument in $C_{2,0}^+$ approaches $\mathbb R$, $\omega^t$ restricts to $0$, respectively to the smooth, exact $1$-form
\[
\omega^t(z_1,z_2)=\frac{1}{\pi}d\mathrm{arg}(z_1-z_2),
\]
if we choose {\em e.g.} the principal branch of the complex logarithm, for which the argument function $\mathrm{arg}(\bullet)$ on $\mathbb H^+$ takes its values in $(0,\pi)$ and $\mathrm{arg}(i)=\pi/2$.

On the other hand, let us consider local coordinates of $\overline C_{2,0}^+$ near the stratum $C_{1,0}^+\times C_2$ such that $z_1=i$ and $z_2=i+w$, $w=\rho e^{i\varphi}$, and let us compute $\omega^t$ near the said stratum:
\[
\omega^t(i,i+w)=\frac{1-t}{2\pi i}\frac{dw}w-\frac{t}{2\pi i}\frac{d\overline w}{\overline w}-\frac{1-t}{2\pi i}\frac{dw}{w+2i}+\frac{t}{2\pi i}\frac{d\overline w}{\overline w-2i}=\frac{1-2t}{2\pi i}\frac{d\rho}\rho+\frac{d\varphi}{2\pi}+\cdots,
\]
where $\cdots$ denotes a complex-valued, real analytic form on a small punctured disk around $0$.

Thus, unless $t=1/2$, $\omega^t$ does not extend to a neighborhood of the boundary stratum $C_{1,0}^+\times C_2$ of $\overline C_{2,0}^+$: in fact, the previous computations show that $\omega^t$ defines a polynomial family over $\mathbb R$ of closed elements of $\Omega^1_1(\overline C_{2,0}^+)$, thus yielding a polynomial family of singular propagators.

The family $\omega^t$ can be re-written as 
\[
\omega^t(z_1,z_2)=\frac{1}{2\pi i}d\log\!\left(\frac{z_1-z_2}{\overline z_1-z_2}\right)-\frac{t}{\pi i}d\log\!\left(\left|\frac{z_1-z_2}{\overline z_1-z_2}\right|\right)=\frac{1}{2\pi}d \mathrm{arg}\!\left(\frac{z_1-z_2}{\overline z_1-z_2}\right)+\frac{1-2t}{2\pi i}d\log\!\left(\left|\frac{z_1-z_2}{\overline z_1-z_2}\right|\right).
\]
Observe that the first term in the middle, resp.\ rightmost expression is the logarithmic propagator, resp.\ Kontsevich's argument propagator; the other term in both expressions, on the other hand, is an exact $1$-form on $C_{2,0}^+$ associated with a function proportional to
\[
\log\!\left(\left|\frac{z_1-z_2}{\overline z_1-z_2}\right|\right).
\]
By using the above local coordinates near the boundary strata of $\overline C_{2,0}^+$, it is clear that the previous function belongs to the algebra $\mathcal O_{\log}(\overline C_{2,0}^+)$ of complex-valued, real analytic functions on $\overline C_{2,0}^+$ with logarithmic singularities along the boundary.

Finally, $\omega^0$, $\omega^{\frac{1}2}$ and $\omega^1$ coincide with the logarithmic propagator, Kontsevich's argument propagator and the anti-logarithmic propagator respectively: observe that $\omega^{\frac{1}2}$ is the only element of the family which actually belongs to $\Omega^1(\overline C_{2,0}^+)$.

\subsection{Singular weight forms}
Given a Kontsevich graph $\Gamma\in \kgra_{n,m,k}$, we will consider the differential $k$-form 
\begin{align}\label{eq-t-weight-omega}
\omega^t_\Gamma&=\prod_{e\in E(\Gamma)}\omega_e^t.
\end{align}
Here the notation is as follows. With an edge $e$ of $\Gamma$ we associate the natural projection $\pi_e$ from $C_{n,m}^+$ onto $C_{2,0}^+$ or $C_{1,1}^+$ and set $\omega^t_e=\pi_e^*(\omega^t)$.
The product of the $1$-forms is well-defined in virtue of the total order on the edge set $E(\Gamma)= [k]$ of $\Gamma$.


The derivative with respect to $t$ satisfies 
\[
\p_t \omega^t_\Gamma = d \tilde \omega^t_\Gamma
\]
with
\[
\tilde \omega^t_\Gamma := 
\sum_{e\in E(\Gamma)}(-1)^{e-1} \beta_e \left(\prod_{e'\neq e}\omega^t_{e'}\right),
\]
where
\[
\beta_e=\beta(z_{s(e)}, z_{t(e)}) := \frac{i}\pi \log\!\left(\left|\frac{z_{s(e)}-z_{t(e)}}{\overline z_{s(e)}-z_{t(e)}}\right|\right).
\]

\subsection{Regularizability}
In this section we will show that the top minus one degree forms of the type $\omega^t_\Gamma$ or $\tilde \omega^t_\Gamma$ are regularizable. The proof will closely follow the proof of the analogous statements in \cite{ARTW}. As in loc. cit. the only possible singularities appear at the type I boundary strata, i.~e., the strata corresponding to multiple type I vertices collapsing in the interior of the upper halfplane, away from the real axis.

\begin{Prop}\label{prop:splitting}
Let $\Gamma\in \kgra_{n,m,k}$ be an admissible graph, and let $U_i$ be a chart, where $i$ is a nested family. Let $B$ be a type I subset in the family $i$. Then the forms $\iota_{v_B} \omega_\Gamma^t$ and $\iota_{v_B} \omega_\Gamma^t$ are regular in $r_B$ and furthermore 
\begin{align*}
 \omega_\Gamma^t &= \frac{dr_B}{r_B}\wedge \alpha^t + (\text{terms regular in $r_B$}) \\
 \tilde \omega_\Gamma^t &= \frac{dr_B}{r_B}\wedge \tilde\alpha^t + \log(r_B) \hat\alpha^t + (\text{terms regular in $r_B$}) 
\end{align*}
where $\alpha^t$, $\tilde \alpha^t$ and $\hat \alpha^t$ are differential forms independent of $r_B$, $\iota_{v_B}\alpha^t=\iota_{v_B}\tilde \alpha^t=\iota_{v_B}\hat \alpha^t=0$ and $d\alpha^t=\iota_{v_B}d\tilde \alpha^t=d\hat \alpha^t=0$.
\end{Prop}
\begin{proof}
It is sufficient to show the statement for $\tilde \omega_\Gamma^t$. Then the statement for $\omega_\Gamma^t$ follows since
\begin{align*}
\omega_\Gamma^t &= \omega_\Gamma^{\frac 1 2} + \int_{\frac 1 2}^t \p_\tau \omega_\Gamma^\tau  d\tau 
= \omega_\Gamma^{\frac 1 2} + \int_{\frac 1 2}^t  d\tilde  \omega_\Gamma^\tau  d\tau \\
&= \omega_\Gamma^{\frac 1 2} +  \int_{\frac 1 2}^t \frac{dr_B}{r_B} \wedge \left( d\tilde  \alpha^\tau +  \hat\alpha^\tau    \right)d\tau + (\text{terms regular in $r_B$})
\end{align*}
and since the Kontsevich weight form $\omega_\Gamma^{\frac 1 2}$ is regular.

Consider $\tilde \omega_\Gamma^t$. There are three sorts of edges in $\Gamma$ contributing to the differential form:
\begin{itemize}
\item Edges $e$ both of whose endpoints $z$, $w$ belong to the complement of $B$. The corresponding 1-form $\omega_e^t$ and the function $\beta_e^t$ do not depend on $r_B$. Hence these edges play a minor role in the discussion.
\item Edges $e$ one of whose endpoints $z$ does not belong to $B$, and one of whose endpoints $w$ belongs to $B$ contribute terms of the form
\[
\omega_e^t = \omega^t(z, \zeta_B) + r_B(\cdots)
\] 
or 
\[
\beta_e = \beta(z, \zeta_B) + r_B(\cdots).
\] 
or the analogous terms with $z$ and $\zeta_B$ interchanged.
\item Edges $e$ both of whose endpoints $z,w$ belong to $B$ contribute singularities
\[
\omega_e^t = \frac{i}{\pi} \frac{dr_B}{r_B} + (\text{terms regular in $r_B$})
\]
or 
\[
\beta_e = \frac{i}{\pi} \log r_B  + (\text{terms regular in $r_B$})
\]
\end{itemize}
We will denote the edges of the first type by $E_1(\Gamma)$, of the second type by $E_2(\Gamma)$ and $E_3(\Gamma)$, depending on whether the edge starts in $B$ or not,  and the edges of the last type by $E_4(\Gamma)$.
Collecting the terms singular in $r_B$ we hence find a splitting
\begin{align*}
\tilde \omega_\Gamma^t
&= \frac{dr_B}{r_B}\left(
\sum_{e'\in E_1(\Gamma)}\sum_{e''\in E_4(\Gamma)}\psgn{e',e''} \beta_{e'}
\prod_{e\in E_1(\Gamma)\setminus\{e'\}}\omega^t_e 
\prod_{(z,w)\in E_2(\Gamma)} \omega^t(z,\zeta_B)
\prod_{(w,z)\in E_3(\Gamma)} \omega^t(\zeta_B,z)
\prod_{e\in E_4(\Gamma)\setminus\{e''\}} \omega^t_e \right.
\\ &\phantom{=} +
\sum_{e'=(z',w')\in E_2(\Gamma)}\sum_{e''\in E_4(\Gamma)}\psgn{e',e''} \beta(z',\zeta_B)
\prod_{e\in E_1(\Gamma)}\omega^t_e 
\prod_{(z,w)\in E_2(\Gamma)\setminus\{e'\}} \omega^t(z,\zeta_B)
\prod_{(w,z)\in E_3(\Gamma)} \omega^t(\zeta_B,z)
\prod_{e\in E_4(\Gamma)\setminus\{e''\}} \omega^t_e
\\ &\phantom{=} +
\sum_{e'=(z',w')\in E_3(\Gamma)}\sum_{e''\in E_4(\Gamma)}\psgn{e',e''}  \beta(z',\zeta_B)
\prod_{e\in E_1(\Gamma)}\omega^t_e 
\prod_{(z,w)\in E_2(\Gamma)} \omega^t(z,\zeta_B)
\prod_{(w,z)\in E_3(\Gamma)\setminus\{e'\}} \omega^t(\zeta_B,z)
\prod_{e\in E_4(\Gamma)\setminus\{e''\}} \omega^t_e
\\ &\phantom{=}+ \left.
\sum_{\substack{e',e''\in E_4(\Gamma) \\ e'\neq e''}}\psgn{e',e''}\beta(z',\zeta_B)
\prod_{e\in E_1(\Gamma)}\omega^t_e 
\prod_{(z,w)\in E_2(\Gamma)} \omega^t(z,\zeta_B)
\prod_{(w,z)\in E_3(\Gamma)} \omega^t(\zeta_B,z)
\prod_{e\in E_4(\Gamma)\setminus\{e',e''\}} \omega^t_e
\right)
\\ &\phantom{=} +
\log(r_B) \left(
\sum_{e'\in E_4(\Gamma)}(-1)^{e'-1}
\prod_{e\in E_1(\Gamma)}\omega^t_e 
\prod_{(z,w)\in E_2(\Gamma)} \omega^t(z,\zeta_B)
\prod_{(w,z)\in E_3(\Gamma)} \omega^t(\zeta_B,z)
\prod_{e\in E_4(\Gamma)\setminus\{e'\}} \omega^t_e
\right)
\\ &\quad\quad\quad + (\text{terms regular in $r_B$}),
\end{align*}
where we used the notation \eqref{equ:odef}.
Note that there are no terms proportional to $\log(r_B)\frac{dr_B}{r_B}$ in the sum since the corresponding terms cancel in pairs. The terms in the first pair of parentheses are $\tilde \alpha^t$, while the terms in the second pair are $\hat \alpha^t$. Clearly, both of these terms are independent of $r_B$. One checks that 
\begin{align*}
\iota_{v_B}\hat \alpha^t
&=
\iota_{v_B}\left(
\sum_{e'\in E_4(\Gamma)} (-1)^{e'-1}
\prod_{e\in E_1(\Gamma)}\omega^t_e 
\prod_{(z,w)\in E_2(\Gamma)} \omega^t(z,\zeta_B)
\prod_{(w,z)\in E_3(\Gamma)} \omega^t(\zeta_B,z)
\prod_{e\in E_4(\Gamma)\setminus\{e'\}} \omega^t_e \right)
\\&=
\sum_{\substack{e', e''\in E_4(\Gamma) \\ e'\neq e''}} \psgn{e',e''}
\prod_{e\in E_1(\Gamma)}\omega^t_e 
\prod_{(z,w)\in E_2(\Gamma)} \omega^t(z,\zeta_B)
\prod_{(w,z)\in E_3(\Gamma)} \omega^t(\zeta_B,z)
\prod_{e\in E_4(\Gamma)\setminus\{e',e''\}} \omega^t_e
\\&=0
\end{align*}
since the summand is antisymmetric under interchange of $e'$ and $e''$. Similarly one checks that $\iota_{v_B}\tilde \alpha^t=0$. It is immediate that $d \hat \alpha^t=0$ since only closed forms appear in the definition. The fact that $\iota_{v_B}d\tilde \alpha^t=0$ follows by a similar argument, using again the antisymmetry of the summand in a double sum.
\end{proof}

\begin{Rem}
 Note that the differential form we consider are all sums of products of the forms $\omega_e^t$ or $\beta_e$. It follows that in the coordinates from section \ref{sec:3-4} the only potentially singular terms are of the form $\frac{dr_B}{r_B}$ or $\log r_B$. To check whether some form with only these singularities is regular it suffices to check that it is regular in each variable $r_B$ separately. This observation is used in the proof of the following result.
\end{Rem}

\begin{Prop}\label{prop:iotaregular}
Let $\Gamma$ be an admissible graph. Then in every chart $U_i$ the forms $\iota_{\xi_i}\omega^t_\Gamma$ and $\iota_{\xi_i}\tilde \omega^t_\Gamma$ are regular.
\end{Prop}
\begin{proof}
The proof is the same as the proof of \cite{ARTW}*{Proposition 4}: By Proposition \ref{prop:splitting} the form $\iota_{\xi_i}\tilde \omega^t_\Gamma$ is regular in each of the $r_B$ and hence regular.
\end{proof}

\begin{Thm}\label{thm:omegareg}
Let $\Gamma$ be an admissible graph such that $|E(\Gamma)|=2n+m-2$, i.~e., such that the form $\omega_\Gamma^t$ is of top degree. Then the form $\omega_\Gamma^t$ is regular.
Let $\Gamma$ be an admissible graph such that $|E(\Gamma)|=2n+m-3$, i.~e., such that the form $\tilde \omega_\Gamma^t$ is of top degree. Then the form $\tilde \omega_\Gamma^t$ is regular.
\end{Thm}
\begin{proof}
The proof is the same as the proof of \cite{ARTW}*{Theorem 2}. Fix some nested family $i$ and pick a connection $\theta$ for the action of torus $T_i$ on $U_i$, and $\theta_1,\dots, \theta_{\mathrm{dim} T_i}$ for the parts corresponding to the various $S^1$ actions. Write
\[
 vol_i = \theta_{\mathrm{dim} T_i} \wedge \cdots \wedge \theta_2 \wedge \theta_1.
\]
Since $\omega_\Gamma^t$ (respectively $\tilde \omega_\Gamma^t$) is a top form we can write on $U_i$
\begin{align*}
\omega_\Gamma^t &= vol_i\wedge \iota_{\xi_i}\omega_\Gamma^t 
&
\tilde \omega_\Gamma^t &= vol_i\wedge \iota_{\xi_i}\tilde  \omega_\Gamma^t. 
\end{align*}
The right hand terms are regular by Proposition \ref{prop:iotaregular}, and hence so are $\omega_\Gamma^t$ and $\tilde \omega_\Gamma^t$.
\end{proof}

\begin{Prop}\label{prop:regularizable}
\hfill
\begin{itemize}
\item
Let $\Gamma$ be an admissible graph such that $|E(\Gamma)|=2n+m-3$, i.~e., such that the form $\omega_\Gamma^t$ is of top minus one degree. Then in every chart $U_i$ the form $\omega_\Gamma^t$ admits a decomposition
\[
\omega_\Gamma^t = \sum_{B=1}^k \frac{dr_B}{r_B} \wedge \alpha_B +(\text{regular terms})
\]
where $\iota_{\xi_i}\alpha_B=0$ and $\iota_{\xi_i}d\alpha_B=0$ for all $B$.
\item Let $\Gamma$ be an admissible graph such that $|E(\Gamma)|=2n+m-2$, {\em i.e.} such that the form $\tilde \omega_\Gamma^t$ is of top minus one degree. Then in every chart $U_i$ the form $\tilde\omega_\Gamma^t$ admits a decomposition
\[
\tilde\omega_\Gamma^t = \sum_{B=1}^k \frac{dr_B}{r_B} \wedge \tilde\alpha_B +\log (r_B) \hat \alpha_B+(\text{regular terms})
\]
where $\iota_{\xi_i}\tilde \alpha_B=\iota_{\xi_i}\hat \alpha_B=0$ and $\iota_{\xi_i}d\tilde \alpha_B=\iota_{\xi_i}d\hat \alpha_B=0$ for all $B$.
\end{itemize}
\end{Prop}
\begin{proof}
The proof is a copy of the proof of \cite{ARTW}*{Proposition 5}.
\end{proof}

\begin{Prop}\label{prop:bdry1}
Let $\Gamma$ be an admissible graph, $U_i$ be a chart and let $B$ be a vertex of the tree defining $i$. Denote by $\p_B U_i$ the codimension one boundary stratum of $U_i$ corresponding to $B$, and denote $\p_B\Gamma=\Gamma'\cup \Gamma''$. Then 
\begin{align*}
\iota_{v_B}\omega^t_\Gamma \mid_{\p_B U_i} &= \iota_{v_B} \omega^t_{\Gamma'}\mid_{\p_B U_i} \wedge \omega^t_{\Gamma''}\mid_{\p_B U_i} \\
\iota_{v_B}\tilde \omega^t_\Gamma \mid_{\p_B U_i} &= 
\iota_{v_B} \tilde\omega^t_{\Gamma'}\mid_{\p_B U_i} \wedge \omega^t_{\Gamma''}\mid_{\p_B U_i} +
\iota_{v_B} \omega^t_{\Gamma'}\mid_{\p_B U_i} \wedge \tilde \omega^t_{\Gamma''}\mid_{\p_B U_i} .
\end{align*}
\end{Prop}
\begin{proof}
Same as the proof of \cite{ARTW}*{Proposition 6}.
\end{proof}

\subsection{The boundary terms}\label{sec:bdryterms}
Let $\Gamma\in \dgra_{n,k}$ be a directed graph, cf. section \ref{ss-2-1}. We can associate to it the following differential form on the configuration space of points $\Conf_A$, for $A=[n]$.
\begin{equation}\label{equ:tildebetadef}
\tilde \beta_\Gamma^t
=
\sum_{\substack{e,e'\in E(\Gamma)\\e\neq e'}} \psgn{e',e} 
\frac 1 {\pi i} \log|z_{s(e')}-z_{t(e')}|
\prod_{e''\in E(\Gamma)\setminus \{e,e'\}}
\frac 1 {2\pi i} \left( (1-t) d\log(z_{s(e'')}-z_{t(e'')}) + t \, d\log(\bar z_{s(e'')}-\bar z_{t(e'')})\right).
\end{equation}
As before, the product over edges in \eqref{equ:tildebetadef}  shall be taken in the order dictated by the positions of the vertices.

\begin{Lem}
The differential form $\tilde \beta_\Gamma^t$ is $\bbC^\times \ltimes \bbC$-basic, and hence descends to a differential form on the quotient $C_A/S^1 = \Conf_A / \bbC^\times \ltimes \bbC$ for all $t\in \R$. We denote the differential form on the quotient by $\tilde \beta_\Gamma^t$ as well, abusing notation.
\end{Lem}
\begin{proof}
It is clear that the form is basic under translations.
Let $v$ be the rotation generating vector field. Then 
\begin{align*}
\iota_{v} \tilde \beta_\Gamma^t &=
\sum_{\substack{e, e',e''\in E(\Gamma) \\ e\neq e' \neq e'' \neq e} }\psgn{e',e,e''} 
\frac 1 {\pi i} \log|z_{s(e')}-z_{t(e')}|
\\&\quad\quad\quad\quad\quad\quad\quad
\prod_{e'''\in E(\Gamma)\setminus \{e,e',e''\}}
\frac 1 {2\pi i} \left( (1-t) d\log(z_{s(e''')}-z_{t(e''')}) + t d\log(\bar z_{s(e''')}-\bar z_{t(e''')})\right)\\
&=0
\end{align*}

since the summand is antisymmetric under interchange of $e$ and $e''$. For the vector field $v'$ generating the scaling transformation an almost identical calculation shows that $\iota_{v} \tilde \beta_\Gamma^t=0$.

All terms appearing in the definition of $\tilde \beta_\Gamma^t$ are rotation invariant, so the Lie derivative with respect to $v$ vanishes, $L_v\tilde \beta_\Gamma^t=0$. Furthermore, compute the Lie derivative
\[
L_{v'} \tilde \beta_\Gamma^t = 
\sum_{\substack{e,e'\in E(\Gamma)\\ e\neq e'} }\psgn{e',e}
\frac 1 {\pi i} 
\prod_{e''\in E(\Gamma)\setminus \{e,e'\}}
\frac 1 {2\pi i} \left( (1-t) d\log(z_{s(e'')}-z_{t(e'')}) + t d\log(\bar z_{s(e'')}-\bar z_{t(e'')})\right)
=0
\]
by antisymmetry of the summand under interchange of $e$ and $e'$.
\end{proof}

\begin{Prop}\label{prop:tildeomega}
If $\Gamma$ is a graph such that $|E(\Gamma)|=2|V(\Gamma)|-4$, then the top degree form $\tilde \beta_\Gamma^t$ extends to a regular form on $\overline C_A/S^1$ and furthermore
\[
\tilde \beta_\Gamma^t = (4t(1-t))^{|V(\Gamma)|-2}\tilde \beta_\Gamma^{\frac 1 2}.
\]
\end{Prop}
\begin{proof}
Since $C_A/S^1$ is a complex manifold, the only terms contributing upon expanding \eqref{equ:tildebetadef} have an equal number of terms $dz_\cdot$ and $d\bar z_\cdot$. But the former such terms are always rescaled by $(1-t)$, while the latter terms are rescaled by $t$. Hence
\[
\tilde \beta_\Gamma^t = (4t(1-t))^{|V(\Gamma)|-2}\tilde \beta_\Gamma^{\frac 1 2}
\]
and it is sufficient to consider the $t=\frac 1 2$ case. This is advantageous since the forms $d\arg (z-w)$ are regular on $\overline C_A$, and hence the only singularity can potentially be contributed by the logarithm term.

Fix a chart $U_i$ and some subset $B$ in the nested family $i$. 
Our goal is to show that $\tilde \beta_\Gamma^{\frac 1 2}$ is regular in $r_B$. 
It is sufficient to show that $\iota_{v_B}\tilde \beta_\Gamma^t$ is regular since $\tilde \beta_\Gamma^t$ is a top degree form.
As in the proof of Proposition~\eqref{prop:splitting} there are three sorts of edges in $\Gamma$.
\begin{itemize}
\item If both endpoints of an edge $e$ are in the complement of $B$, then the edge cannot contribute a singularity in $r_B$.
\item If one endpoint, say $z$, is in the complement of $B$, and the other, say $w$, is in $B$ then the edge can contribute either a factor $d\arg(z-\zeta_B)+r_B(\cdots)$ or $\log|z-\zeta_B|+r_B(\cdots)$.
\item If both endpoints $z$, $w$ are in $B$, then there can be a divergent contribution $\log r_B+(\text{terms regular in $r_B$})$. More precisely, let us use coordinates 
\[
z=\zeta_B+r_B z',\quad w=\zeta_B+r_B w',
\]
for some edge $e=(z,w)$ in $E_3(\Gamma)$.
\end{itemize} 
Let us call these three sets of edges $E_j(\Gamma)$, $j=1,2,3$.
Collecting all potentially singular terms we can hence write
\begin{align*}
\iota_{v_B} \tilde \beta_\Gamma^{\frac 1 2} 
&\propto  
\iota_{v_B}\left(\sum_{e\in E(\Gamma)} \sum_{e'\neq e \in E_3(\Gamma)}
\psgn{e',e}\log(r_B)\prod_{e''\in E_1(\Gamma)\smallsetminus\{ e\}}d\arg(z_{s(e'')}-z_{t(e'')})\prod_{e'''=(z,w)\in E_2(\Gamma)\smallsetminus\{ e\}}d\arg(\zeta_B-z)\right. \\
&\phantom{=}\quad\quad\quad\quad\quad\quad\quad\quad\left.\prod_{\hat e=(z',w') \in E_3(\Gamma)\smallsetminus\{e,e'\}}d\arg(z'-w')\right)+\left(\text{terms regular in $r_B$}\right)\\
&=
\sum_{e\in E(\Gamma)}
\sum_{e'\neq \hat e \in E_3(\Gamma)}\psgn{e',e, \hat e}\log(r_B)
\prod_{e''\in E_1(\Gamma)\smallsetminus\{ e\}}d\arg(z_{s(e'')}-z_{t(e'')})
\prod_{e'''=(z,w)\in E_2(\Gamma)\smallsetminus\{ e\}}d\arg(\zeta_B-z) \\
&\phantom{=}\quad\quad\quad\quad\quad\quad\quad\quad 
\prod_{\bar e=(z',w') \in E_3(\Gamma)\smallsetminus\{e,e',\hat e\}}d\arg(z'-w')+\left(\text{terms regular in $r_B$}\right)\\
\\&=0+(\text{terms regular in $r_B$}).
\end{align*}
Here we used again that the summand is antisymmetric under interchange of $e'$ and $\hat e$, whence the singular terms may be dropped.
\end{proof}

Using the form $\tilde \beta^t_\Gamma$, the boundary terms in Proposition \ref{prop:bdry1} may be written down more explicitly. 
\begin{Thm}\label{thm:factoring}
Let $\Gamma\in kgra_{n,m}$ be a graph such that $|E(\Gamma)|=2n+m-3$. Let $U_i$ be a chart and $B$ be a vertex of the tree defined by $i$. 
If $\p_B U_i$ is a type I boundary stratum then
\[
\iota_{v_B} \mathrm{Reg}_B\!\left(\omega^t_{\Gamma}\right) 
=
\begin{cases}
(4t(1-t))^{|V(\Gamma')|-2} \iota_{v_B}\omega^{\frac 1 2}_{\Gamma'} \mid_{\p_B U_i} \wedge \omega^t_{\Gamma''}\mid_{\p_B U_i}
& \text{if } |E(\Gamma')|=2|V(\Gamma')|-3 \\
0 & \text{otherwise}
\end{cases}
\]
If $\p_B U_i$ is a type II boundary stratum then
\[
 \mathrm{Reg}_B\!\left(\omega^t_{\Gamma}\right)=\omega^t_{\Gamma'}\mid_{\p_B U_i}\wedge \omega^t_{\Gamma''}\mid_{\p_B U_i}.
\]

Similarly, let $\Gamma\in kgra_{n,m}$ be a graph such that $|E(\Gamma)|=2n+m-2$. If $\p_B U_i$ is a type I boundary stratum then
\[
\iota_{v_B}\Reg_B\!\left(\tilde \omega^t_{\Gamma}\right) 
=
\begin{cases}
(4t(1-t))^{|V(\Gamma')|-2} \tilde \beta^{\frac 1 2}_{\Gamma'} \mid_{\p_B U_i} \wedge \omega^t_{\Gamma''}\mid_{\p_B U_i}
& \text{if } |E(\Gamma')|=2|V(\Gamma')|-2 \\
(4t(1-t))^{|V(\Gamma')|-2} \iota_{v_B}\omega^{\frac 1 2}_{\Gamma'}\mid_{\p_B U_i} \wedge \tilde \omega^t_{\Gamma''}\mid_{\p_B U_i}
& \text{if } |E(\Gamma')|=2|V(\Gamma')|-3 \\
0 & \text{otherwise}.
\end{cases}
\]
If $\p_B U_i$ is a type II boundary stratum then
\[
 \Reg_B\!\left(\tilde \omega^t_{\Gamma}\right) 
 =
 \tilde \omega^t_{\Gamma'}\mid_{\p_B U_i} \wedge \omega^t_{\Gamma''}\mid_{\p_B U_i}
 +
 \omega^t_{\Gamma'}\mid_{\p_B U_i} \wedge \tilde \omega^t_{\Gamma''}\mid_{\p_B U_i}
\]
\end{Thm}
\begin{proof}
For a type II boundary stratum the differential forms involved are regular in $r_B$, hence
\begin{align*}
 \mathrm{Reg}_B\!\left(\omega^t_{\Gamma}\right) &= \omega^t_{\Gamma}\mid_{\p_B U_i}
 =
 \omega^t_{\Gamma'}\mid_{\p_B U_i}\wedge \omega^t_{\Gamma''}\mid_{\p_B U_i}
 \\
 \Reg_B\!\left(\tilde \omega^t_{\Gamma} \right)&= \tilde \omega^t_{\Gamma} \mid_{\p_B U_i}
 =
 \tilde \omega^t_{\Gamma'}\mid_{\p_B U_i} \wedge \omega^t_{\Gamma''}\mid_{\p_B U_i}
 +
 \omega^t_{\Gamma'}\mid_{\p_B U_i} \wedge \tilde \omega^t_{\Gamma''}\mid_{\p_B U_i}.
\end{align*}

For a type I boundary stratum, and for a given $T_i$-connection $\theta$ with components $\theta_1, \dots \theta_k$, $k=\mathrm{dim} T_i$ we can write by definition
\[
  \mathrm{Reg}_B\!\left( \omega^t_{\Gamma}\right) = \theta_1\cdots \theta_k \iota_{\xi_k}\cdots \iota_{\xi_1} \omega^t_{\Gamma} \mid_{\p_B U_i}
\]
and hence (say $B\sim 1$ w.l.o.g.)
\begin{align*}
\iota_{v_B} \mathrm{Reg}_B\!\left( \omega^t_{\Gamma}\right)
&=
 \theta_2\cdots \theta_k \iota_{\xi_k}\cdots \iota_{\xi_1} \omega^t_{\Gamma} \mid_{\p_B U_i}
\\&=
 (\theta_2\cdots \theta_k \iota_{\xi_k}\cdots \iota_{\xi_2}) \iota_{\xi_1}\omega^t_{\Gamma} \mid_{\p_B U_i}
\\& =\iota_{v_B} \omega^t_{\Gamma'}\mid_{\p_B U_i} \wedge \omega^t_{\Gamma''}\mid_{\p_B U_i}.
\end{align*}
Here we used Proposition \ref{prop:bdry1} and the fact that the operator $(\theta_2\cdots \theta_k \iota_{\xi_k}\cdots \iota_{\xi_2})$ acts as the identity on top degree differential forms.
By copying the argument in the proof of Proposition \ref{prop:tildeomega} one can see that
\[
\iota_{v_B} \omega^t_{\Gamma'}\mid_{\p_B U_i} =  \left(4t(1-t)\right)^{|V(\Gamma')|-2} \iota_{v_B} \omega^{\frac 1 2}_{\Gamma'}\mid_{\p_B U_i}
\]
and we hence obtain the first equation.

Similarly, using again Proposition \ref{prop:bdry1}
\begin{align*}
\iota_{v_B} \mathrm{Reg}_B\!\left( \tilde \omega^t_{\Gamma}\right)
&=
\iota_{v_B} \tilde\omega^t_{\Gamma'}\mid_{\p_B U_i} \wedge \omega^t_{\Gamma''}\mid_{\p_B U_i} +
\iota_{v_B} \omega^t_{\Gamma'}\mid_{\p_B U_i} \wedge \tilde \omega^t_{\Gamma''}\mid_{\p_B U_i}
\end{align*}
and using Proposition \ref{prop:tildeomega} finishes the proof.
\end{proof}

\section{A family of stable formality morphisms}\label{s-5}
With the help of the family $\omega^t$, we construct a family of stable formality morphisms in the sense of Definition~\ref{d-stable}.

The family $\mathcal U^t$ is defined by a sum over graphs formula,
\begin{equation}\label{eq-stable}
\mathcal U^t(\mathsf t_{n,m}^\mathfrak o)=\sum_{\Gamma}\varpi_\Gamma^t\ \Gamma,\ n\geq 2,
\end{equation}
where we sum over a set of graphs in $\cup_k \kgra_{n,m,k}$ forming a basis of $\KGra(m,n)^\mathfrak o$.
Note that the weights $\varpi^t_\Gamma$ depend polynomially on the variable $t$ and are explicitly defined {\em via}
\begin{align}\label{eq-t-weight-varpi}
\varpi_\Gamma^t&=\int_{C_{n,m}^+}\omega_\Gamma^t 
\end{align}

The goal of this section is to show the following result.
\begin{Prop}\label{p-sec5-main}
 The integrals \eqref{eq-t-weight-varpi} exist, and $\mathcal U^t$ is a stable formality morphism for all $t$.
\end{Prop}

Note that $\mathcal U^t$ is a morphism of degree $0$: namely, $\varpi^t_\Gamma$ is non-trivial, only if the degree of $\Gamma$ in $\KGra(n,m)^\mathfrak o$ equals $2-2n-m$, which is precisely the degree of the generator $\mathsf t_{n,m}^\mathfrak o$. 

The first statement of Proposition \ref{p-sec5-main} follows immediately from Theorem \ref{thm:omegareg} and the compactness of the configuration spaces. The second statement will be shown along the lines of Kontsevich's original proof of his formality Theorem, but using the regularized Stokes' Theorem (Theorem \ref{thm:regstokes}) instead of the ordinary Stokes' Theorem.

\subsection{The Maurer--Cartan equation for \texorpdfstring{$\mathcal U^t$}{Ut}}\label{ss-5-2}
The Maurer--Cartan equation for $\mathcal U^t$ is equivalent, by the very definition of stable formality morphism, to the condition that $\mathcal U^t$ intertwines the dg structures on the $2$-colored operads $\OC$ and $\KGra$: it translates into an infinite family of quadratic equations for the integral weights~\eqref{eq-t-weight-varpi}.

In fact, a stable formality morphism $\mathcal F$ as in Definition~\ref{d-stable} satisfies 
\[
\mathcal F\circ d_\OC=0 
\]
as $\KGra$ has trivial differential. More precisely, the boundary conditions on $\mathcal F$ imply that we only have to verify the identity $F(d_\OC(\mathsf t_{n,m}^\mathfrak o))=0$, for $n\geq 2$ and $m\geq 0$.
Again the boundary conditions for $\mathcal F$ and the compatibility of $\mathcal F$ with the operadic structures on $\OC$ and $\KGra$ imply that the previous identity can be re-written as
\[
\begin{aligned}
0&=\sum_{\Gamma\in \KGra(n-1,m)^\mathfrak o}\alpha_\Gamma\left(\sum_{i=1}^{n-1}\Gamma\circ_i\SNbr\right)+\sum_{n_1=0}^n\sum_{m_1=0}^m\frac{1}2\left[\sum_{\Gamma_1\in\KGra(n_1,m_1)^\mathfrak o}\alpha_{\Gamma_1}\Gamma_1,\sum_{\Gamma_2\in\KGra(n-n_1,m+1-m_1)^\mathfrak o}\alpha_{\Gamma_2}\Gamma_2\right]=\\
&=\delta(\mathcal F(\mathsf t_{n-1,m}^\mathfrak o))+\sum_{n_1=0}^n\sum_{m_1=2}^{m-1}\frac{1}2\left[\mathcal F(\mathsf t_{n_1,m_1}^\mathfrak o),\mathcal F(\mathsf t_{n-n_1,m+1-m_1}^\mathfrak o)\right],
\end{aligned}
\]
where $\delta=[\SNbr,\bullet]$.

The first equation can be re-written as an infinite family of quadratic equations for the weights $\alpha_\Gamma$.
Kontsevich showed that for $\alpha_\Gamma=\varpi_\Gamma^{\frac 1 2}$ these quadratic equations are exactly the  
quadratic equations obtained by applying the Stokes formula to the regular forms $\omega_\Gamma^{\frac 1 2}$, where $|E(\Gamma)]=2n+m-3$,
\[
0=\int d \omega_\Gamma^{\frac 1 2} 
=
\int_{\p} \omega_\Gamma^{\frac 1 2}
=
\sum_B
\int_{\p_B} \omega_\Gamma^{\frac 1 2}
\]
where the sum is over all codimension 1 boundary strata $B$. 
The expression on the right-hand side factorizes:
\[
\int_{\p_B} \omega_\Gamma^{\frac 1 2}
=
\int_{} \omega_{\Gamma'}^{\frac 1 2}
\int_{} \omega_{\Gamma''}^{\frac 1 2}.
\]
Now the quadratic equations for the $\alpha_\Gamma$ are recovered provided that the Kontsevich vanishing property holds:

\vspace{3mm}

{\bf Kontsevich vanishing property:} The contributions from the boundary strata of type I in the above formula vanish, unless the graph $\Gamma'$ consists of exactly 2 vertices connected by an edge.

\vspace{3mm}

Our task is to extend the Kontsevich proof from $t=\frac 1 2$ to all $t$, i.~e., to show that the quadratic equations are satisfied for $\alpha_\Gamma=\varpi_\Gamma^{t}$ for all $t$. 
We can follow the lines of the Kontsevich proof except that we apply the regularized Stokes Theorem \ref{thm:regstokes} to the singular differential form $\omega_\Gamma^{t}$, where $|E(\Gamma)]=2n+m-3$. This form is regularizable by Proposition \ref{prop:regularizable} and hence
\[
0=\int d \omega_\Gamma^{t} 
=
\int_{\p} \mathrm{Reg}\!\left(\omega_\Gamma^{t}\right)
=
\sum_B
\int_{\p_B} \mathrm{Reg}\!\left(\omega_\Gamma^{t}\right).
\]
The expression on the right factorizes according to Theorem \ref{thm:factoring}.
\[
\int_{\p_B} \mathrm{Reg}\!\left(\omega_\Gamma^{t}\right)
=
\int_{} \mathrm{Reg}\!\left(\omega_{\Gamma'}^{t}\right)
\int_{} \omega_{\Gamma''}^{t}
=
\begin{cases}
\int_{} \omega_{\Gamma'}^{t}
\int_{} \omega_{\Gamma''}^{t} 
& \text{if $B$ describes a type II stratum} 
\\
 \left(4t(1-t)\right)^{|V(\Gamma')|-2}
\int_{} \mathrm{Reg}\!\left(\omega_{\Gamma'}^{\frac 1 2}\right)
\int_{} \omega_{\Gamma''}^{t} 
& \text{if $B$ describes a type I stratum}
\end{cases}
\]
To recover the quadratic identities and hence show Proposition \ref{p-sec5-main} it hence suffices to verify the Kontsevich vanishing condition for the type I boundary strata. 
However, since in the type I case the boundary contribution is a rescaling of that present when $t=\frac 1 2$, the Kontsevich vanishing property for general $t$ is equivalent to the Kontsevich vanishing property for the $t=\frac 1 2$ case. The latter has been proven by Kontsevich \cite{K}*{section 6.6.1}. This shows Proposition \ref{p-sec5-main}.
\hfill\qed

\section{A family of cocycles in Kontsevich's graph complex}\label{s-6}
In the preceding section we have constructed a family $\mathcal U^t$ of stable formality morphisms over $\mathbb R$. The weights of graphs in $\mathcal U^t$ containing only a fixed number of vertices depend polynomially on $t$ by construction.
In this section we compute the derivative with respect to $t$ of $\mathcal U^t$. The result will be the following.

\begin{Prop}\label{p-sec6-main}
There is a family of graph cocycles $x^t$ and a family of homotopies (i. e. degree $0$ elements of $\mathrm{Conv}(\mathfrak{oc}^\vee,\KGra)$) $\widetilde{\mathcal U}^t$ such that
\begin{equation}\label{equ:stablehomotopy}
\partial_t\mathcal U^t=x^t\cdot \mathcal U^t+\delta\widetilde{\mathcal U}^t+\left[\mathcal U^t,\widetilde{\mathcal U}^t\right].
\end{equation}
Here $\delta$ and $[,]$ are the differential and Lie bracket on the convolution dg Lie algebra $\mathrm{Conv}(\mathfrak{oc}^\vee,\KGra)$, while $x^t\cdot \mathcal U^t$ denotes the action of the element of $\GC$ on the stable formality morphism $\mathcal U^t$, cf. Subsection \ref{ss-2-3}.
\end{Prop}

In fact, the family of homotopies $\tilde \mU^t$ is defined as follows.
\begin{equation}\label{equ:homotopy}
\tilde \mU^t(\mathsf t_{n,m}^\mathfrak o)=\sum_{\Gamma}\tilde \varpi_\Gamma^t\ \Gamma,\ n\geq 2,
\end{equation}
where we sum over a set of graphs in $\cup_k \kgra_{n,m,k}$ forming a basis of $\KGra(m,n)^\mathfrak o$.
The the weights $\tilde \varpi^t_\Gamma$ depend polynomially on the variable $t$ and are explicitly defined {\em via}
\begin{align}\label{eq-t-weight-tildevarpi}
\tilde \varpi_\Gamma^t&=\int_{C_{n,m}^+} \tilde\omega_\Gamma^t 
\end{align}
Note that by Theorem \ref{thm:omegareg} the above integral converges.

Similarly we define the family $x^t$ as 
\be{equ:xtdef}
x^t = \sum_\Gamma c_\Gamma^t \Gamma  \in \GC\subset \fdGC
\ee
with the weights being 
\be{equ:xtweightdef}
c_\Gamma^t = \int_{C_{n}/S^1 } \tilde \beta_\Gamma^t = (4t(1-t))^{|V(\Gamma)|-2} a_\Gamma^t
\ee
where $n=|V(\Gamma)|$ and $\tilde \beta_\Gamma^t$ is as in \eqref{equ:tildebetadef}.
Our goal for the remainder of this section is to show Proposition \ref{p-sec6-main}, i.~e., to show \eqref{equ:stablehomotopy} and that the $x^t$ are indeed cocycles. 

\subsection{The derivative of \texorpdfstring{$\mathcal U^t$ with respect to $t$}{Ut with respect to t}}\label{ss-6-1}
Let us now consider the family $\mathcal U^t$ and let us compute the derivative of $\mathcal U^t(\mathsf t_{n,m}^\mathfrak o)$, $n\geq 2$, with respect to $t$. 
Let $\Gamma$ be a graph in $\KGra(n,m)^\mathfrak o$ of degree $2-2n-m$.
The dependence of 
\[
\varpi^t_\Gamma=\int_{\overline C_{n,m}^+}\omega^t_\Gamma,
\]
is clearly polynomial in $t$. Hence one may interchange the derivative with the integral
\[
\p_t \varpi^t_\Gamma=\int_{\overline C_{n,m}^+}\omega^t_\Gamma
=
\int_{\overline C_{n,m}^+}\p_t  \omega^t_\Gamma
=
\int_{\overline C_{n,m}^+}d  \tilde \omega^t_\Gamma.
\]
By Proposition \ref{prop:regularizable} the form $\tilde \omega^t_\Gamma$ is regularizable. Hence we may apply the regularized Stokes' Theorem (Theorem \ref{thm:regstokes}) and compute
\[
\int_{\overline C_{n,m}^+}d  \tilde \omega^t_\Gamma
=
\int_{\p \overline C_{n,m}^+} \mathrm{Reg}\!\left(\tilde \omega^t_\Gamma\right)
=
\sum_B
\int_{\p_B \overline C_{n,m}^+} \mathrm{Reg}\!\left(\tilde \omega^t_\Gamma\right).
\]
By Theorem \ref{thm:factoring} the right hand side may be evaluated using the formula
\[
\int_{\p_B \overline C_{n,m}^+} \mathrm{Reg}\!\left(\tilde \omega^t_\Gamma\right)
=
\int_{}\mathrm{Reg}\!\left(\tilde \omega^t_{\Gamma'}\right) \int_{} \omega^t_{\Gamma''} 
+
\int_{}\mathrm{Reg}\!\left(\omega^t_{\Gamma'}\right) \int_{}\tilde \omega^t_{\Gamma''}.
\]

By the Kontsevich vanishing property the second term vanishes for type I strata $B$ unless the graph $\Gamma'$ contains exactly two vertices and one edge. The total contribution of such terms produces the term $\delta\tilde \mU^t$ in \eqref{equ:stablehomotopy}.
For type I strata, the total contribution of the first terms yield the term $x^t\cdot \tilde \mU^t$ in \eqref{equ:stablehomotopy}. Similarly, the contribution of the type II strata is $[\mU^t, \tilde\mU^t]$ and hence the equality \eqref{equ:stablehomotopy} is shown.

\subsection{The family of graph cocycles}\label{ss-6-2} 

Let us also remark that from equation \eqref{equ:stablehomotopy} and the Maurer--Cartan equation 
\[
 \delta \mathcal U^t + \frac 1 2\left[\mathcal U^t, \mathcal U^t\right]=0
\]
for $\mathcal U^t$ it follows that $x^t$ is a family of graph cocycles. 
Indeed, taking the derivative of the Maurer--Cartan equation we obtain
\[
0 = \delta\left(\partial_t \mathcal U^t\right) + \left[\partial_t \mathcal U^t, \mathcal U^t\right].
\]
Inserting \eqref{equ:stablehomotopy} we obtain 
\begin{align*}
0 &=\delta(x^t\cdot \mathcal U^t)
+\delta\left(\left[\mathcal U^t,\widetilde{\mathcal U}^t\right]\right)
+\left[x^t\cdot \mathcal U^t+\delta\widetilde{\mathcal U}^t+\left[\mathcal U^t,\widetilde{\mathcal U}^t\right],\mathcal U^t\right]=\\
&=(\delta x^t)\cdot \mathcal U^t
+x^t \cdot \left( \delta \mathcal U^t + \frac 1 2 \left[\mathcal U^t,\mathcal U^t\right]\right)
+\left[\delta  \mathcal U^t + \frac 1 2 \left[\mathcal U^t,\mathcal U^t\right],
\widetilde{\mathcal U}^t
\right]=\\
&=(\delta x^t)\cdot \mathcal U^t.
\end{align*}
For the last equality we again used the Maurer--Cartan equation for $\mathcal U^t$.
Finally note that the action 
\begin{align*}
\dfGC &\to \mathrm{Conv}(\mathfrak{oc}^\vee,\KGra) \\
x &\mapsto x\cdot \mathcal U^t
\end{align*}
is an injective map. Hence it follows that $\delta x^t=0$, {\em i.e.} $x^t$ is a graph cocycle for all $t$.

\section{Several operads of Lie algebras and \texorpdfstring{$\grt_1$}{grt1}}\label{s-7}
In this section we review the operads of Lie algebras $\tder$, $\sder$ and $\mathfrak t$. We recall their algebraic definition, and their combinatorial-graphical interpretation.

We furthermore introduce the Grothendieck--Teichm\"uller Lie algebra $\grt_1$ defined by V.~Drinfel{\cprime}d~\cite{Dr}, and we explore in detail the connection between graph cocycles in $\GC$ of degree $0$ and $\grt_1$. We construct a suitable alternative to the construction in~\cite{Will}*{Section 6} of the map from $H^0(\GC)$ to $\grt_1$, through which we compute the image $\tau^t$ in $\grt_1$ of the graph cocycle $x^t$ from Section~\ref{s-6}.

\subsection{The operads of Lie algebras \texorpdfstring{$\tder$, $\sder$ and $\mathfrak t$}{tder, sder and t}}\label{ss-7-1}
First of all, we denote by $\mathfrak{Lie}_k$, for $k\geq 1$, the degree completion of the free Lie algebra over $\mathbb K$ with $k$ generators, which we typically denote by $\{X_1,\dots,X_k\}$.
There is a natural grading on $\mathfrak{Lie}_k$ by the number of Lie brackets appearing in Lie monomials. For example, $X_i$ has degree $0$, $[X_i,X_j]$ has degree $1$ {\em etc.}.

Following~\cite{AT-2}*{Section 3} we consider the vector space $\tder_k$ of ``tangential'' derivations, i. e. $\mathbb K$-linear derivations $u$ of $\Lie_k$ of the form
\[
u(X_i)=[X_i,u_i]
\]
for some $u_i\in\Lie_k$, $i=1,\dots,k$.
The standard Lie bracket on the $\mathbb K$-linear derivations of $\Lie_k$ restricts to $\tder_k$, making it into a Lie algebra.
More precisely, an element $u$ of $\tder_k$ is uniquely represented by a $k$-tuple $(u_1,\dots,u_k)$ of elements of $\Lie_k$ with the property that the term of order $1$ with respect to $x_i$ in $u_i$ is $0$. The bracket of two elements $u=(u_1,\dots,u_k), v=(v_1,\dots,v_k)$ of $\tder_k$ is then explicitly given by
\[
[u,v]=\left(u(v_1)-v(u_1)+[u_1,v_1],\dots,u(v_k)-v(u_k)+[u_k,v_k]\right).
\]
We may consider further the subspace $\sder_k\subset \tder_k$, consisting of all tangential derivations satisfying the additional property
\[
u\!\left(\sum_{i=1}^k X_i\right)=\sum_{i=1}^k[X_i,u_i]=0.
\]
It is pretty obvious that $\sder_k$ defines a Lie subalgebra of $\tder_k$.

The Kohno--Drinfel{\cprime}d Lie algebra $\mathfrak t_k$, for $k\geq 2$, is the free Lie algebra spanned by generators $t_{ij}$, $1\leq i\neq j \leq k$, {\em modulo} the following relations:
\begin{align*}
t_{ij}&=t_{ji} & [t_{ij},t_{kl}]&=0,\ \{i,j\}\cap\{k,l\}=\emptyset & [t_{ij},t_{ik}+t_{jk}]&=0,\ k\neq i,j.
\end{align*}
Observe that $\mathfrak t_2\cong \Lie_1$ is $1$-dimensional. For $k\geq 3$, $\mathfrak t_k$ can be written as a semidirect product of Lie algebras 
\[
\mathfrak t_k=\mathfrak t_{k-1} \ltimes \Lie(t_{1k},\dots,t_{k-1,k}),
\]
where $\mathfrak t_{k-1}$ is generated by $t_{ij}$, $1\leq i<j\leq k-1$; the free Lie algebra $\Lie(t_{1k},\dots,t_{k-1,k})=\Lie_{k-1}$ is an ideal with respect to the action of $\mathfrak t_{k-1}$.
Observe that $c=\sum_{1\leq i<j\leq k}t_{ij}$ belongs to the center of $\mathfrak t_k$.
Furthermore, there is an injective map from $\mathfrak t_k$ to $\tder_k$ given by the assignment
\[
\mathfrak t_k\ni t_{ij}\mapsto t_{ij}=\left(0,\dots,\underset{\text{$i$-th}}{\underbrace{X_j}},\dots,\underset{\text{$j$-th}}{\underbrace{X_i}},\dots,0\right),\ 1\leq i<j\leq k.
\]

Elements of $\tder_k$ and $\sder_k$ admit combinatorial representations, which we now discuss.


In particular, Lie monomials in $\Lie_k$ of degree $n\geq 1$ are naturally associated with directed rooted trees with $k$ external vertices and $n$ internal vertices, with the additional properties that every internal vertex has exactly one incoming and two outgoing edges, except the root (which has only two outgoing edges), and there is no edge outgoing from any one of the internal vertices.
Such a directed rooted tree is called a Lie tree with $k$ external and $n$ internal vertices.
Moreover, one has to quotient the graded vector space spanned by Lie trees with respect to the anti-symmetry relation and the IHX relation: the anti-symmetry relation encodes the skew-symmetry of the Lie bracket; the IHX relation, on the other hand, encodes the Jacobi identity.
An example of a Lie tree and the corresponding Lie monomial in $\Lie_3$, the anti-symmetry and the IHX relations are depicted in Figure~\ref{fig-IHX}.
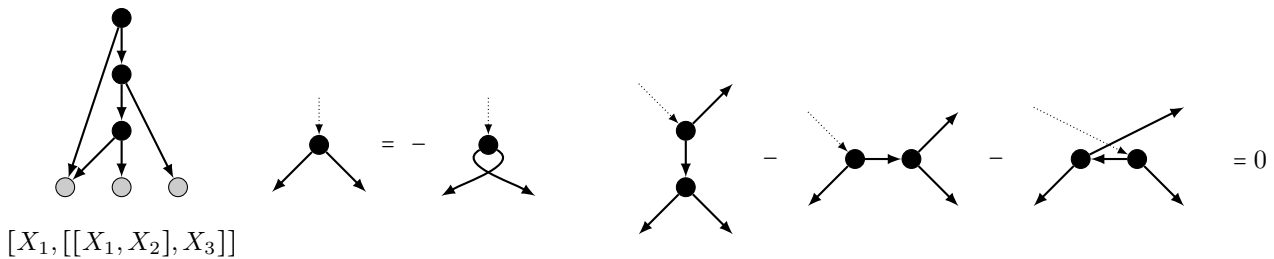
\begin{figure}
\centering
\begin{tikzpicture}[>=latex]
\tikzstyle{k-int}=[draw,fill=gray!40,circle,inner sep=0pt,minimum size=2.5mm]
\tikzstyle{n-int}=[draw,fill=black,circle,inner sep=0pt,minimum size=2.5mm]
\tikzstyle{ext}=[draw,fill=white,circle,inner sep=0pt,minimum size=2.5mm]
\tikzstyle{ext0}=[draw,fill=white,circle,inner sep=0pt,minimum size=4mm]
\tikzstyle{spec}=[draw,rectangle,inner sep=0pt,minimum size=4mm]

\begin{scope}[scale=0.75]
\node[k-int] (1) at (-1,0) {};
\node[k-int] (2) at (0,0) {};
\node[k-int] (3) at (1,0) {};
\node[n-int] (1') at (0,1) {};
\node[n-int] (2') at (0,2) {};
\node[n-int] (3') at (0,3) {};
\draw[thick,->] (1') to (1);
\draw[thick,->] (1') to (2);
\draw[thick,->] (2') to (1');
\draw[thick,->] (2') to (3);
\draw[thick,->] (3') to (1);
\draw[thick,->] (3') to (2');
\node at (0,-1) {$[X_1,[[X_1,X_2],X_3]]$};
\end{scope}

\begin{scope}[scale=0.75,shift={(2.5,-0.25)}]
\begin{scope}
\node[n-int] (1) at (1,1) {};
\node (1') at (0,0) {};
\node (2') at (2,0) {};
\node (out) at (1,2) {};
\draw[densely dotted,->] (out) to (1);
\draw[thick,->] (1) to (1');
\draw[thick,->] (1) to (2');
\end{scope}
\node at (2.5,1) {$=\ -$};
\begin{scope}[shift={(3,0)}]
\node[n-int] (1) at (1,1) {};
\node (1') at (0,0) {};
\node (2') at (2,0) {};
\node (out) at (1,2) {};
\draw[densely dotted,->] (out) to (1);
\draw[thick,->] (1) [out=-30,in=30] to (1');
\draw[thick,->] (1) [out=210,in=150] to (2');
\end{scope}
\end{scope}

\begin{scope}[scale=0.75,shift={(10,0)}]
\begin{scope}
\node (out) at (-1,2) {};
\node (1) at (-1,-1) {};
\node[n-int] (1') at (0,0) {};
\node[n-int] (2') at (0,1) {};
\node (2) at (1,-1) {};
\node (3) at (1,2) {};
\draw[densely dotted,->] (out) to (2');
\draw[thick,->] (2') to (1');
\draw[thick,->] (2') to (3);
\draw[thick,->] (1') to (1);
\draw[thick,->] (1') to (2);
\end{scope}
\node at (1.5,0.5) {$-$};
\begin{scope}[shift={(3,0)}]
\node (out) at (-1,1.5) {};
\node (1) at (-1,-.5) {};
\node[n-int] (1') at (0,.5) {};
\node[n-int] (2') at (1,.5) {};
\node (3) at (2,1.5) {};
\node (2) at (2,-.5) {};
\draw[densely dotted,->] (out) to (1');
\draw[thick,->] (1') to (2');
\draw[thick,->] (1') to (1);
\draw[thick,->] (2') to (2);
\draw[thick,->] (2') to (3);
\end{scope}
\node at (5.5,0.5) {$-$};
\begin{scope}[shift={(7,0)}]
\node (out) at (-1,1.5) {};
\node (1) at (-1,-.5) {};
\node[n-int] (1') at (0,.5) {};
\node[n-int] (2') at (1,.5) {};
\node (3) at (2,1.5) {};
\node (2) at (2,-.5) {};
\draw[densely dotted,->] (out) to (2');
\draw[thick,->] (2') to (1');
\draw[thick,->] (2') to (2);
\draw[thick,->] (1') to (1);
\draw[thick,->] (1') to (3);
\end{scope}
\node at (10,0.5) {$=0$};
\end{scope}

\end{tikzpicture}
\caption{\label{fig-IHX} The first picture illustrates the rooted tree representing the Lie monomial $[X_1,[[X_1,X_2],X_3]]$; the shaded gray vertices are external, while the black vertices are internal.
The second picture illustrates the anti-symmetry relation and the third one depicts the IHX relation.
The dotted incoming edge means that the corresponding directed edge may or may not be actually be present; if not, the bivalent edge represents the root of a Lie tree.}
\end{figure}
In fact, the anti-symmetry relation may be discarded by choosing a total order on the set of edges of a Lie tree, which is what we always do.

From the previous discussion, elements of $\tder_k$ are in one-to-one correspondence with $k$-tuples of Lie trees with $k$ external vertices {\em modulo} the IHX relation.
There is a more convenient way to encode such $k$-tuples into a linear combination of directed graphs with $k$ external vertices and an arbitrary number of internal, trivalent vertices.

Namely, let us consider an element $u=(u_1,\dots,u_k)$ of $\tder_k$: for $1\leq i\leq k$, let us consider the (possibly infinite) linear combination of Lie trees corresponding to $u_i$.
For each Lie tree corresponding to a summand of $u_i$, we draw an additional directed edge from the external $i$-th vertex to the root of the Lie tree: this way, out of a Lie tree is produced a directed graph with $k$ external vertices and all internal trivalent vertices.


This induces an identification between elements of $\sder_k$ and the (graded) vector space spanned by internally connected, internally trivalent un-directed graphs with $k$ external vertices and an arbitrary number of internal trivalent vertices {\em modulo} the IHX relation.
Roughly, given an un-directed graph $\Gamma$ with $k$ external vertices and an arbitrary number of internal trivalent vertices, we may construct, for $i=1,\dots,k$, $m_i$ directed graphs $\Gamma_{i,l}$, $l=1,\dots,m_i$, with $k$ external vertices and the same number of internal trivalent vertices as $\Gamma$, where $m_i$ is the number of edges connected to the $i$-th external vertex.
Namely, for $l=1,\dots,m_i$, we choose the direction of the $l$-th edge in such a way that the edge departs from the $i$-th external vertex: then, the directions of the remaining edges is automatically determined by the fact that all internal vertices of $\Gamma$ are trivalent and that every internal vertex has one ingoing and two outgoing edges.
The Jacobi identity and an induction argument in $\Lie_k$ imply that the unique element of $\tder_k$ obtained this way out of an un-directed graph as before belongs to $\sder_k$.

Let us discuss the graphical interpretation of the Lie bracket on $\tder_k$ (and also on $\sder_k$).

The prescription for the Lie bracket on the combinatorial version of $\tder_k$ can be deduced quite easily from the expression $u(v_i)-v(u_i)-[u_i,v_i]$, $i=1,\dots,k$, for the $i$-th component of $[u,v]$.
Namely, the expression $u(v_i)$ can be easily computed by recalling that $u$ is a tangential derivation, hence it acts on the Lie monomial $v_i$ by Leibniz' rule with respect to the Lie bracket, and in each summand, where $u$ acts on $X_j$, $j=1,\dots,d$, it replaces $X_j$ by $[u_j,X_j]$.
Similar arguments apply for $v(u_i)$ by interchanging the {\em r\^oles} of $u$ and $v$.

Therefore, let us consider two elements $\Gamma_i$, $i=1,2$, of the combinatorial version of $\tder_k$: then, the bracket $\left[\Gamma_1,\Gamma_2\right]$ is defined by the following prescriptions:
\begin{itemize}
\item[$i)$] $\Gamma_1$ and $\Gamma_2$ are glued together at their external vertices, so as to obtain an internally non-connected graph $\Gamma_1\cdot\Gamma_2$ with two internally connected, internally trivalent components and exactly two directed edges from two external vertices to the roots of $\Gamma_1$, $\Gamma_2$ (observe that these two external vertices may coincide); 
\item[$ii)$] we sum over all possible ways of splitting the external vertices of $\Gamma_1\cdot\Gamma_2$ into an external and an internal vertex by inserting a directed edge between them and reconnecting the remaining directed edges in all possible ways;
\item[$iii)$] from the previously obtained linear combinations of graphs, we discard all graphs which are not internally connected, internally trivalent, and the result is $\left[\Gamma_1,\Gamma_2\right]$.
\end{itemize}
Because of the previous prescriptions, it is clear that the only external vertices, whose splitting according to $ii)$ produces possibly non-trivial internally connected, internally trivalent graphs, are the external vertices from which departs a directed edge to the roots of $\Gamma_i$, $i=1,2$.
Furthermore, assume that the unique external vertex of $\Gamma_1$, from which departs the directed edge to its root, differs from the unique external vertex of $\Gamma_2$, from which departs the directed edge to its root, and assume none of these two external vertices is the endpoint of a directed edge from the other Lie tree, then $\left[\Gamma_1,\Gamma_2\right]=0$.

We illustrate in Figure~\ref{fig-tder-bracket} the Lie bracket $[\Gamma_1,\Gamma_2]$ of two elements $\Gamma_i$, $i=1,2$, of the combinatorial version of $\tder_4$: it is not difficult to write down the corresponding tangential derivations of $\Lie_4$, compute explicitly their Lie bracket and identify it with the directed tree on the right-hand side.
Later on, we will encounter a cohomological interpretation of the Lie bracket on the combinatorial version of $\tder_k$.
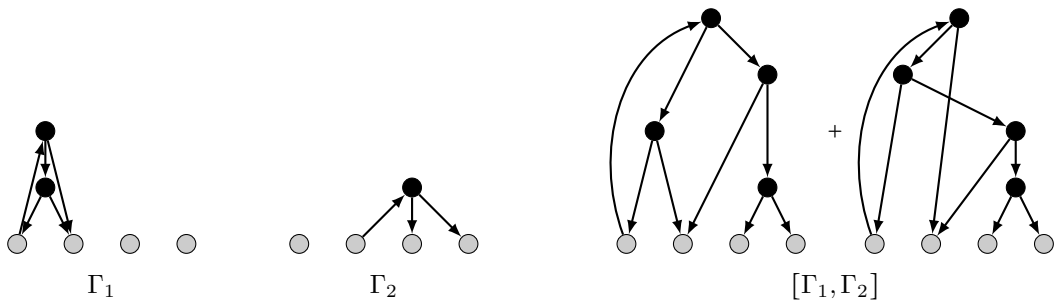
\begin{figure}
\centering
\begin{tikzpicture}[>=latex]

\tikzstyle{k-int}=[draw,fill=gray!40,circle,inner sep=0pt,minimum size=2.5mm]
\tikzstyle{n-int}=[draw,fill=black,circle,inner sep=0pt,minimum size=2.5mm]
\tikzstyle{ext}=[draw,fill=white,circle,inner sep=0pt,minimum size=2.5mm]
\tikzstyle{ext0}=[draw,fill=white,circle,inner sep=0pt,minimum size=4mm]
\tikzstyle{spec}=[draw,rectangle,inner sep=0pt,minimum size=4mm]

\begin{scope}[scale=0.75]
\node[k-int] (1) at (-1.5,0) {};
\node[k-int] (2) at (-.5,0) {};
\node[k-int] (3) at (.5,0) {};
\node[k-int] (4) at (1.5,0) {};
\node[n-int] (v1) at (-1,1) {};
\node[n-int] (v2) at (-1,2) {};
\draw[thick,->] (v2) to (v1) {};
\draw[thick,->] (v2) to (2) {};
\draw[thick,->] (v1) to (1) {};
\draw[thick,->] (v1) to (2) {};
\draw[thick,->] (1) to (v2) {};
\node at (0,-.75) {$\Gamma_1$};
\end{scope}

\begin{scope}[scale=0.75,shift={(5,0)}]
\node[k-int] (1) at (-1.5,0) {};
\node[k-int] (2) at (-.5,0) {};
\node[k-int] (3) at (.5,0) {};
\node[k-int] (4) at (1.5,0) {};
\node[n-int] (v1) at (.5,1) {};
\draw[thick,->] (v1) to (3) {};
\draw[thick,->] (v1) to (4) {};
\draw[thick,->] (2) to (v1) {};
\node at (0,-.75) {$\Gamma_2$};
\end{scope}

\begin{scope}[scale=0.75,shift={(13,0)}]
\begin{scope}[shift={(-2.2,0)}]
\node[k-int] (1) at (-1.5,0) {};
\node[k-int] (2) at (-.5,0) {};
\node[k-int] (3) at (.5,0) {};
\node[k-int] (4) at (1.5,0) {};
\node[n-int] (v1) at (1,1) {};
\node[n-int] (v2) at (-1,2) {};
\node[n-int] (v3) at (1,3) {};
\node[n-int] (v4) at (0,4) {};
\draw[thick,->] (v4) to (v2) {};
\draw[thick,->] (v4) to (v3) {};
\draw[thick,->] (v3) to (2) {};
\draw[thick,->] (v3) to (v1) {};
\draw[thick,->] (v2) to (1) {};
\draw[thick,->] (v2) to (2) {};
\draw[thick,->] (v1) to (3) {};
\draw[thick,->] (v1) to (4) {};
\draw[thick,->] (1) [out=110,in=200] to (v4) {};
\end{scope}
\node at (0,2) {$+$};
\node at (0,-.75) {$\displaystyle \left[\Gamma_1,\Gamma_2\right]$};
\begin{scope}[shift={(2.2,0)}]
\node[k-int] (1) at (-1.5,0) {};
\node[k-int] (2) at (-.5,0) {};
\node[k-int] (3) at (.5,0) {};
\node[k-int] (4) at (1.5,0) {};
\node[n-int] (v1) at (1,1) {};
\node[n-int] (v2) at (1,2) {};
\node[n-int] (v3) at (-1,3) {};
\node[n-int] (v4) at (0,4) {};
\draw[thick,->] (v4) to (v3);
\draw[thick,->] (v4) to (2);
\draw[thick,->] (v3) to (1);
\draw[thick,->] (v3) to (v2);
\draw[thick,->] (v2) to (2);
\draw[thick,->] (v2) to (v1);
\draw[thick,->] (v1) to (3);
\draw[thick,->] (v1) to (4);
\draw[thick,->] (1) [out=110,in=200] to (v4);
\end{scope}
\end{scope}

\end{tikzpicture}
\caption{\label{fig-tder-bracket} The Lie bracket of two elements $\Gamma_i$, $i=1,2$, of the combinatorial version of $\tder_4$.}
\end{figure}

Finally, let us recall the definition of simplicial and coproduct maps on $\tder$ and $\sder$; we follow closely the description in~\cite{AT-2}*{Subsection 3.2}, though we do not treat here the topic in all its generality as in {\em loc.~cit.}.
Let us consider a general element $u$ of $\tder_k$, which we write in a unique way as a $k$-tuple $(u_1,\dots,u_k)$ in $\Lie_k$.
We define $u^{1,\dots,k}$, resp.\ $u^{2,\dots,k+1}$, in $\tder_{k+1}$ {\em via}
\begin{equation}\label{eq-cob-tder-1}
\begin{aligned}
u^{1,\dots,k}&=\left(u_1(X_1,\dots,X_k),\dots,u_k(X_1,\dots,X_k),0\right),\ & u^{2,\dots,k+1}&=\left(0,u_1(X_2,\dots,X_{k+1}),\dots,u_k(X_2,\dots,X_{k+1})\right).
\end{aligned}
\end{equation}
It is clear that, if $u$ belongs to $\sder_k$, both $u^{1,\dots,k}$ and $u^{2,\dots,k+1}$ belong to $\sder_{k+1}$.

On the other hand, for $i=1,\dots,k$, we define $u^{1,\dots,ii+1,\dots,k+1}$ in $\tder_{k+1}$ {\em via}
\begin{equation}\label{eq-cob-tder-2}
\begin{aligned}
u^{1,\dots,i i+1,\dots,k+1}=&\left(u_1(X_1,\dots,X_i+X_{i+1},\dots,X_{k+1}),\dots,u_i(X_1,\dots,X_i+X_{i+1},\dots,X_{k+1}),\right.\\
&\phantom{=}\left.u_i(X_1,\dots,X_i+X_{i+1},\dots,X_{k+1}),\dots,u_k(X_1,\dots,X_i+X_{i+1},\dots,X_{k+1})\right).
\end{aligned}
\end{equation}
Again, an easy computation shows that, if $u$ belongs to $\sder_k$, $u^{1,\dots,ii+1,\dots,k+1}$ belongs to $\sder_{k+1}$.
A bit more involved is the proof that~\eqref{eq-cob-tder-1} and~\eqref{eq-cob-tder-2} preserve the Lie algebra structure on $\tder$ (thus also on $\sder$), see~\cite{AT-2}*{Subsection 3.2}.

Since we have preferred to consider a combinatorial-graphical description of $\tder$ and $\sder$, let us characterize the simplicial and coproduct maps specified by~\eqref{eq-cob-tder-1} and~\eqref{eq-cob-tder-2} in graphical terms.
The simplicial maps in~\eqref{eq-cob-tder-1} are described graphically by simply adding a $0$-valent external vertex on the left and on the right respectively of a graph either in $\tder_k$ or $\sder_k$: more conceptually, 
\[
u^{1,\dots,k}=\grcup\circ_1 u,\ \text{resp.}\ u^{2,\dots,k+1}=\grcup\circ_2 u,\ u\in \tder_k\ \text{or}\ \sder_k.
\]
The coproduct maps in~\eqref{eq-cob-tder-1} are also similarly described graphically
\[
u^{1,\dots,ii+1,\dots,k+1}=u\circ_i \grcup,\ i=1,\dots,k,\ u\in \tder_k\ \text{or}\ \sder_k.
\]
Observe that the differential $d$ in~\cite{AT-2}*{Subsection 3.3} admits the graphical description $d=\left[\grcup,\bullet\right]$. 

The group $\mathfrak S_k$ of permutations of $k$ elements acts from the right in a natural way on $\Lie_k$ {\em via} 
\[
a^\sigma(X_1,\dots,X_k)=a(X_{\sigma(1)},\dots,X_{\sigma(n)}),\ a\in\Lie_k.
\]
As a consequence, there is a right $\mathfrak S_k$-action on $\tder_k$, which descends to $\sder_k$ and to $\mathfrak t_k$, explicitly given by 
\[
u^\sigma(a)=(u(a^\sigma))^{\sigma^{-1}},\ u\in\tder_k,\ a\in Lie_k.
\] 
In this way, $\mathfrak S_k$ acts on $\tder_k$, $\sder_k$ and $\mathfrak t_k$ by Lie algebra automorphisms, thus also on $\TAut_k$, $\SAut_k$ and $\mathsf T_k$ by group automorphisms.

As was done in the preceding Sections, we will adopt the graphical interpretation of the Lie algebras $\tder_k$, $\sder_k$ and $\mathfrak t_k$: this allows to interpret the collections $\{\tder_k\}_k$, $\{\sder_k\}_k$ and $\{\mathfrak t_k\}_k$ as operads of Lie algebras, which we denote simply by $\tder$, $\sder$ and $\mathfrak t$ respectively.

Finally, for a given $k\geq 1$, we denote by $\mathsf{TAut}_k$, $\mathsf{SAut}_k$ and $\mathsf T_k$ respectively the pro-unipotent groups which integrate the pro-nilpotent Lie algebras $\tder_k$, $\sder_k$ and $\mathfrak t_k$ respectively.

\subsection{The Grothendieck--Teichm\"uller Lie algebra \texorpdfstring{$\grt_1$}{grt1}}\label{ss-7-2}
Let us introduce the Grothendieck--Teichm\"uller Lie algebra $\grt_1$ over $\mathbb K$. We follow closely~\cite{AT-2}*{Subsection 4.2}.

The vector space underlying $\grt_1$ is the space of elements $\psi$ in $\Lie_2$, which obey the following three properties:
\begin{align}
\label{eq-anti}\psi(X_1,X_2)&=-\psi(X_2,X_1),\\
\label{eq-hexa}\psi(x,y)+\psi(y,z)+\psi(z,x)&=0,\quad \quad \text{where } x+y+z=0,\\
\label{eq-penta}\psi(t_{12},t_{23}+t_{34})+\psi(t_{13}+t_{23},t_{34})&=\psi(t_{23},t_{34})+\psi(t_{12}+t_{13},t_{24}+t_{34})+\psi(t_{12},t_{23}).
\end{align}
To define the Lie bracket, one understands $\psi$ as the element $(0,\psi)\in \tder_2$.
The Lie bracket on $\tder_2$ then induces the Lie bracket on $\grt_1$. 
The explicit expression for the Lie bracket on $\grt_1$ (the Ihara bracket) is given by
\[
[\psi_1,\psi_2]_\mathrm{Ih}=(0,\psi_1)(\psi_2)-(0,\psi_2)(\psi_1)+[\psi_1,\psi_2],
\]
where $(0,\psi_i)$, $i=1,2$, is the tangential derivative associated with $\psi_i$, and the last term in the right-hand side denotes the Lie bracket in $\Lie_2$.

Identity~\eqref{eq-anti} is the (infinitesimal version of the) antisymmetry relation; Identity~\eqref{eq-hexa} is the (infinitesimal version of the) hexagon relation, and finally Identity~\eqref{eq-penta} is the (infinitesimal version of the) pentagon relation.
It has been proved recently in~\cite{Fur} that in fact Identity~\eqref{eq-penta} together with the assumption that $\psi(x_1,x_2)$ does not contain a term of the form $[x_1,x_2]$ implies Identities~\eqref{eq-hexa} and~\eqref{eq-anti}; furthermore, Identity~\eqref{eq-penta} implies that elements of $\grt_1$ start with Lie monomials of degree at least $2$. 
We also observe that, actually, $\grt_1$ is a Lie subalgebra of $\sder_2$, for more detail we refer to~\cite{AT-2}*{Theorem 4.1}: in particular, elements of $\grt_1$ can be represented as (possibly infinite) linear combinations of un-directed graphs with two external vertices and an arbitrary number of internal trivalent vertices.

We have already mentioned, without proof, that the cohomology of Kontsevich's graph complex $\GC$ is concentrated in non-negative degrees, and that its $0$-th cohomology coincides with $\grt_1$: we are going to describe now explicitly the construction of a natural map from $0$-th degree cocycles of $\GC$ to $\sder_2$, whose image turns out to be precisely $\grt_1$.
In fact, for later computational reasons, we will describe a slight variant of the map between $\GC$ and $\sder_2$ presented in explicit terms in~\cite{Will}*{Section 6}.

\subsection{A map between \texorpdfstring{$\GC$ and $\sder_2$}{GC and sder2}}\label{ss-7-3}
In \cite{Will}*{Section 6}, the second author introduced a map $\phi:\GC_\mathrm{cl}\to \grt \subset \sder_2$ from the space $\GC_\mathrm{cl}$ of closed elements of degree $0$ of the graph complex $\GC$ into $\grt_1$, hence into $\sder_2$.

Let us recall the explicit construction of the map $\phi$, at least on the $1$-vertex irreducible subspace of $\GC_\mathrm{cl}$.
A graph $\Gamma$ is said to be $1$-vertex irreducible, if there is no vertex $v$ of $\Gamma$ such that $\Gamma\smallsetminus\{v\}$ splits into $k\geq 2$ connected components.

Let $\gamma$ in $\GC$ be a cocycle of degree $0$, which we assume to be given by $1$-vertex irreducible graphs.
Further, let $\gamma_1\in \ICG(1)$ be the element obtained by marking the vertex $1$ as external.
$\ICG(k)$, for $k\geq 1$, denotes the space of internally connected graphs with $k$ external vertices, see~\cite{Will}*{Subsection 2.1}.
Then, the element 
\[
Y = \gamma_1 \circ \grcup - \grcup\circ \gamma_1
\]
is closed, hence exact in $\ICG(2)$ by the results of~\cite{Will}*{Sections 3,5}: thus, we may write $Y=\delta(\widetilde Y)$ for some $\widetilde Y$ in $\ICG(2)$. Here, $\widetilde Y$ is defined up to closed and hence exact elements.
 
There is a projection $\pi$ from $\ICG(2)$ to $\sder_2$, defined by forgetting all non-internally trivalent graphs and modding out by the IHX relations. Note in particular that $\pi\circ \delta=0$. 

We define 
\[
\phi(\gamma) = \pi(\tilde Y).
\]
This is well defined because $\pi\circ \delta=0$.
We quote without proof the following proposition from \cite{Will}*{Sections 5,6}.
\begin{Prop}
\label{p-imphiingrt}
The image of the map $\phi$ above is $\grt_1\subset \sder_2\subset \tder_2$.
\end{Prop}
The goal of this subsection is to simplify this map a bit. 

Let us define a map (of graded vector spaces)
\begin{align*}
\psi:\GC &\to \Graphs(2),\\
\gamma &\mapsto \gamma_{12},
\end{align*}
where the element $\gamma_{12}$ is zero if the vertices $1$ and $2$ are not connected by an edge, and is the graph obtained by marking the vertices $1$ and $2$ as external and then deleting the edge between these two vertices otherwise.
To fix signs, we assume that the edge between vertices $1$ and $2$ is the first one with respect to the total ordering on edges.
Here, $\Graphs(k)$, $k\geq 1$, denotes the graded vector space of graphs with $k$ external vertices.
\begin{Lem}\label{l-psiprop}
\begin{equation}
\label{eq-psiprop}
\delta(\psi(\gamma)) - \psi(\delta(\gamma))=\gamma_1 \circ \grcup - \grcup\circ \gamma_1.
\end{equation}
\end{Lem}
\begin{proof}
To prove the statement, it suffices to unravel the definition of the differential $\delta$ on $\Graphs(k)$ and $\GC$: observe that $\delta$ acts only on internal vertices of elements of $\Graphs(k)$. 

First, $\delta(\gamma)$ has the form 
\[
\delta(\gamma)= \SNbr\circ\gamma-(-1)^{|\gamma|}\gamma\circ \SNbr.
\]
On the right-hand side, the newly inserted edge may either $i)$ become the edge $(1,2)$ or $ii)$ not: accordingly, one obtains two types of terms when applying $\psi$.
One checks that the terms of type $ii)$ are precisely those appearing in $\delta(\psi(\gamma))$, as $\delta$ does not act on the two external vertices, while the terms of type $i)$ are exactly the ones in the right-hand side of~\eqref{eq-psiprop}. 
\end{proof}
Let $Y$ be as above, for $\gamma$ in $\GC_\mathrm{cl}$ as above. 
Then, by means of Lemma~\ref{l-psiprop} and if we assume that $\psi(\gamma)\in\ICG(2)[1]$, we may take 
\[
\widetilde Y = \psi(\gamma)
\]
and clearly
\begin{align*}
\delta(\widetilde Y) 
&= \gamma_1 \circ \grcup - \grcup\circ \gamma_1 + \psi(\delta(\gamma))=Y+0=Y.
\end{align*} 
Let us summarize this discussion.
\begin{Lem}\label{l-altphi}
Suppose that $\gamma\in\GC_\mathrm{cl}^{1\mathrm{vi}}$ is such that $\psi(\gamma)$ contains only graphs with one internally connected component. Then $\phi(\gamma)=\pi(\psi(\gamma))\in \sder_2$
\end{Lem}

\subsection{The image \texorpdfstring{$\tau^t$ of $x^t$ in $\sder_2$}{taut of xt in sder2 }}\label{ss-7-4}
Let us again consider the family of graph cocycles $x^t$ discussed in Section~\ref{s-6}.
The goal of the present Subsection is to produce an explicit integral formula for its image $\phi(x^t)$ under the map $\phi$ from Subsection~\ref{ss-6-2}.

First recall from~\eqref{equ:xtdef} that 
\[
x^t=\sum_{\Gamma} c_\Gamma^t\ \Gamma
\]
with the weights defined in~\eqref{equ:xtweightdef} as
\[
c_\Gamma^t
=
\int_{C_n/S^1}
\tilde \beta^t_\Gamma
=
\int_{\Conf_{n-2}(\mathbb C\smallsetminus\{0,1\})}
\tilde \beta^t_\Gamma.
\]
For the last equality we used the isomorphism $\Conf_{n-2}(\mathbb C\smallsetminus\{0,1\}) \to C_n/S^1$ given by fixing the first point of the configuration at $z_1=0$ and the second at $z_2=1$.
Suppose there is an edge between vertices $1$ and $2$. Then note that of the terms in \eqref{equ:tildebetadef} only those will contribute for which the $e$ in the first sum is the edge $(1,2)$.

\begin{Lem}\label{l-1vi}
If $\Gamma$ is not $1$-vertex irreducible, then $c^t_\Gamma=0$. 
In particular, $x^t\in \GC_{cl}^{1vi}$ for all $t$.
\end{Lem}
\begin{proof}
Recall that, by means of the second identity of \eqref{equ:xtweightdef}, it will be mostly sufficient to restrict the analysis to $x^{\frac{1}2}$.

Let $v$ be a vertex of $\Gamma$ such that $\Gamma\smallsetminus\{v\}$ splits into $k\geq 2$ connected components. 
%
Without loss of generality we may assume that $v=1$, that the vertex 2 lies in the first connected component and that there is an edge connecting $1$ and $2$.

Then, the integrand in 
\[
 \int_{\Conf_{n-2}(\mathbb C\smallsetminus\{0,1\})}
\tilde \beta^{\frac{1}2}_\Gamma
\]
can be written into a sum of products of forms according to the connected components of $\Gamma\smallsetminus\{v\}$ and to the edge, different from the one connecting $1$ and $2$, with which is associated a function $\log|z-w|$.

Fubini's Theorem implies that the corresponding integral $c_\Gamma^t$ can be written as a sum of products of integrals over $\mathrm{Conf}_{n_i}(\mathbb C\smallsetminus\{0,1\})$, $i=1,\dots,k$, $k$ being the number of connected components of $\Gamma\smallsetminus\{v\}$ and $n_i$ the number of vertices of the $i$-th connected component other than 1 and 2.
In fact, for each term in the sum \eqref{equ:tildebetadef} there is at least one connected component such that $e'$ does not lie in this component. But that term produces a factor zero upon integration by the Kontsevich Vanishing Lemma~\cite{K}*{Lemma 6.6}.
\end{proof}

\begin{Lem}\label{lem:intconnected}
 If a graph $\Gamma$ is such that $\psi(\Gamma)$ contains more than one internally connected component, then  $c^t_\Gamma=0$ for all $t$. 
\end{Lem}
\begin{proof}
 Suppose $\Gamma$ has $n$ vertices.
Recall that the weight $c^t_\Gamma$ is defined by an integral over $\Conf_{n-2}(\bbC\setminus\{0,1\})$ of the form \eqref{equ:tildebetadef}. Here the singled out points $0$ and $1$ correspond to vertices 1 and 2 in the contributing graphs. If vertices $1$ and $2$ are not connected by an edge, there is no contribution to $\psi(\Gamma)$. Next note that a term in the sum \eqref{equ:tildebetadef} can only contribute non-trivially if $e'$ is the edge between vertices $1$ and $2$, for otherwise the form is zero when restricted to $z_1=0$, $z_2=1$. Suppose the graph decomposes into $k$ connected components after deleting vertices 1, 2 and the edge between them. We want to show for $k\neq 1$ the integral vanishes. 
Suppose further the $k$ components have $n_1,\dots, n_k$ vertices and $e_1,\dots, e_k$ edges. If for some $j$ we have $v_j-2n_j\notin \{0,1\}$ the integral vanishes by degree reasons. If $v_j-2n_j=0$ for some $j$ the integral also vanishes by M. Kontsevich's vanishing Lemma. But since $\sum_{j=1}^k (v_j-2n_j)=1$, the only case for which $v_j-2n_j=1$ for all $j$ is that $k=1$. Hence the Lemma follows.
\end{proof}

Let us next compute $\phi(x^t)$, using the alternative description of $\phi$ from Lemma~\ref{l-altphi}, which is applicable due to Lemma~\ref{lem:intconnected}.

Since the projection $\pi$ sends to zero all graphs with non-trivalent internal vertices, a graph $\Gamma$ appearing in $x^t$ can only contribute if it has at most two vertices of valence $\geq 4$, which correspond to the vertices $1$ and $2$. 
Fix such a graph $\Gamma$. 
Thus, the vertices in $\Gamma$ may be numbered w.l.o.g.\ in such a way that all vertices except possibly vertices $1$ and $2$ have valence exactly $3$.
Furthermore we may assume that the vertices $1$ and $2$ are connected by an edge, otherwise the graph would not contribute.

The contribution of $\Gamma$ to $\phi(x^t)$ is then an element $\Gamma_{12}$, obtained by making the vertices $1$ and $2$ external and deleting the edge between them.
Hence we obtain
\[
\tau^t=\phi(x^t) 
=\sum_{\Gamma'} c_{\Gamma'}^t\ \pi(\Gamma_{12}')
\]
where now the sum is only over graphs $\Gamma'$ such that $i)$ all vertices except possibly $1$ and $2$ are trivalent and $ii)$ the vertices $1$ and $2$ are connected by an edge.

For such a graph $\Gamma'$, let us examine more closely the integral $c_{\Gamma'}^t$.
Using the assumptions on $\Gamma'$, many terms of the integrand \eqref{equ:tildebetadef} will not contribute.
Concretely, if one of the endpoints of the edge $e$ (as in \eqref{equ:tildebetadef}) is not among the vertices 1 and 2, then $\Gamma'\smallsetminus \{e\}$ contains at least one vertex of valence two, whence the integral vanishes. 
Therefore, the only contributing term in the sum over $e$ in \eqref{equ:tildebetadef} is the one for which $e$ is the edge connecting the vertices $1$ and $2$.

Summarizing, we obtain the following identity for $c^t_{\Gamma'}$:
\begin{equation}\label{eq-ctgammasimpl}
c^t_{\Gamma'}=-[4t(1-t)]^{n-2}
\sum_{e\neq (1,2)}
(-1)^{e-1}
\int_{\mathrm{Conf}_{n-2}(\mathbb C\smallsetminus\{0,1\})}
 \beta_e \omega^{\frac 1 2}_{\Gamma' \smallsetminus \{e,(1,2)\}}.
\end{equation}

\section{A family of AT connections and \texorpdfstring{$\tder$}{tder}-associators}\label{s-8}
The present Section is devoted to the construction of a family of Drinfel{\cprime}d associators which will be central in settling P.~Etingof's conjecture.

The main idea of the construction of the aforementioned family of Drinfel{\cprime}d associators is to construct a family of flat connections $\nabla^t_k$ 
on the trivial principal $\SAut_k$-bundle over $\underline C_k=\mathrm{Conf}_k/\bbC$; moreover, we prove that these connections (for all $t$) are gauge equivalent. 
We further establish an explicit connection between the family of gauge transformations for $\nabla^t$ on $\underline C_k$ and the family of graph cocycles $x^t$.

The family of associators we are interested in is defined as (a suitable regularization of) the parallel transport with respect to $\nabla^t_3$ on $\underline C_3$ along a path which connects two boundary configurations in a suitable compactification. 

\subsection{More configuration spaces and compactifications}\label{ss-8-0}
In section \ref{s-3} we reviewed the spaces $\overline C_n$ obtained by a Fulton-MacPherson-Axelrod-Singer type compactification of the configuration spaces $\Conf_n(\bbC) / \bbR^+ \ltimes \bbC$. In this section we will mostly use the configuration spaces
$\underline C_n=\underline \Conf_n(\bbC) / \bbC$, i.~e., we do not take a quotient by rescalings. 

There is a natural projection $\pi_{k,n}:\underline C_{k+n}\to \underline C_k$ by forgetting the last $n$ points.
It will be useful to consider a fiberwise compactification $\overline C_{k,n}^f\to \underline C_k$ of the bundle $\underline C_{k+n}\to \underline C_k$ obtained as a pull-back
\[
 \begin{tikzpicture}
  \matrix(m)[diagram]{\overline C_{k,n}^f & \overline C_{k+n} \\ \underline C_k & \overline C_k\\ };
  \draw[-latex] (m-1-1) edge (m-1-2) edge (m-2-1) (m-2-1) edge (m-2-2) (m-1-2) edge (m-2-2);
 \end{tikzpicture}
\]
The fibers of $\overline C_{k+n}^f\to \underline C_k$ are compact smooth manifolds with corners. 
Furthermore $\overline C_{k+n}^f$ fits into the framework of the regularized Stokes' Theorem of \cite{ARTW}, i.~e., Theorem \ref{thm:regstokes}, as is explained in more detail in Appendix \ref{app:stokesbundle}.

\subsection{A family of Alekseev-Torossian connections}\label{ss-8-1}

Let us consider a graph $\Gamma$ in $\sder_k$ with $n$ internal vertices and hence $2n+1$ edges. Assuming for now that the following integrals exists, we associate to $\Gamma$ a differential one-form $\vartheta_\Gamma^t$ and a function $\widetilde\vartheta^t_\Gamma$ on $\underline C_k$ {\em via}
\begin{align}
\label{eq-weight-AT-1}\vartheta_{\Gamma}^t&=\int_{\underline C_{k+n}/\underline C_k}\theta^t_\Gamma,\\
\label{eq-weight-AT-0}\widetilde\vartheta_{\Gamma}^t&=\int_{\underline C_{k+n}/\underline C_k}\widetilde\theta^t_\Gamma
\end{align}
where the notation on the right-hand side means integration along the fiber of the projection $\pi_{k,n}:\underline C_{k+n}\to \underline C_k$ and the integrands are defined as
\begin{align}
\label{eq-weight-AT-1-integrand}\theta_{\Gamma}^t&= \prod_{e\in E(\Gamma)} \frac 1 {2\pi i} \left( (1-t) d\log(z_{s(e)}-z_{t(e)}) + t\, d\log(\bar z_{s(e)}-\bar z_{t(e)})\right),\\
\label{eq-weight-AT-0-integrand}\widetilde\theta_{\Gamma}^t&=\sum_{e\in E(\Gamma)} (-1)^{e-1} \frac i \pi 
\log|z_{s(e)}- z_{t(e)}|
\prod_{e'\neq e\in E(\Gamma)} \frac 1 {2\pi i} \left( (1-t) d\log(z_{s(e')}-z_{t(e')}) + t\, d\log(\bar z_{s(e')}-\bar z_{t(e')})\right).
\end{align}
We furthermore set
\begin{align}
\label{eq-AT-bound-1}\omega^t_{\mathrm{AT},k}&=\sum_{\Gamma
}\vartheta^t_\Gamma\ \Gamma,\\
\label{eq-AT-bound-0}a^t_k&=\sum_{\Gamma
}\widetilde\vartheta^t_\Gamma\ \Gamma.
\end{align}
Hence $\nabla^t_k=d-\omega^t_{\mathrm{AT},k}$ defines a connection on the trivial principal $\SAut_k$-bundle over $\underline C_k$, while $a^t_k$ defines an $\sder_k$-valued function on $\underline C_k$. 
We will see below that the connections $\nabla^t_k$ (for varying $t$) are all gauge equivalent, and are transformed into each other by (infinitesimal) gauge transformations determined by $a^t_k$.
Let us however begin by showing that the integrals \eqref{eq-weight-AT-1} and \eqref{eq-weight-AT-0} indeed converge.

\begin{Lem}\label{l-conv-fib}
For $\Gamma$ a graph in $\sder_k$ the $1$-form $\vartheta^t_\Gamma$ and the $0$-form $\widetilde\vartheta^t_\Gamma$ are well-defined ({\em i.e.} the integral converges) and smooth on $\underline C_k$, for all $t$.
\end{Lem}

We will use the following Lemma. 
\begin{Lem}\label{lem:a12exists}
 Let $\Gamma$ be a graph in $\sder_k$. Then the fiberwise top degree part of the form $\widetilde\theta^{\frac 1 2}_\Gamma$ on $\underline C_{k+n}$ extends to the compactification $\overline C_{k,n}^f$.
\end{Lem}
The proof is similar to that of Proposition \ref{prop:tildeomega} and will be given in Appendix \ref{app:thetaextendsproof}.

\begin{proof}[Proof of Lemma \ref{l-conv-fib}]
The forms $\theta_\Gamma^{\frac 1 2}$ extend to the compactification $\overline C_{k,n}^f$. Hence the integral clearly exists. The resulting connection $\omega^{\frac 1 2}_{\mathrm{AT},k}$ has been defined by Alekseev and Torossian \cite{AT-1}. It is smooth since it is an integral over a compact set of a smooth family of functions. 

We know that $\underline C_{k}$ is a complex manifold and $\vartheta^{\frac 1 2}_\Gamma$ is real, hence we may write
\[
\vartheta^{\frac 1 2}_\Gamma = A_\Gamma + \overline A_\Gamma
\]
where $A_\Gamma$ in $\Omega^{1,0}(\underline C_{k})$ and $\overline A_\Gamma$ in $\Omega^{0,1}(\underline C_{k})$ is its complex conjugate.
Smoothness of $\vartheta^{\frac 1 2}_\Gamma$ is equivalent to smoothness of $A_\Gamma$ (and hence also $\bar A_\Gamma$).
Observe that in the definition of $\theta^t_\Gamma$, whenever a form $dz_j$ occurs, it is scaled by a factor $(1-t)$, while, whenever a form $d\bar z_j$ occurs, it is scaled by a factor $t$.
Since, as already observed, the fiber integral~\eqref{eq-weight-AT-1} requires only the component of $\theta^t_\Gamma$ of top degree along the fiber, the relevant terms of $\theta^t_\Gamma$ read
\[
\theta^t_\Gamma
=\sum_{\alpha=1}^k
\left(
(1-t)^{n+1}t^n dz_\alpha f_\alpha(z_1,\dots z_{n+k-2})+
(1-t)^{n}t^{n+1} d\bar z_\alpha \bar f_\alpha(z_1,\dots z_{n+k-2})
\right) dz_{k+1}d\bar z_{k+1}\cdots dz_{k+n-2}d\bar z_{k+n-2}
+\cdots
\]
where $f_\alpha$ are some functions independent of $t$ and $\cdots$ denotes irrelevant terms for integration along the fiber. 
Now, since the integral for $t=1/2$ exists and is smooth, it must exist and be smooth for all $t$ and furthermore 
\begin{equation}\label{equ:conrescaling}
\vartheta^{t}_\Gamma = 2^{2n+1} (1-t)^{n+1}t^n A_\Gamma + 2^{2n+1}(1-t)^{n}t^{n+1} \bar A_\Gamma.
\end{equation}


Similarly, by Lemma \ref{lem:a12exists} it is clear that integrals defining $\widetilde \theta_\Gamma^{\frac 1 2}$ exist. Then, since the fibers of $\underline C_{k+n}$ and $\underline C_k$ are complex manifolds, the only contributing piece of the integrand has to contain an equal number of holomorphic and antiholomorphic $1$-form factors. This number equals the complex dimension $n$ of the fiber over which we integrate. As the holomorphic, resp.\ antiholomorphic part of the propagator is scaled by $1-t$, resp.\ $t$, we see that the integrand defining $\widetilde \vartheta_\Gamma^{t}$ is the same as that in the $t=\frac 1 2$ case, except for a rescaling factor $\left(4 t(1-t)\right)^n$. In particular, the integral converges and furthermore
\[
 \widetilde \vartheta_\Gamma^{t} = \left(4 t(1-t)\right)^n \widetilde \vartheta_\Gamma^{\frac 1 2}.
\]
\end{proof}

Let us extract two simple corollaries of the above proof.

\begin{Lem}\label{l-ascaling}
The functions $a^t_k$ obey
\[
a^t_k=(4t(1-t))^N a^{\frac 1 2}_k,
\]
where $N$ is the operator acting on a graph by multiplication by the number of internal vertices.
\end{Lem}

\begin{Cor}
The connections $\nabla^t_k$ are $\alg t_k\subset \sder_k$ connections for all $t$.
\end{Cor}
\begin{proof}
For $t=1/2$, the claim has been proved in \cite{SW}*{Section 4}. 
By the arguments in the second part of the proof of Lemma~\ref{l-conv-fib}, the connection forms for all other $t$ are obtained by rescaling the holomorphic and anti-holomorphic parts of the connection form of $\nabla^{\frac 1 2}_k$ as in \eqref{equ:conrescaling}. 
The rescaling factor there depends only on the degree (of $\Gamma$) and hence (since $\alg t_k$ is graded) the resulting connection is indeed a $\alg t_k$-connection.
\end{proof}

Later will also need the following auxiliary result.

\begin{Lem}\label{l-triv-int-bas}
For $k\geq 2$ and $\Gamma\in\sder_k$ a graph with at least one internal vertex, the $1$-form $\vartheta^t_\Gamma$ and the $0$-form $\widetilde\vartheta^t_\Gamma$ descend to $\underline C_n/\mathbb C^\times$. 
\end{Lem}
\begin{proof}
We consider $\bbC^\times=\mathbb R_+\times S^1$. We further borrow previous notation for the vector fields $u$, $v$ generating the infinitesimal action of $S^1$ and $\mathbb R_+$ on $\Conf_n$ from Subsubsection~\ref{ss-5-2}.

Observe that the $\bbC^\times$-action on $\Conf_n$ is orientation-preserving.
Furthermore, the projection $\pi_{k+n}$ from $\underline C_{k+n}$ onto $\underline C_k$ is obviously $\bbC^\times$-equivariant.

For $\Gamma$ a general element of $\sder_k$ with at least an internal vertex, let us consider the $1$-form $\vartheta^t_\Gamma$ and the $0$-form $\widetilde\vartheta_\Gamma^t$.

By the very definition of integration along the fiber, we get 
\[
\lambda^*(\vartheta^t_\Gamma)=\lambda^*\!\left(\int_{\underline C_{k+n}/\underline C_k}\theta^t_\Gamma\right)=\int_{\underline C_{k+n}/\underline C_k}\lambda^*\!\left(\theta^t_\Gamma\right)=\vartheta^t_\Gamma,
\]
for $\lambda$ in $\bbC^\times$, because $\theta^t$ is $\bbC^\times$-invariant; the second equality follows from the $\bbC^\times$-equivariance of $\pi_{k+n}$.

Furthermore, 
\[
\iota_\xi\!\left(\vartheta^t_\Gamma\right)=\iota_\xi\!\left(\int_{\underline C_{k+n}/\underline C_k}\theta^t_\Gamma\right)=\int_{\underline C_{k+n}/\underline C_k}\iota_\xi\!\left(\theta^t_\Gamma\right)=0,
\]
where $\xi$ is either $u$ or $v$.
The second equality follows again from the $\bbC^\times$-equivariance of $\pi_{k+n}$.
On the other hand, the third equality follows from the fact that contraction by $\xi$ of $\theta^t_\Gamma$ produces a sum of forms indexed by the edges of $\Gamma$: the summand corresponding to $e$ is obtained by associating with every edge $e'\neq e$ of $\Gamma$ the $1$-form $\theta_{e'}^t$ and to $e$ a constant depending on $u$ or $v$.
Since $\Gamma$ is trivalent, it is as if such a summand were associated with a graph $\Gamma_e$ with exactly one bivalent internal vertex: the involution argument from~\cite{K3}*{Lemma 2.2} yields the third equality.

If we consider the $0$-form $\widetilde\vartheta^t_\Gamma$, it suffices to prove $\bbC^\times$-invariance.
Previous computations imply
\[
\lambda^*(\widetilde\vartheta^t_\Gamma)=\int_{\underline C_{k+n}/\underline C_k}\lambda^*\!\left(\widetilde\theta^t_\Gamma\right)=\widetilde\vartheta^t_\Gamma+\log(\lambda)\int_{\underline C_{k+n}/\underline C_k}\alpha^t_\Gamma,
\]
where $\alpha^t_\Gamma$ is defined as $\widetilde\theta^t_\Gamma$ but replacing $\eta$ by $i/\pi$: this follows from 
\[
\lambda^*(\eta)=\eta+\frac{1}{\pi i}\log(\lambda).
\]
In particular, the second summand in the rightmost expression in the previous chain of equalities is proportional to the integration along the fiber of $\iota_\xi(\theta^t_\Gamma)$, hence it vanishes.
\end{proof}


Observe now that $\underline C_2/\bbC^\times=\{pt\}$ is $0$-dimensional: therefore, 
\[
\omega^t_{\mathrm{AT},2}=\theta^t\ t_{12}.
\]
The $1$-form $\theta^t$, on the other hand, by previous computations, is obviously not $\bbC^\times$-basic on $\underline C_2$, yielding an obviously flat connection over the trivial principal $\mathsf T_2$-bundle over $\bbC^\times$.

Similar arguments imply the following corollary.
\begin{Cor}\label{c-a2}
The functions $a^t_2$ have the form
\[
a^t_2=-\frac{1}{\pi i} \log|z_1-z_2|\ t_{12} 
+\nu^t_2,
\]
for some elements $\nu^t_2$ of $\sder_2$, which we understand as constant functions on $\underline C_2$.
\end{Cor}
\begin{proof}
The claim follows immediately from the fact that 
\[
a^t_2=-\frac{1}{\pi i} \log|z_1-z_2|\ t_{12}+\cdots,
\]
where now $\cdots$ is a function on $\underline C_2$ which descends to $\underline C_2/\bbC^\times=\{0,1\}$ in virtue of Lemma~\ref{l-triv-int-bas}, which we denote by $\nu^t_2$. 
\end{proof}
Finally, the next proposition puts into relationship $a^t_2$ and the family $x^t$ in $\GC_\mathrm{cl}$ of degree $0$ from Subsection~\ref{ss-6-2}.
\begin{Prop}\label{p-nuingrt}
The following equality holds true:
\[
\nu^t_2=\phi(x^t)=:\tau^t\in \grt_1\subset \sder_2,
\]
for all $t$, where the family of graph cocycles $x^t$ in $\GC$ has been defined in~\eqref{equ:xtdef}.
\end{Prop}
\begin{proof}
The claim follows immediately from Identity~\eqref{eq-ctgammasimpl}, Subsection~\ref{ss-7-4}, and the defining formula for $\nu^t_2$, by taking into account that $\nu_2^t$ is defined by the very same integrals: in fact, $\mathrm{Conf}_n/\mathbb C=\mathbb C^\times\times\mathrm{Conf}_{n-2}(\mathbb C\smallsetminus\{0,1\})$, for $n\geq 3$, and since $\nu_2^t$ is a constant function on $\mathbb C^\times$, thus it can be evaluated at $1$, and the fiber identifies with $\mathrm{Conf}_{n-2}(\mathbb C\smallsetminus\{0,1\})$.
\end{proof}

\subsection{Flatness and gauge transformations}

The following Proposition describes the two main properties of $\nabla^t_k$.
\begin{Prop}\label{p-flat-gauge-AT-bound}
For $k\geq 2$, the family $\nabla^t_k$ of connections on the trivial principal $\SAut_k$-bundle over $\underline C_k$ is flat, and the connections in the family are transformed into each other by the family of gauge transformation corresponding to $a^t_k$.
\begin{align}
\label{eq-flat-gauge-AT-bound}
d\omega^t_{\mathrm{AT},k}-\frac{1}2\left[\omega^t_{\mathrm{AT},k},\omega^t_{\mathrm{AT},k}\right]&=0 
&
\p_t \nabla^t_k&=\nabla^t_k a^t_k.
\end{align}
\end{Prop}

For $t=1/2$, the connection $\nabla^t_k$ reduces to the Alekseev-Torossian connection, which is known to be flat, see also~\cites{AT-1,SW}.
Hence, the first equality of \eqref{eq-flat-gauge-AT-bound} holds for $t=1/2$. 
In order to show the proposition, it suffices therefore to prove the second equality, since any gauge transformation of a flat connection is again flat.

Let us prove the second identity in~\eqref{eq-flat-gauge-AT-bound}. 
It is sufficient to show the weak version of that equality, i.~e., that for each compactly supported $2k-3$-test-form $\psi$ on $\underline C_k$ 
\begin{equation}\label{equ:weakform}
\int_{\underline C_k} \psi \p_t \omega^t_{\mathrm{AT},k} = \int_{\underline C_k} (d\psi) a^t_k  - \int_{\underline C_k} \psi [\omega^t_{\mathrm{AT},k}, a^t_k].
\end{equation}
The first term on the right-hand side may be rewritten as 
\[
 \int_{\underline C_k} (d\psi) a^t_k = 
 \sum_\Gamma \Gamma \int_{\underline C_k} (d\psi)\! \left(\int_{\underline C_{n+k}/\underline C_k} \widetilde \theta^t_\Gamma \right)= 
 \sum_\Gamma \Gamma \int_{\underline C_{n+k}} (d\psi) \widetilde \theta^t_\Gamma
\]
where we used the definition \eqref{eq-AT-bound-0} of $a^t_k$ and in the sum $n$ denotes the number of internal vertices of the graph $\Gamma$.
We next want to apply the regularized Stokes' Theorem to the above expression. 
To this end denote by $\widehat C_{n,k}\to K$ the restriction of the bundle $\overline C^f_{n,k}\to \underline C_k$ to some compact $K\subset C_k$ containing the support of $\psi$.
We need to show that the integrand is regularizable on $\widehat C_{n,k}$, and compute the regularization on the boundary.

\begin{Prop}\label{prop:8-10}
 With the above notation, the form $\psi \widetilde \theta^t_\Gamma$ defined on the interior of $\overline C^f_{n,k}$ is regularizable.
 Furthermore, fix a codimension 1 boundary stratum $B$. Let the corresponding subgraph be $\Gamma'$ and let the graph obtained by contracting $\Gamma'$ be $\Gamma''$.
 Then the regularization satisfies
 \[
  \Reg_B (\psi \widetilde \theta^t_\Gamma )=
  \begin{cases}
   \psi \wedge \frac{d\phi}{2\pi i} \wedge \widetilde \theta^t_{\Gamma''}
   & \text{if $\Gamma'$ consists of 2 vertices connected by an edge}
   \\
   0 & \text{otherwise}
  \end{cases}.
 \]
\end{Prop}
The proof will be postponed to Appendix \ref{app:8-10proof}.
%
%

\begin{proof}[Proof of Proposition \ref{p-flat-gauge-AT-bound}]

Compute, using the regularized Stokes' Theorem
\begin{align*}
 \int_{\underline C_{n+k}} (d\psi) \widetilde \theta^t_\Gamma
 &=
 \int_{\underline C_{n+k}} (d(\psi \widetilde \theta^t_\Gamma) + \psi d\widetilde \theta^t_\Gamma )
 \\ &=
 \int_{\p_f \underline C_{n+k}} \Reg ( \psi \widetilde \theta^t_\Gamma)
 + \int_{\underline C_{n+k}} \psi \p_t \theta_\Gamma^t.
\end{align*}
The very last term is equal to $\int_{\underline C_{k}} \psi \p_t \vartheta_\Gamma^t$. This produces the left-hand side of \eqref{equ:weakform}. Next let us turn to the first term on the right-hand side
\begin{align*}
  \int_{\p_f \underline C_{n+k}} \Reg (\psi \widetilde \theta^t_\Gamma)
  &=
  \sum_B \int_{\p_B \overline C_{n+k}} \Reg_B (\psi \widetilde \theta^t_\Gamma).
\end{align*}

By Proposition \ref{prop:8-10} the only contributing boundary strata are those corresponding to a collapse of exactly two points.
There are 2 cases to be considered.
\begin{itemize}
 \item Suppose $B$ describes a boundary stratum on which two internal points collapse, away from the external points. 
 Then the corresponding terms at the end do not contribute to \eqref{equ:weakform}, since the contribution is annihilated by the IHX relations.
 \item Suppose $B$ describes a boundary stratum on which one internal point collapses towards an external point.
 Note that the graph $\Gamma''$ in this case naturally splits into two internally connected trivalent trees, say $\Gamma_1''$ and $\Gamma_2''$. Furthermore
 \[
  \widetilde \theta^t_{\Gamma''} =  \widetilde \theta^t_{\Gamma_1''}  \theta^t_{\Gamma_2''} 
  +  \theta^t_{\Gamma_1''} \widetilde \theta^t_{\Gamma_2''}.
 \]
 The contribution of such a boundary stratum is hence
 \[
  \int_{\p_B \overline C_{n+k}} \Reg_B (\psi \widetilde \theta^t_\Gamma)
  =
  \int_{\underline C_{n_1+n_2+k}}
  \phi(\widetilde \theta^t_{\Gamma_1''}  \theta^t_{\Gamma_2''} 
  +  \theta^t_{\Gamma_1''} \widetilde \theta^t_{\Gamma_2''})
  =
  \int_{\underline C_{n_1+n_2+k}} \phi(\widetilde \vartheta^t_{\Gamma_1''}  \vartheta^t_{\Gamma_2''} 
  +  \vartheta^t_{\Gamma_1''} \widetilde \vartheta^t_{\Gamma_2''})
 \]
 where $n_1$ and $n_2$ (with $n_1+n_2=n-1$) are the numbers of internal vertices in $\Gamma_1''$ and $\Gamma_2''$.
\end{itemize}

The latter boundary terms together produce the last term in \eqref{equ:weakform}, and hence the Proposition is shown.
\end{proof}

\begin{Rem}
 The statement of Proposition \ref{p-flat-gauge-AT-bound} may be recast as a set of properties of the Alekseev-Torossian connection $\omega_{{\rm AT},k}^{\frac 1 2}$. For example, if we split the latter connection into a $(1,0)$- and a $(0,1)$-form part, say, $\omega_{{\rm AT},k}^{\frac 1 2}=\omega+\bar  \omega$ then the first equation of \eqref{eq-flat-gauge-AT-bound} is equivalent to the following four equations for $\omega$.
 \begin{align*}
  \p \omega &= 0 & [\omega,\omega]&=0 \\
  \p \bar \omega &= \bar \p \omega & \p \bar \omega + \bar \p \omega + [\omega, \bar \omega] &= 0
 \end{align*}
Let us derive the first two equations and leave the other two to the reader. The $(2,0)$-part of the first equation of \eqref{eq-flat-gauge-AT-bound} reads 
\[
 2^{2N+1}t^N(1-t)^{N+1} \p \omega - \frac 1 2 [ 2^{2N+1}t^N(1-t)^{N+1} \omega, 2^{2N+1}t^N(1-t)^{N+1} \omega] = 0
\]
where we again used the operator $N$ acting on graphs by the multiplication with the nuber of internal vertices and the scaling behavior \eqref{equ:conrescaling} of $\omega_{{\rm AT},k}^{t}$. Using that the operator $(N+1)$ is a derivation of the Lie algebra $\sder_k$ we may simplify the above equation to 
\begin{align*}
 0 &= 2^{2N+1}t^N(1-t)^{N+1} \p \omega - \frac 1 2 [ 2^{2N+1}t^N(1-t)^{N+1} \omega, 2^{2N+1}t^N(1-t)^{N+1} \omega]
 \\
 &=
 2^{2N+1}t^N(1-t)^{N+1} \p \omega - 2^{2N-1} t^{N-1}(1-t)^{N+1} [  \omega, \omega]
 \\
 &=
 2^{2N-1}t^{N-1}(1-t)^{N+1} \left( 4 t \p \omega - [  \omega, \omega]\right).
\end{align*}
Hence the validity of this equation for all $t$ is equivalent to demanding that $\p \omega= [  \omega, \omega]=0$.
\end{Rem}

\subsection{Example: Explicit computation of a term of \texorpdfstring{$\nabla^t_k$}{nablatk}}\label{sss-8-1-1} 
We have proved that $\nabla^t_k$ is a smooth connection on the trivial principal $\SAut_k$-bundle over $\underline C_k$; however, it does not extend to the compactification $\overline{\underline C}_k$.
Concretely, the forms $\vartheta^t_\Gamma$ may have singularities at configurations where two or more points collapse: this is indeed the case, as the next computations show. 

As already remarked, for $k\geq 2$, the family $\nabla^t_k$ of flat connections on the trivial principal $\SAut_k$-bundle over $\underline C_k$ can be written as follows:
\[
\omega^t_{\mathrm{AT},k}=\sum_{1\leq i<j\leq k}\theta^t_{ij}\ t_{ij}+\cdots,
\]
where $\cdots$ denotes an $\sder_k$-valued $1$-form over $\underline C_k/\bbC^\times=\Conf_{k-2}(\bbC\smallsetminus\{0,1\})$.

We want to compute explicitly the piece of $\omega^t_{\mathrm{AT},k}$, for $k\geq 3$, associated to elements of $\sder_k$ with exactly one internal vertex.
Such elements are uniquely determined by triples $(i_1,i_2,i_3)$ of elements of $[k]$, such that $1\leq i_1<i_2<i_3\leq k$: with the triple $(i_1,i_2,i_3)$ we associate the element of $\sder_k$ depicted in Figure~\ref{fig-1-int-sder}.
\begin{figure}
\centering
\begin{tikzpicture}[>=latex]
\tikzstyle{k-int}=[draw,fill=gray!40,circle,inner sep=0pt,minimum size=2.5mm]
\tikzstyle{n-int}=[draw,fill=black,circle,inner sep=0pt,minimum size=2.5mm]
\tikzstyle{ext}=[draw,fill=white,circle,inner sep=0pt,minimum size=2.5mm]
\tikzstyle{ext0}=[draw,fill=white,circle,inner sep=0pt,minimum size=4mm]
\tikzstyle{spec}=[draw,rectangle,inner sep=0pt,minimum size=4mm]

\begin{scope}[scale=.75]
\node[k-int] (1) at (-2,0) {};
\node[k-int] (i1) at (-1,0) {};
\node[k-int] (i2) at (0,0) {};
\node[k-int] (i3) at (1,0) {};
\node[k-int] (k) at (2,0) {};
\node at (-1.5,0) {$\dots$};
\node at (-.5,0) {$\dots$};
\node at (.5,0) {$\dots$};
\node at (1.5,0) {$\dots$};
\node[n-int] (i) at (0,1) {};
\draw[thick] (i) to (i1);
\draw[thick] (i) to (i2);
\draw[thick] (i) to (i3);
\node at (-1,-.5) {$i_1$};
\node at (0,-.5) {$i_2$};
\node at (1,-.5) {$i_3$};
\end{scope}
\end{tikzpicture}
\caption{\label{fig-1-int-sder} The unique element of $\sder_k$ associated with the ordered triple $\underline i$.}
\end{figure}
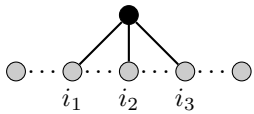

For a triple $\underline i=(i_1,i_2,i_3)$ as before, let us denote by $\Gamma_{\underline i}$ the corresponding element of $\sder_k$: hence, the term of $\nabla^t_k$ associated with elements of $\sder_k$ with exactly one internal vertex has the form
\[
\sum_{\underline i}\vartheta_{\Gamma_{\underline i}}^t\ \Gamma_{\underline i}.
\]
We want to compute explicitly the integral weight $\vartheta_{\Gamma_{\underline i}}^t$.

It suffices to compute it for $k=3$, whence $\underline i=(1,2,3)$, and we introduce for the sake of simplicity the notation $\Gamma_{\underline i}=\Gamma_3$.

%
%

First of all, we identify $\underline C_3$ and $\underline C_4$ with $\bbC^\times\times \Conf_2(\bbC\smallsetminus\{0,1\})$ and $\bbC^\times\times \Conf_1(\bbC\smallsetminus \{0,1\})$ by fixing to $0$ the point corresponding to $1$: more precisely,
\[
\begin{aligned}
\bbC^\times\times \Conf_2(\bbC\smallsetminus\{0,1\})\ni (\lambda;(z,w))&\mapsto \left[\left(0,\lambda,\lambda z,\lambda w\right)\right]\in \underline C_4,\\
\bbC^\times\times \Conf_1(\bbC\smallsetminus\{0,1\})\ni (\lambda;z)&\mapsto \left[\left(0,\lambda,\lambda z\right)\right]\in \underline C_3.
\end{aligned}
\]
Accordingly, the projection $\pi_4$ from $\underline C_4$ onto $\underline C_3$ yields a projection $\pi_4$ from $\bbC^\times \times \Conf_2(\bbC\smallsetminus\{0,1\})$ onto $\bbC^\times \times \Conf_1(\bbC\smallsetminus\{0,1\})$ which simply forgets the fourth coordinate.

By its very definition, the integral weight $\vartheta^t_\Gamma$ is given by the expression
\begin{equation}\label{eq-AT-3}
\vartheta_{\Gamma_3}^t=\int_{w\in \bbC\smallsetminus \{0,1,z\}}\theta^t(\lambda w,0)\theta^t(\lambda w,\lambda)\theta^t(\lambda w,\lambda z).
\end{equation}

The multiplicative property of the complex logarithm yields
\[
\theta^t(\lambda w,0)=\xi(\lambda)+\widehat\theta^t(w,0),\ \theta^t(\lambda w,\lambda)=\xi(\lambda)+\widehat\theta^t(w,1),\ \theta^t(\lambda z,\lambda w)=\xi(\lambda)+\widehat\theta^t(z,w),
\]
borrowing previous notation, whence the integrand in~\eqref{eq-AT-3} may be re-written as
\[
\theta^t(\lambda w,0)\theta^t(\lambda w,\lambda)\theta^t(\lambda z,\lambda w)=\widehat\theta^t(w,0)\widehat\theta^t(w,1)\widehat\theta^t(z,w)+\xi(\lambda)\alpha(z,w).
\]
Observe that the $2$-form $\alpha(z,w)$ is associated with a linear combination of bivalent graphs by its very construction: thus, its integral over $w$ vanishes by means of the involution argument from~\cite{K3}*{Lemma 2.2}.

We are thus reduced to compute the integral 
\[
\int_{w\in \bbC\smallsetminus \{0,1,z\}}\widehat\theta^t(w,0)\widehat\theta^t(w,1)\widehat\theta^t(z,w)=\int_{w\in \bbC\smallsetminus \{0,1,z\}}\widehat\theta^t(z,w)\widehat\theta^t(w,0)\widehat\theta^t(w,1).
\]
In fact, the integration is readily verified to be over the whole complex plane $\bbC$ by cutting out infinitesimal circles around $0$, $1$ and $z$ and the circle at infinity and proving that the integrand extends to the corresponding limits as the radii tend to $0$ or infinity: this is a very simple case of what we considered before in proving the convergence of the integral weights $\vartheta^t_\Gamma$, see Lemma~\ref{l-conv-fib}.

The second and third factor do not depend on $z$, hence they both provide the volume form with respect to which we integrate: this must be in turn a $(1,1)$-form, and standard manipulations using the trick with logarithms imply that the integrand on the right-hand side takes the form
\[
\begin{aligned}
\widehat\theta^t(z,w)\widehat\theta^t(w,0)\widehat\theta^t(w,1)&=-\frac{(1-t)[t(1-t)]}{2\pi^3 i}\frac{dz}{w-z}d\mathrm{arg}(w)d\mathrm{arg}(w-1)+\frac{t[t(1-t)]}{2\pi^3 i}\frac{d\overline z}{\overline w-\overline z}d\mathrm{arg}(w)d\mathrm{arg}(w-1)=\\
&=-\frac{(1-t)[t(1-t)]}{2\pi^3 i}\frac{dz}{w-z}d\mathrm{arg}(w)d\mathrm{arg}(w-1)+\frac{t[t(1-t)]}{2\pi^3 i}\frac{d\overline z}{w-\overline z}d\mathrm{arg}(w)d\mathrm{arg}(w-1),
\end{aligned}
\]
where we have used the orientation-reversing involution $w\mapsto \overline w$ on the second term on the right-hand side of the first row.

Hence, it suffices to compute the integral 
\[
\begin{aligned}
\int_\bbC \frac{1}{w-z}d\mathrm{arg}(w)d\mathrm{arg}(w-1)&=\int_{|w|<1}\frac{1}{w-z}d\mathrm{arg}(w)d\mathrm{arg}(w-1)+\int_{|w|>1}\frac{1}{w-z}d\mathrm{arg}(w)d\mathrm{arg}(w-1)=\\
&=\int_{|w|<1}\frac{1}{w-z}d\mathrm{arg}(w)d\mathrm{arg}(w-1)+\int_{|w|<1}\frac{1}{\frac 1 w-z}d\mathrm{arg}\!\left(\frac 1 w\right)d\mathrm{arg}\!\left(\frac 1 w-1\right)=\\
&=\int_{|w|<1}\frac{1}{w-z}d\mathrm{arg}(w)d\mathrm{arg}(w-1)-\int_{|w|<1}\frac{w}{1-zw}d\mathrm{arg}(w)\left(d\mathrm{arg}(1-w)-d\mathrm{arg}(w)\right)=\\
&=\int_{|w|<1}\frac{1}{w-z}d\mathrm{arg}(w)d\mathrm{arg}(w-1)-\int_{|w|<1}\frac{w}{1-zw}d\mathrm{arg}(w)d\mathrm{arg}(w-1).
\end{aligned}
\]

Let us consider first the case $|z|<1$.
By recalling the convergent geometric series expansion in the complex domain ${|w|<1}$ and by standard manipulations, the last expression in the previous chain of equalities can be re-written as 
\[
\begin{aligned}
&\int_{|w|<1}\frac{1}{w-z}d\mathrm{arg}(w)d\mathrm{arg}(w-1)-\int_{|w|<1}\frac{w}{1-zw}d\mathrm{arg}(w)d\mathrm{arg}(w-1)=-\sum_{m\geq 0} \frac{1}{z^{m+1}}\ \left(\int_{|w|<|z|}w^m d\mathrm{arg}(w)d\mathrm{arg}(w-1)\right)+\\
&\phantom{=}+\sum_{m\geq 0} z^m\ \left(\int_{|z|<|w|<1}\frac{1}{w^{m+1}} d\mathrm{arg}(w)d\mathrm{arg}(w-1)\right)-\sum_{m\geq 0} z^m\ \left(\int_{|w|<|1|}w^{m+1} d\mathrm{arg}(w)d\mathrm{arg}(w-1)\right). 
\end{aligned}
\]
The last three inner integrals have been computed in the proof of Proposition~\ref{p-polylog}, Appendix~\ref{app-coc}, whence we obtain
\[
\int_{|w|<1}\frac{1}{w-z}d\mathrm{arg}(w)d\mathrm{arg}(w-1)-\int_{|w|<1}\frac{w}{1-zw}d\mathrm{arg}(w)d\mathrm{arg}(w-1)=-\pi i\left(\frac{\log(|1-z|)}{z}+\frac{\log(|z|)}{1-z}\right).
\]
Similar computations hold true for $|z|\geq 1$ and yield the same result.
\begin{Prop}\label{p-AT-3}
The piece of the family $\nabla^t_3$ of flat connections on the trivial principal $\SAut_3$-bundle over $\underline C_3$ associated to the unique element $\Gamma_3$ of $\sder_3$ with exactly one internal vertex has the explicit form
\[
\left[\frac{(1-t)[t(1-t)]}{2\pi^2}\left(\frac{\log(|1-z|)}{z}+\frac{\log(|z|)}{1-z}\right)dz+\frac{t[t(1-t)]}{2\pi^2}\left(\frac{\log(|1-z|)}{\overline z}+\frac{\log(|z|)}{1-\overline z}\right)d\overline z\right]\ \Gamma_3.
\]
\end{Prop}
For $k\geq 4$ and an ordered triple $\underline i$ as above, we obtain a similar expression by choosing well-suited global sections of $\underline C_{k+1}$ and $\underline C_k$, where the point labeled by $i_1$ is fixed at $0$, the point labeled by $i_2$ is in $\bbC^\times$: then, the previous computations apply {\em verbatim}.

\subsection{The (anti-)Knizhnik--Zamolodchikov connection}\label{sss-8-1-2}
In the previous Subsection, we have explicitly computed the term of the family $\nabla^t_k$ of flat $\sder_k$-connections corresponding to elements of $\sder_k$ with exactly one internal vertex, and have proved that it is in general non-trivial.

In view of the next Subsection, where we analyze in detail the parallel transport for $\nabla^t_3$ along a path in $\underline C_3$ which connects two points in different boundary strata of $\overline{ C}_3$, and the corresponding regularization, let us briefly discuss the explicit shape of $\nabla^t_k$, $k\geq 3$, at $t=0$, $t=1/2$ and $t=1$.

We already know from previous arguments that
\[
\nabla^0_k=d-\sum_{1\leq i<j\leq k} \theta^0_{ij}\ t_{ij}+\cdots,\ \theta^0_{ij}=\frac{1}{2\pi i}d\log(z_i-z_j)=\frac{1}{2\pi i}\frac{dz_i-dz_j}{z_i-z_j},
\]
and $\cdots$ denotes the sum of contributions associated with elements of $\sder_k$ with at least one internal vertex.

Let us consider such an element $\Gamma$ of $\sder_k$ with $n\geq 1$ internal vertices.
We identify, as before, $\underline C_{k+n}$ and $\underline C_k$ with $\bbC^\times \times\Conf_{k+n-2}(\bbC\smallsetminus\{0,1\})$ and $\bbC^\times \times\Conf_{k-2}(\bbC\smallsetminus\{0,1\})$ respectively, {\em via}
\[
\begin{aligned}
\bbC^\times \times\Conf_{k+n-2}(\bbC\smallsetminus\{0,1\})\ni \left(\lambda,z_3,\dots,z_k,z_{k+1},\dots,z_{k+n}\right)&\mapsto \left[\left(0,\lambda,\lambda z_3,\dots,\lambda z_{k+n}\right)\right]\in \underline C_{k+n},\\
\bbC^\times \times\Conf_{k-2}(\bbC\smallsetminus\{0,1\})\ni \left(\lambda,z_3,\dots,z_k\right)&\mapsto \left[\left(0,\lambda,\lambda z_3,\dots,\lambda z_{k}\right)\right]\in \underline C_{k}.
\end{aligned}
\]
The projection $\pi_{k+n}$ from $\bbC^\times \times\Conf_{k+n-2}(\bbC\smallsetminus\{0,1\})$ onto $\bbC^\times \times\Conf_{k-2}(\bbC\smallsetminus\{0,1\})$ simply forgets the last $n$ complex coordinates.
Previous computations show that $\theta^0_e$, for $e$ a general edge of $\Gamma$, is obviously a form of type $(1,0)$ on $\bbC^\times \times\Conf_{k+n-2}(\bbC\smallsetminus\{0,1\})$: in particular, $\theta^0_\Gamma$ is a form of type $(2n+1,0)$ on $\bbC^\times \times\Conf_{k+n-2}(\bbC\smallsetminus\{0,1\})$.
Since we consider its integral along the fiber of the projection $\pi_{k+n}$, whose general fiber is a complex manifold of complex dimension $n$, $\vartheta^0_\Gamma$ is non-trivial, only if the corresponding integrand has a non-trivial piece of type $(n,n)$ with respect to the fiber coordinates.
As the integrand is of type $(2n+1,0)$ in the complex coordinates of $\bbC^\times \times\Conf_{k+n-2}(\bbC\smallsetminus\{0,1\})$, there cannot be such a piece, whence $\vartheta^0_\Gamma=0$, for $\Gamma$ in $\sder_k$ with $n\geq 1$ internal vertices.

On the other hand, let us consider the case $t=1$:
\[
\nabla^1_k=d+\sum_{1\leq i<j\leq k} \theta^0_{ij}\ t_{ij}+\cdots,\ \theta^1_{ij}=\frac{1}{2\pi i}d\log(\overline z_i-\overline z_j)=\frac{1}{2\pi i}\frac{d\overline z_i-d\overline z_j}{\overline z_i-\overline z_j},
\]
and $\cdots$ as above.

The very same arguments for the case $t=0$ imply that $\vartheta^1_\Gamma=0$, for $\Gamma$ in $\sder_k$ with $n\geq 1$ internal vertices: the only modification to be taken into account is that $\theta^1_e$ is of type $(0,1)$, hence $\theta^1_\Gamma$ is of type $(0,2n+1)$, but the remaining arguments can be repeated {\em verbatim}.

As a consequence, we get
\[
\nabla_k^0=d-\frac{1}{2\pi i}\sum_{1\leq i<j\leq k}\frac{dz_i-dz_j}{z_i-z_j}\ t_{ij}=\nabla_{\mathrm{KZ},k},\quad \nabla_k^1=d+\frac{1}{2\pi i}\sum_{1\leq i<j\leq k}\frac{d\overline z_i-d\overline z_j}{\overline z_i-\overline z_j}\ t_{ij}=\overline \nabla_{\mathrm{KZ},k},
\]
{\em i.~e.} $\nabla^0_k$ and $\nabla^1_k$ identify with the Knizhnik--Zamolodchikov (shortly, from now on, KZ) holomorphic and anti-holomorphic connection on $\underline C_k$.

Let us finally consider $\nabla_k^\frac 1 2$.
We know that the contributions to $\nabla^\frac 1 2_k$ associated with elements of $\sder_k$ with at least one internal vertex are actually defined over $\underline C_k/\bbC^\times=\Conf_{k-2}(\bbC\smallsetminus\{0,1\})$: in particular, they contribute to an $\sder$-valued $S^1$-basic $1$-form on $C_k$.
The remaining contributions, {\em i.e.} 
\[
\sum_{1\leq i<j\leq k} \theta^\frac 1 2_{ij}\ t_{ij},\ \theta^\frac 1 2_{ij}=\frac{1}{2\pi}d\mathrm{arg}(z_i-z_j),
\]
are also obviously $\mathbb R_+$-basic, hence descend to $C_k$.
Therefore, $\nabla^\frac 1 2_k$ identifies with the ``value at the boundary stratum at infinity'' $\nabla_{\mathrm{AT},k}$ of the AT connection~\cite{AT-1}*{Subsection 4.1}.

\subsection{Associators}\label{ss-8-2}
For $\varepsilon>0$ small enough, let $\Phi^t_\varepsilon$ in $\SAut_3$ be the parallel transport with respect to the family $\nabla^t_3$ of flat connections along the path between the configurations $A_\epsilon$ and $B_\epsilon$ in $\underline C_3$ depicted in Figure~\ref{fig-monod}.
Concretely, identify $\underline C_3$ with $\bbC^\times\times \Conf_1(\bbC\smallsetminus\{0,1\})$. Then consider the curve in $\underline C_3$ given by 
\[
[0,1]\in s\mapsto \left[\left(0,(1-s)\varepsilon+s(1-\varepsilon),1\right)\right]\in \underline C_3,
\]
in particular, $A_\varepsilon=[(0,\varepsilon,1)]$ and $B_\varepsilon=[(0,1-\varepsilon,1)]$.
Observe that the parallel transport $\Phi^t_\varepsilon$ does not depend on the specific chosen path because of the flatness of $\nabla_3^t$.

\begin{figure}
\centering
\begin{tikzpicture}[>=latex]
\tikzstyle{k-int}=[draw,fill=gray!40,circle,inner sep=0pt,minimum size=1.5mm]
\tikzstyle{n-int}=[draw,fill=black,circle,inner sep=0pt,minimum size=1.5mm]
\tikzstyle{ext}=[draw,fill=white,circle,inner sep=0pt,minimum size=1.5mm]
\tikzstyle{ext0}=[draw,fill=white,circle,inner sep=0pt,minimum size=4mm]
\tikzstyle{spec}=[draw,rectangle,inner sep=0pt,minimum size=4mm]

\begin{scope}
\draw[dashed] (-2,0) to (2,0);
\node[n-int] (1) at (-1,0) {};
\node[n-int] (2) at (-0.5,0) {};
\node[n-int] (3) at (1,0) {};
\node at (0,-.5) {$\displaystyle A_\varepsilon\subseteq \underline C_3$}; 
\node at (-1,.5) {$0$};
\node at (1,.5) {$1$};
\node at (-.5,.5) {$\varepsilon$};
\end{scope}

\begin{scope}[shift={(6,0)}]
\draw[dashed] (-2,0) to (2,0);
\node[n-int] (1) at (-1,0) {};
\node[n-int] (2) at (0.5,0) {};
\node[n-int] (3) at (1,0) {};
\node at (0,-.5) {$\displaystyle B_\varepsilon\subseteq \underline C_3$};
\node at (-1,.5) {$0$};
\node at (1,.5) {$1$}; 
\node at (.5,.5) {$1-\varepsilon$};
\end{scope}

\end{tikzpicture}
\caption{\label{fig-monod} For $\varepsilon>0$ small enough, the configurations $A_\varepsilon$ and $B_\varepsilon$ in $\underline C_3$.}
\end{figure}
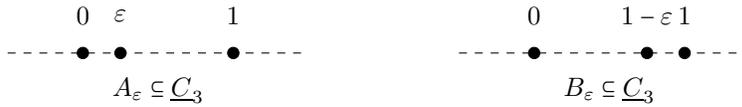
For general $t$, $\Phi^t_\varepsilon$ becomes singular as $\varepsilon$ approaches $0$. 
The goal of the present subsection is to understand the nature of the singularity of $\Phi_\varepsilon^t$ at $\varepsilon=0$ and to construct a suitable regularization.

First of all, we need an explicit expression for the gauge transformation $G^t$ relating $\nabla^t_k$ and $\nabla_k^\frac 1 2$. 
Recalling the second equality in~\eqref{eq-flat-gauge-AT-bound}, $G^t$ is an $\mathsf{SAut}_k$-valued function on $\underline C_k$, whose explicit shape is 
\[
G^t=\mathcal P\exp\!\left(\int_{\frac 1 2}^ta^s_kds\right)=1+\int_{\frac 1 2}^t a^s_kds+\int_{\frac 1 2}^t\left(\int_\frac 1 2 ^{t_1}a^{s_2}_kds_2\right)a^{s_1}_kds_1+\cdots
\]
where the notation $\mathcal P\exp$ refers to the path-ordered exponential, which is given by a sum of iterated integrals as above.
Observe that the iterated integrals exist because of the scaling property of $a^t_k$ with respect of $t$ proved in Lemma~\ref{l-ascaling}.

Consequently, we may write
\begin{equation}\label{eq-Phiepsgauge}
\Phi^t_\varepsilon =G^t(B_\varepsilon)\ \Phi^{\frac 1 2}_\varepsilon\ G^t(A_\varepsilon)^{-1},
\end{equation}
where $G^t(B_\varepsilon)$, $G^t(A_\varepsilon)$ denote the values of $G^t$ at the points $B_\varepsilon$, $A_\varepsilon$ of $C_3$ respectively. 
Since the $\sder_3$-valued flat connection $\nabla^{\frac 1 2}_3=\nabla_\mathrm{AT}$ on $C_3$ has coefficients in PA-forms, the limit of $\Phi^{\frac 1 2}_\varepsilon$ as $\epsilon$ approaches $0$ exists and coincides with the AT associator $\Phi_\mathrm{AT}$. 

We recall the following result from~\cite{AT-1}*{Conjecture 1} and~\cite{SW}*{Proposition 4 and Section 5}.
\begin{Prop}[\cite{SW}]\label{p-PhiATDrinfeld}
The AT associator $\Phi_\mathrm{AT}$ is a Drinfel{\cprime}d associator, i.~e.\ it takes values in $\mathrm T_3$.
\end{Prop}
Therefore, to understand the singular behavior of $\Phi^t_\varepsilon$, it suffices to understand the singular behavior of $G^t(A_\varepsilon)$ and $G^t(B_\epsilon)$. 
This is analyzed in detail in the following Proposition, whose proof is postponed to Subsection~\ref{ss-8-3}.
\begin{Prop}\label{p-asingular}
The limits 
\begin{align*}
a^{\frac 1 2}_A 
&:= 
\lim_{\varepsilon\to 0}\ \left(a^{\frac 1 2}_3(A_\varepsilon) +\frac{1}{\pi i}\log(\varepsilon)\ t_{12}\right)
 & a^{\frac 1 2}_B 
 &:=\lim_{\varepsilon\to 0}\ \left(a^{\frac 1 2}_3(B_\varepsilon) +\frac{1}{\pi i}\log(\varepsilon)\ t_{23}\right)
\end{align*}
exist and are equal to 
\begin{align*}
a^{\frac 1 2}_A 
&= 
\left(\tau^{\frac 1 2}\right)^{1,2} + \left(\tau^{\frac 1 2}\right)^{12,3}
 & a^{\frac 1 2}_B 
 &= 
 \left(\tau^{\frac 1 2}\right)^{2,3} + \left(\tau^{\frac 1 2}\right)^{1,23},
\end{align*}
where $\tau^\frac{1}{2}$ is as in Corollary~\ref{c-a2}, and the simplicial and coproduct maps on $\sder_2$ are described in~\eqref{eq-cob-tder-1} and~\eqref{eq-cob-tder-2}, along with their combinatorial-graphical counterparts.
\end{Prop}
We will also need the following result, stating that the limit is approached sufficiently fast.
\begin{Prop}\label{prop:fastdecay}
Consider any linear map $l : \sder_3\to \bbC$ of degree $n\geq 2$, i.~e., which is 0 on all graphs except possibly those with $n-1$ internal vertices. Then there are constants $C$ and $N$ such that for all $\varepsilon$ sufficiently small
\begin{align}\label{equ:fastdecay}
|l(a^{\frac 1 2}_3(A_\varepsilon) -  a^{\frac 1 2}_A) |
&\leq C \varepsilon |\log\varepsilon|^N
 & 
|l(a^{\frac 1 2}_3(B_\varepsilon) - a^{\frac 1 2}_B)|
 &\leq  C \varepsilon |\log\varepsilon|^N
\end{align}
\end{Prop}
Taking into account Lemma~\ref{l-ascaling}, Corollary~\ref{c-a2} and Propositions~\ref{p-nuingrt},~\ref{p-asingular},~\ref{prop:fastdecay} and inserting in~\eqref{eq-Phiepsgauge}, we obtain the following Corollary.
\begin{Cor}\label{c-assoc-asymp}
As $\varepsilon$ approaches $0$, $\Phi_\varepsilon^t$ has an expansion of the form
\begin{multline}\label{eq-assoc-asymp}
\Phi_\varepsilon^t
=
\mathcal P\exp\left(\int_{\frac 1 2}^t\left(-\frac{1}{i\pi}\log(\varepsilon)\ t_{23}+\left(\tau^s\right)^{2,3}+\left(\tau^s\right)^{1,23}\right)ds\right)\ \Phi_{\varepsilon}^{\frac 1 2}\ \mathcal P\exp\left(-\int_{\frac 1 2}^t\left(-\frac{1}{i\pi}\log(\varepsilon)\ t_{12}+\left(\tau^s\right)^{1,2}+\left(\tau^s\right)^{12,3}\right)ds\right)
\\+ O(\varepsilon|\log\varepsilon|^\bullet).
\end{multline}
Here the notation $O(\varepsilon|\log\varepsilon|^\bullet)$ shall indicate terms whose coefficients drop off faster than $(\text{const})\cdot\varepsilon|\log\varepsilon|^N$ in each degree, where $N$ may depend on the degree.
\end{Cor}
\begin{proof}
First of all, let us re-write the right-hand side of~\eqref{eq-Phiepsgauge} in the form
\[
\Phi^t_\varepsilon=\mathcal P\exp\!\left(\int_{\frac 1 2}^t a_3^s(B_\varepsilon)ds\right)\ \Phi^\frac 1 2_\varepsilon\ \mathcal P\exp\!\left(-\int_{\frac 1 2}^t a_3^s(A_\varepsilon)ds\right).
\]

Propositions~\ref{p-asingular} and \ref{prop:fastdecay} imply 
\[
a^\frac 1 2_3(A_\varepsilon)=-\frac{1}{i\pi} \log(\varepsilon)\ t_{12}+\left(\tau^\frac 1 2\right)^{1,2}+\left(\tau^\frac 1 2\right)^{12,3}+\mathcal O(\varepsilon|\log\varepsilon|^\bullet),
\quad 
a^\frac 1 2_3(B_\varepsilon)=-\frac{1}{i\pi} \log(\varepsilon)\ t_{23}+\left(\tau^\frac 1 2\right)^{2,3}+\left(\tau^\frac 1 2\right)^{1,23}+\mathcal O(\varepsilon|\log\varepsilon|^\bullet).
\]

In virtue of Lemma~\ref{l-ascaling}, the $\varepsilon$-dependence in both path-ordered exponentials can be traced back to the $\varepsilon$-dependence in $a_3^\frac 1 2(A_\varepsilon)$ and $a_3^\frac 1 2(B_\varepsilon)$: observe that the simplicial and coproduct maps~\eqref{eq-cob-tder-1} and~\eqref{eq-cob-tder-2} affect only the graph part, hence the $t$-dependence remains unaffected by them, whence
\[
a^s_3(A_\varepsilon)=-\frac{1}{i\pi} \log(\varepsilon)\ t_{12}+\left(\tau^s\right)^{1,2}+\left(\tau^s\right)^{12,3}+\mathcal O(\varepsilon|\log\varepsilon|^\bullet),
\quad
a^s_3(B_\varepsilon)=-\frac{1}{i\pi} \log(\varepsilon)\ t_{23}+\left(\tau^s\right)^{2,3}+\left(\tau^s\right)^{1,23}+\mathcal O(\varepsilon|\log\varepsilon|^\bullet).
\]

Therefore, we get
\[
\begin{aligned}
\mathcal P\exp\!\left(-\int_{\frac 1 2}^t a_3^s(A_\varepsilon)ds\right)&=\mathcal P\exp\!\left(\int_{\frac 1 2}^t \left(\frac{1}{i\pi}\log(\varepsilon)\ t_{12}-\left(\tau^s\right)^{1,2}-\left(\tau^s\right)^{12,3}+\mathcal O(\varepsilon|\log\varepsilon|^\bullet)\right)ds\right),\\
\mathcal P\exp\!\left(\int_{\frac 1 2}^t a_3^s(B_\varepsilon)ds\right)&=\mathcal P\exp\!\left(\int_{\frac 1 2}^t \left(-\frac{1}{i\pi}\log(\varepsilon)\ t_{23}+\left(\tau^s\right)^{2,3}+\left(\tau^s\right)^{1,23}+\mathcal O(\varepsilon|\log\varepsilon|^\bullet)\right)ds\right),
\end{aligned}
\]
whence 
\begin{multline*}
\Phi^t_\varepsilon=\mathcal P\exp\!\left(\int_{\frac 1 2}^t\left(-\frac{1}{i\pi}\log(\varepsilon)\ t_{23}+\left(\tau^s\right)^{2,3}+\left(\tau^s\right)^{1,23}\right) ds\right)\ \Phi^\frac 1 2_\varepsilon\ \mathcal P\exp\!\left(\int_{\frac 1 2}^t \left(\frac{1}{\pi i}\log(\varepsilon)\ t_{12}-\left(\nu^s_2\right)^{1,2}-\left(\tau^s\right)^{12,3}\right)ds\right)
\\
+\mathcal O(\varepsilon|\log\varepsilon|^\bullet).
\end{multline*}

The claim follows.
\end{proof}
In particular, Corollary~\ref{c-assoc-asymp} justifies the following definition.
\begin{Def}\label{d-reg-assoc}
The regularized associator $\Phi^t_\mathrm{reg}$ is defined as the limit
\[
\Phi^t_\mathrm{reg}
:=
 \lim_{\varepsilon\to 0}
\varepsilon^{\frac{2t-1}{2\pi i} t_{23}}
\Phi^t_\varepsilon
\varepsilon^{-\frac{2t-1}{2\pi i}t_{12}}
\]
or equivalently by formally setting $\varepsilon=\log(\varepsilon) =0$ in the expression~\eqref{eq-assoc-asymp} for $\Phi^t_\varepsilon$. 
\end{Def}
\begin{Rem}
The existence of the limit can be seen as follows. Note that $t_{12}$ and $\left(\tau^s\right)^{1,2}+\left(\tau^s\right)^{12,3}$ commute. hence we may write
\[
P\exp\left(-\int_{\frac 1 2}^t\left(-\frac{1}{i\pi}\log(\varepsilon)\ t_{12}+\left(\tau^s\right)^{1,2}+\left(\tau^s\right)^{12,3}\right)ds\right)
=
P\exp\left( -\int_{\frac 1 2}^t \left(  \left(\tau^s\right)^{1,2}+\left(\tau^s\right)^{12,3} \right)ds  \right)
\cdot \exp\left(-\frac{2t-1}{2\pi i}\log(\varepsilon)\ t_{12} \right)
\]
and similarly
\[
 \mathcal P\exp\left(\int_{\frac 1 2}^t\left(-\frac{1}{i\pi}\log(\varepsilon)\ t_{23}+\left(\tau^s\right)^{2,3}+\left(\tau^s\right)^{1,23}\right)ds\right)
 =
 \exp\left(\frac{2t-1}{2\pi i}\log(\varepsilon)\ t_{23} \right)\cdot
  \mathcal P\exp\left(\int_{\frac 1 2}^t\left(\left(\tau^s\right)^{2,3}+\left(\tau^s\right)^{1,23}\right)ds\right).
\]
Hence using the expansion \eqref{eq-assoc-asymp} the divergent terms cancel and we have to calculate
\begin{equation}\label{equ:Phiregcomp}
\begin{aligned}
 \Phi^t_\mathrm{reg}
 &=
 \lim_{\varepsilon\to 0}
 \mathcal P\exp\left(\int_{\frac 1 2}^t\left(\left(\tau^s\right)^{2,3}+\left(\tau^s\right)^{1,23}\right)ds\right)
 \Phi^{\frac 1 2}_\varepsilon
 \mathcal P\exp\left(-\int_{\frac 1 2}^t\left(\left(\tau^s\right)^{1,2}+\left(\tau^s\right)^{12,3}\right)ds\right)
 +O(\varepsilon|\log\varepsilon|^\bullet)
\\&=
 \mathcal P\exp\left(\int_{\frac 1 2}^t\left(\left(\tau^s\right)^{2,3}+\left(\tau^s\right)^{1,23}\right)ds\right)
 \Phi_{\rm AT}\
 \mathcal P\exp\left(-\int_{\frac 1 2}^t\left(\left(\tau^s\right)^{1,2}+\left(\tau^s\right)^{12,3}\right)ds\right)
\end{aligned}
\end{equation}
\end{Rem}

Recall now the computations of Subsubsection~\ref{sss-8-1-2}: for $t=0$, $t=1/2$ and $t=1$, $\nabla^t_3$ equals the KZ connection, the AT connection and the anti-KZ connection respectively.

The parallel transport with respect to the KZ connection on $\underline C_3$ along the path connecting the two configurations in $\underline C_3$ depicted in Figure~\ref{fig-monod} has been first computed (after regularization) in~\cite{Dr}*{Section 2} in order to construct an explicit example of Drinfel{\cprime}d associator, the KZ associator; see also Appendix~\ref{app:KZ} for a brief review of the KZ associator and its construction.
Similarly, the anti-KZ associator is defined as the parallel transport of the anti-KZ connection.

Therefore, we obtain 
\begin{align*}
\Phi^0_\mathrm{reg} 
&= 
\Phi_\mathrm{KZ},\ &\ \Phi^\frac 1 2_\mathrm{reg} &= \Phi_\mathrm{AT},\ &\ \Phi^{1}_\mathrm{reg}&=\Phi_{\overline{\mathrm{KZ}}}.
\end{align*}
Moreover, \eqref{equ:Phiregcomp} may be written more concisely as
\begin{equation}
\label{equ:Phitexpl}
\Phi^t_\mathrm{reg}=\mathcal P\exp\!\left(\int_{\frac 1 2}^t\tau^sds\right)\cdot \Phi_\mathrm{AT}
\end{equation}
where $\cdot$ denotes the $\GRT_1$-action on Drinfel{\cprime}d associators defined in Appendix~\ref{app:assoc}. Combining the previous expression with Propositions~\ref{p-nuingrt} and~\ref{p-PhiATDrinfeld}, we obtain the following result. 
\begin{Cor}\label{c-drinfeld}
For each $t$, $\Phi^t_\mathrm{reg}$ is a Drinfel{\cprime}d associator.
\end{Cor}
\begin{proof}
For $t=\frac 1 2$, this is the statement of Proposition~\ref{p-PhiATDrinfeld}. 
Proposition~\ref{p-nuingrt} and Corollary~\ref{c-assoc-asymp} imply further that $\Phi^t_\mathrm{reg}$ is obtained from $\Phi_{AT}$ {\em via} the action of an element of $\GRT_1$, and the set of Drinfel{\cprime}d associators is a $\GRT_1$-torsor.
\end{proof}

\subsection{Proof of Proposition \ref{p-asingular}}\label{ss-8-3}
Consider the projection
\[
 \overline C_{n+3}\to \overline C_{3}.
\]
We embed the closed interval $[0,1]$ into $C_{3}$ by assigning to $x\in (0,1)$ the configuration $(0,x,1)\in \overline C_{3}$ and extend by continuity. We consider the pullback bundle $X_n\to [0,1]$ defined by
\[
 \begin{tikzpicture}
  \matrix(m)[diagram]{X_n & \overline C_{n+3} \\ {[0,1]} & \overline C_3\\ };
  \draw[-latex] (m-1-1) edge (m-1-2) edge (m-2-1) (m-2-1) edge (m-2-2) (m-1-2) edge (m-2-2);
 \end{tikzpicture}.
\]
For $\Gamma\in \sder_3$ a graph with $n$ internal vertices we defined in $\eqref{eq-weight-AT-0-integrand}$ a differential form $\widetilde\theta_\Gamma^{\frac 1 2}$ on the interior $\underline C_{n+3}$. This form restricts to a differential form $\widetilde\theta_\Gamma^{\frac 1 2}$ on the interior $X^o$ of $X$.
We understand the function $a^{\frac 1 2}$ defined on $\underline C_{3}$ as a function on $(0,1)$ by restriction, and we will write $a^{\frac 1 2}(x)$, $x\in (0,1)$ accordingly.
The coefficient $\widetilde\vartheta_\Gamma^{\frac 1 2}$ of the graph $\Gamma$ in the series defining $a^{\frac 1 2}(x)$ is given by integrating the form $\widetilde\theta_\Gamma^{\frac 1 2}$ along the fiber of $X^o_n$ over $x$.
Let us define a differential form $\hat \theta_\Gamma$ on $X_n^o$ as follows.
\begin{itemize}
\item If both the first and third external vertices of $\Gamma$ have valence $\geq 1$ then $\hat \theta_\Gamma=\widetilde\theta_\Gamma^{\frac 1 2}$.
\item Suppose the third external vertex of $\Gamma$ has valence $0$, i.~e., the graph $\Gamma$ stems in fact from a graph in $\sder_2$, included in $\sder_3$. Then we set 
\be{equ:hattheta1}
\hat \theta_\Gamma=
\widetilde\theta_\Gamma^{\frac 1 2}
- \sum_{e'\in E} \frac{(-1)^{e'-1}}{\pi i} \log(x) \prod_{e\in E\setminus \{e'\}} \frac{1}{2\pi}d\arg\!\left(z_{s(e')}-z_{t(e')}\right).
\ee
\item Suppose the first external vertex of $\Gamma$ has valence $0$. Then we set 
\be{equ:hattheta2}
\hat \theta_\Gamma=
\widetilde\theta_\Gamma^{\frac 1 2}
- \sum_{e'\in E} \frac{(-1)^{e'-1}}{\pi i} \log(1-x) \prod_{e\in E\setminus \{e'\}} \frac{1}{2\pi}d\arg\!\left(z_{s(e')}-z_{t(e')}\right.
\ee
\end{itemize}
Note that the fiber integrals over the additional terms we added are zero since the corresponding graphs effectively contain bivalent vertices, whence a vanishing Theorem (\cite{K}*{section 6.6.1}) applies.
Hence we may define $a^{\frac 1 2}(x)$ equivalently using the fiber integral over the differential forms $\hat \theta_\Gamma^{\frac 1 2}$ instead of the forms $\widetilde\theta_\Gamma^{\frac 1 2}$.

We will also set $\tilde a(x) := a^{\frac 1 2}(x) +\frac 1 {\pi i} \log(x) t_{12} +\frac 1 {\pi i} \log(1-x) t_{23}$.
Its defining series of graphs is the same as that of $a$, except that we omit graphs with no internal vertices.
Our goal is to show that the limits $\lim_{x\to 0,1 }\tilde a(x)$ exist and have the form stated in Proposition \ref{p-PhiATDrinfeld}. 

The following Lemma will be the key argument in the proof. It is shown in Appendix \ref{app:extendstoXproof} by a slight extension of the arguments leading to Lemma \ref{lem:a12exists}.
\begin{Lem}\label{lem:extendstoX}
 Let $\Gamma\in \sder_3$ be a graph with $n\geq 1$ internal vertices. Then the form $\hat\theta_\Gamma$ on the interior $X^o_n$ extends to the compactification $X_n$.
\end{Lem}

\begin{proof}[Proof of Proposition \ref{p-asingular}]
We only consider the case of $A_\varepsilon$, because the case of $B_\varepsilon$ is analogous.
We have to show that $\lim_{x\to 0}\tilde a(x) = \left(\tau^{\frac 1 2}\right)^{1,2} + \left(\tau^{\frac 1 2}\right)^{12,3}$.
But by Lemma \ref{lem:extendstoX} we may simply evaluate the fiber integrals defining $\tilde a(x)$ at the fiber over $0\in [0,1]$.
We do this for a fixed graph $\Gamma$ with $n$ internal vertices.
The fiber over $0$ has several top dimensional components, indexed by subsets $B\subset [n]$, indicating which points collapse to $0$. The component corresponding to $B$ has the form $\Conf_{|B|}(\mathbb{C}\setminus\{0,1\})\times \Conf_{n-|B|}(\mathbb C\setminus\{0,1\})$.
Accordingly, the graph $\Gamma$ has a subgraph $\Gamma'$ (with the vertices in $B$) and the remainder is a graph $\Gamma''$.
The differential form on the fiber over $0$ is
\begin{equation}\label{equ:termstobeint}
 \widetilde\theta_{\Gamma'}^{\frac 1 2}\wedge\theta_{\Gamma''}^{\frac 1 2}
 +
  \theta_{\Gamma'}^{\frac 1 2}\wedge\widetilde\theta_{\Gamma''}^{\frac 1 2}.
\end{equation}
Note that $\Gamma''$ is an internally trivalent tree, hence the first term vanishes by degree reasons unless $\Gamma''$ has no edges, i.~e., unless $B=[n]$ and the graph $\Gamma$ had no edge connecting to the third external point from the start.
The contributions of these $B$ and $\Gamma$ when integrating the first term produces $\left(\tau^{\frac 1 2}\right)^{1,2}$.

By similar reasoning, if $\Gamma'$ contains an edge than the second term of \eqref{equ:termstobeint} is zero by degree reasons.
The remaining contributions (i.~e., $B=\emptyset$) assemble to produce the term $\left(\tau^{\frac 1 2}\right)^{12,3}$. This shows Proposition~\ref{p-asingular}. 
\end{proof}

\subsection{Proof of Proposition \ref{prop:fastdecay}}\label{sec:prooffastdecay}
Recall that the connections we discuss have the form $\nabla^{t}_k=d-\omega^{t}_{\mathrm{AT},k}$, with $\omega^{t}_{\mathrm{AT},k}=\sum_{\Gamma}\vartheta^{t}_\Gamma\ \Gamma$, cf. \eqref{eq-AT-bound-0}. In particular we are interested in the case $k=3$, with one of the three points in the configuration space fixed at $0$, one at $1$, and the third point at $z\in (0,1)$ moving between $0$ and $1$ on the real axis. In this case we may write 
\[
\vartheta^{t}_\Gamma
= 
f^t_\Gamma(z) dz,
\]
for some function $f_\Gamma^t$ which are polynomials in $t$, but possibly complicated functions in $z$. It is shown in Theorem \ref{thm:logsing} in Appendix \ref{app:ATSing} that the terms in these connections have at most logarithmic singularties as long as $\Gamma$ has at least one internal vertex, i.~e., that there are constants $C,N$, possibly depending on $\Gamma$, such that for all $t\in (0,1)$ and sufficiently small $z$
\[
 |f^t_\Gamma(z)|\leq C |\log |z||^N.
\]

Recall the ``gauge transformation'' $a^t_k=\sum_{\Gamma}\widetilde\vartheta^t_\Gamma\ \Gamma$, cf. \eqref{eq-AT-bound-0}. In the current situation we may write 
\[
 \widetilde\vartheta^t_\Gamma =: F^t_\Gamma(z)
\]
as a function in only one variable, the positions of two points in the configurations we consider being fixed at $0$ and $1$. Again $F^t_\Gamma$ is a polynomial in $t$. By the second identity of \eqref{eq-flat-gauge-AT-bound} the following equation holds:
\[
- \sum_\Gamma \frac{d f^t_\Gamma}{d t} \Gamma
=
 \sum_\Gamma \frac{d F^t_{\Gamma} }{dz}\Gamma
-
\sum_{\Gamma, \Gamma'}
f^t_\Gamma F^t_{\Gamma}
[\Gamma, \Gamma']
\]
Let us define
\begin{align*}
 f^t_n &= \sum_{\substack{ \Gamma \\ \deg\Gamma=n}} f^t_\Gamma \Gamma
&
F^t_n &= \sum_{\substack{ \Gamma \\ \deg\Gamma=n}} F^t_\Gamma \Gamma
\end{align*}
where the sums run only over graphs with $n-1$ internal vertices. Then the above equation is equivalent to the system of equations
\[
 - \frac{d f^t_n}{d t} 
=
 \frac{d F^t_n }{dz}
-
\sum_{j=1}^{n-1}
[f^t_j, F^t_{n-j}]
.
\]
Our goal is to show that 
\[
 F^t_{n}(z)-F^t_n(0) = O(|z||\log |z||^N)
\]
for $n=2,3,\dots$, where $N$ may change with $n$. We do this by an induction on $n$. The equation above says that 
\begin{equation}\label{equ:Fderi}
 \frac{d F^t_n }{dz}
=
- \frac{d f^t_n}{d t}
+
\sum_{j=1}^{n-1}
[f^t_j, F^t_{n-j}]
=
- \frac{d f^t_n}{d t}
+
\sum_{j=2}^{n-1}
[f^t_j, F^t_{n-j}]
+
[f^t_1, F^t_{n-1}]
.
\end{equation}
Note that $\frac{d f^t_n}{d t}=O(|\log |z||^N)$, since $f^t_n=O(|\log |z||^N)$ as we saw above and $f^t_n$ is a polynomial in $t$. Furthermore, by the induction hypothesis 
\[
\sum_{j=2}^{n-2}
[f^t_j, F^t_{n-j}]= O(|\log |z||^N).
\]
Similarly, $F^t_{1}(z) \propto \log( z)\,  t_{12}$ and hence $[f^t_{n-1}, F^t_{1}]= O(|\log |z||^N)$

Note however that $f^t_1(z)= \frac{t_{12}}{z} + \frac{t_{23}}{1-z}$, so that there is potentially a pole on the right hand side. Fortunately, $[t_{12}, F^t_{n-1}(0)]=0$ using the explicit expression for $F^t_{n-1}(0)$ that was derived in Proposition \ref{p-asingular}. But by the induction hypothesis again $F^t_{n-1}(z)- F^t_{n-1}(0)=O(z|\log z|^N)$, so that 
\[
[f^t_1, F^t_{n-1}] =  O(|\log |z||^N).
\]

Hence the right hand side of \eqref{equ:Fderi} is of order $O(|\log |z||^N)$. Hence, by integration we see that $F^t(z)=O(|\log |z||^N)$.
This shows the first identity of \eqref{equ:fastdecay}. The second is shown by an analogous argument, interchanging the roles of $0$ and $1$.
\hfill\qed

\section{The proofs of the main Theorems}\label{s-9}
In the present section, we re-collect the results from Sections~\ref{s-2} to~\ref{s-8} and cast them together in the form of Theorems~\ref{thm:etingof} and~\ref{thm:stablefamily}.

\subsection{Proof of Theorem \ref{thm:etingof}}\label{ss-9-1}
The family of Drinfel{\cprime}d associators $\Phi^t=\Phi^t_\mathrm{reg}$ is constructed in Definition~\ref{d-reg-assoc}, cf.~also Corollary~\ref{c-drinfeld}.
Similarly, the family of $\grt_1$ elements $\tau^t$ is constructed in Corollary~\ref{c-a2}.
By applying the orientation-reversing involution $z\mapsto \overline z$ componentwise on the defining integrals, it follows that $x^t$ is concentrated in odd degrees.
The $t$-dependence of $x^t$ is described in Lemma \ref{l-ascaling}.
 
Identity~\eqref{equ:Phithor} is the infinitesimal form of Identity~\eqref{equ:Phitexpl}.

It is also clear from the definitions that $\Phi^0$, $\Phi^{\frac 1 2}$ and $\Phi^1$ are the KZ associator, the AT associator and the anti-KZ associator respectively.

It immediately follows that the weak form of P.~Etingof's Conjecture~\ref{conj:etingof} holds true. 
It remains to be shown that, on the other hand, the strong form of the conjecture does not hold true.

Let us consider the gauge transformation $H^t$ which relates $\nabla^0_k$ and $\nabla^t_k$. Similarly to the formula at the beginning of Subsection~\ref{ss-8-2}, we find
\[
H^t=\mathcal P\exp\!\left(\int_0^t a^s_kds\right).
\]
For $\varepsilon>0$ sufficiently small, and borrowing notation from Subsection~\ref{ss-8-2}, we find
\[
\Phi^t_\varepsilon=H^t(B_\varepsilon)\ \Phi^0_\varepsilon\ H^t(A_\varepsilon)^{-1}.
\]

By combining Proposition~\ref{p-asingular} and Lemma~\ref{l-ascaling} as in the proof of Corollary~\ref{c-drinfeld}, we find
\[
\Phi^t_\mathrm{reg}=\mathcal P\exp\!\left(\int_0^t\tau^sds\right)\cdot \Phi_\mathrm{KZ},
\]
and the path-ordered exponential defines a family of elements of $\GRT_1$.
It follows that
\begin{align*}
\Phi_\mathrm{AT}=\mathcal P\exp\!\left(\int_0^{\frac 1 2}\tau^sds\right)\cdot \Phi_\mathrm{KZ},\quad \Phi_{\overline{\mathrm{KZ}}}=\mathcal P\exp\!\left(\int_{\frac 1 2}^1\tau^sds\right)\cdot \Phi_\mathrm{AT}.
\end{align*}
We want to show that $a\neq b$, where $a$ and $b$ are the $\GRT_1$ elements
\begin{align*}
a=\mathcal P\exp\!\left(\int_0^{\frac 1 2}\tau^sds\right),\quad b=\mathcal P\exp\!\left(\int_{\frac 1 2}^1\tau^sds\right).
\end{align*}
In fact, by F.~Brown's result, the Lie algebra generated by all $\sigma_{2j+1}'$ is free, thus both $a$ and $b$ may be understood as elements of the free (completed) associative algebra $\C\langle\sigma_3', \sigma_5',\dots \rangle$.
The coefficients of the product $\sigma_3'\sigma_5'$ in $a$ and $b$ are
\begin{align*}
c_a&=\frac 1 2
\int_0^{\frac 1 2}[s_1(1-s_1)]^4\left(\int_0^{s_1}[s_2(1-s_2)]^2 ds_2\right)ds_1=\frac{1199}{309657600}
\\
c_b 
&=
\frac 1 2
\int_{\frac 1 2}^1[s_1(1-s_1)]^4\left(\int_{\frac 1 2}^{s_1}[s_2(1-s_2)]^2 ds_2\right)ds_1
=
\frac{283}{103219200}
\end{align*}
Since $c_a\neq c_b$ it follows that $a\neq b$, whence the claim.
\hfill\qed

\subsection{Proof of Theorem \ref{thm:stablefamily}}\label{ss-9-2}
In order to prove Theorem~\ref{thm:stablefamily}, we need to show several sub-statements. 

First, the family of stable formality morphisms $\mU^t$ is constructed in section~\ref{s-5}, and it is clear that $\mU^{\frac 1 2}$ and $\mU^0$ are the Kontsevich stable formality morphism and the Kontsevich stable formality morphism constructed by means of the logarithmic propagator respectively.

The family of graph cocycles $x^t$ is defined in~\eqref{equ:xtdef}. By a reflection argument (using the reflection $z\mapsto \overline z$) one can see that components of $x^t$ with an odd number of vertices vanish. Hence we may decompose 
\[
 x^t = \sum_{j\geq 1} x_{2j+1}^t
\]
where $x_{2j+1}^t$ is a ($t$-dependent) linear combination of graphs with $2j+2$ vertices.
The $t$-dependence of $x^t$ is explained in Subsection~\ref{ss-6-2}. It follows that $x_{2j+1}^t=(t(1-t))^{2j}x_{2j+1}$ for some graph cocycles $x_{2j+1}\in \GC$.
It is shown in Proposition \ref{p-nuingrt} that the family $x^t$ indeed corresponds to the family of $\grt_1$ elements $\tau^t$ from Theorem~\ref{thm:etingof}. It follows that $x_{2j+1}\in \GC$ corresponds to $\tau_{2j+1}\in \grt_1$.
 
Recall that the homotopy class of the Kontsevich stable formality morphism corresponds by Definition (or by~\cite{Will-2}) to the Drinfel{\cprime}d associator $\Phi_\mathrm{AT}$ under the identification of the torsor of Drinfel{\cprime}d associators and that of homotopy classes of stable formality morphisms. 
Hence, it follows from equations \eqref{equ:Phithor} and \eqref{equ:stablehomotopy} that the homotopy class of the stable formality morphism $\mU^t$ indeed corresponds to the Drinfel{\cprime}d associator $\Phi^t$ from Theorem \ref{thm:etingof} for all $t$. 
In particular, $\mU^0$ corresponds to $\Phi^0=\Phi_\mathrm{KZ}$.

The only remaining statement of Theorem~\ref{thm:stablefamily} to be proved is the property of being divergence-free of the graph cocycle $x^t$.
Recall from Identity~\eqref{equ:xtdef} that it is given by a sum of graphs formula.
By unraveling the adjoint action of $\Gamma_{\lcirclearrowdown}$ on $\GC$, $x^t$ is divergence-free, if   
\[
\sum_{e\in E(\Gamma)}(-1)^{e-1}c^t_{\Gamma\smallsetminus \{e\}} = 0
\]
where $\Gamma\smallsetminus \{e\}$ is the graph $\Gamma$ with the edge $e$ removed, for $\Gamma$ a graph in $\GC$.

Recall that $c^t_{\Gamma}$ is given by an integral over the configuration space $\underline C_n/S^1$ of the differential form $\widetilde\beta_\Gamma^t$, see~\eqref{equ:tildebetadef}). 
On the other hand, 
\[
\sum_{\hat e\in E(\Gamma)} (-1)^{\hat e-1}\widetilde\beta_{\Gamma\setminus \{\hat e\}}^t = 0
\]
By antisymmetry under exchange of the $\hat e$ above with the $e$ appearing in \eqref{equ:tildebetadef}. 
\hfill \qed

\appendix

\section{Associators in general}\label{app:assoc}
We briefly recall the theory of Drinfel{\cprime}d associators following~\cite{AT-2}{Section 9} (but slightly changing conventions): for this purpose, we recall the simplicial and coproduct maps~\eqref{eq-cob-tder-1} and~\eqref{eq-cob-tder-2}.
Since these are morphisms of Lie algebras, they naturally extend to maps from $\SAut_k$ to $\SAut_{k+1}$, for $k\geq 2$; they also obviously descend to maps from $\mathsf T_k$ to $\mathsf T_{k+1}$.

An element $\Phi$ of $\mathsf T_3$ is called a Drinfel{\cprime}d associator, if it satisfies the anti-symmetry, hexagon and pentagon equations in the following form:
\begin{align}
\label{eq:antisymm}\Phi^{1,2,3}\ \Phi^{3,2,1}&=1\\
\label{eq:hexagon}e^{\frac{t_{13}+t_{23}}2}&=\Phi^{3,1,2}\ e^{\frac{t_{13}}2}\ \left(\Phi^{1,3,2}\right)^{-1}\ e^{\frac{t_{23}}2}\ \Phi\\ 
\label{eq:pentagon}\Phi^{1,2,34}\ \Phi^{12,3,4}&=\Phi^{2,3,4}\ \Phi^{1,23,4}\ \Phi^{1,2,3}.
\end{align}
Identity~\eqref{eq:hexagon} is known as the hexagon identity.
Further, Identity~\eqref{eq:pentagon} is called the pentagon identity.


The (non-empty) set of Drinfel{\cprime}d associators is acted on freely and transitively by the group $\GRT_1$, the pro-unipotent group associated with the pro-nilpotent Lie algebra $\grt_1$: the $\GRT_1$-action has been first defined in~\cite{Dr}*{Section 5}, to which we refer for explicit formul\ae.
The formul\ae\ in~\cite{Dr}*{Section 5} are quite complicated: on the other hand, it has been proved in~\cite{AT-2}*{Proposition 9.5} that the action of $\GRT_1$ on the set of Drinfel{\cprime}d associators simplifies considerably and is reduced to an action by Drinfel{\cprime}d twists, which we now shortly describe.
Namely, let us consider $F$ in $\GRT_1\subseteq \SAut_2$: then, the action of $F$ on the set of Drinfel{\cprime}d associators described in~\cite{Dr}*{Section 5} can be re-written as  
\[
F^{2,3}\ F^{1,23}\ \Phi\ \left(F^{12,3}\right)^{-1}\ \left(F^{1,2}\right)^{-1},
\] 
for a general Drinfel{\cprime}d associator $\Phi$ (in the terminology of~\cite{Dr}*{Section 1}, a Drinfel{\cprime}d twist of $\Phi$ with respect to $F$).

At the infinitesimal level, we get a $\grt_1$-action on the set of Drinfel{\cprime}d associators, which can be easily deduced from the previous formula: namely, for a Drinfel{\cprime}d associator $\Phi$,
\[
u\cdot \Phi=\left(u^{2,3}+u^{1,23}\right)\ \Phi-\Phi\ \left(u^{1,2}+u^{12,3}\right),\ u\in\grt_1.
\]
From the previous $\grt_1$-action on Drinfel{\cprime}d associators, we have deduced in Subsection~\ref{ss-8-2} the expression for $\Phi^t_\mathrm{reg}$ in terms of the AT associator $\Phi_\mathrm{AT}$.

\section{The Knizhnik--Zamolodchikov associator}\label{app:KZ}
As previously remarked, the set of Drinfel{\cprime}d associators is non-empty: this has been proved in~\cite{Dr}*{Section 2}, where the author constructs an explicit Drinfel{\cprime}d associator over $\mathbb C$, the Knizhnik--Zamolodchikov (shortly, from now on, KZ) associator, whose construction we now briefly recall. 

Let $X$, $Y$ be formal variables. 
Consider the ordinary differential equation 
\begin{equation}\label{eq-KZ-3}
\frac{df}{dz}=\frac{1}{2\pi i}\left(\frac{X}{z}f+\frac{Y}{1-z}f\right)
\end{equation}
for a function $f$ on $\Conf_1(\bbC \setminus\{0,1\})=\mathbb C\smallsetminus\{0,1\}$ with values in the completed free algebra $\mathbb C\langle\!\!\langle X,Y\rangle\!\!\rangle$ generated by $X$, $Y$: $\mathbb C\langle\!\!\langle X,Y\rangle\!\!\rangle$ becomes a topological Hopf algebra by declaring the generators $X$, $Y$ to be primitive, {\em i.e.} the coproduct and the antipode on it are uniquely determined by 
\[
\Delta(X)=X\widehat\otimes 1+1\widehat\otimes X,\ \Delta(Y)=Y\widehat\otimes 1+1\widehat\otimes Y,\ S(X)=-X,\ S(Y)=-Y.
\]
It is well-known that~\eqref{eq-KZ-3} corresponds to the KZ equations in three variables.
Here, $\widehat\otimes$ denotes the topological tensor product over the base field $\bbC$.

We denote by $f_0$, resp.\ $f_1$, be the unique solution of~\eqref{eq-KZ-3} such that 
\begin{align*}
f_0(z)&=\widetilde f_0(z)z^X
&
f_1(z)&=\widetilde f_1(z)(1-z)^Y 
\end{align*}
where $\widetilde f_0$, resp. $\widetilde f_1$, may be extended continuously to $z=0$, resp. $z=1$, and both take the value $1$ there. 
Further, $z^X$ is to be understood as $z^X := \exp(X\log z )$ and similarly $(1-z)^Y:=\exp(Y\log(1- z) )$. 
One may easily check that the function
\[
f_1^{-1}(z) f_0(z)
\]
is constant in $z$. 

The KZ associator $\Phi_\mathrm{KZ}(X,Y)$ is defined to be the constant value of this function.
It is a group-like element of $\mathbb C\langle\!\!\langle X,Y\rangle\!\!\rangle$. 
We set $\Phi_\mathrm{KZ}=\Phi_\mathrm{KZ}(t_{12},t_{23})$ in $T_3$.
The latter element satisfies Identities~\eqref{eq:antisymm},~\eqref{eq:hexagon} and~\eqref{eq:pentagon}.

To connect with the other associators we construct in the paper, let us also sketch two alternative descriptions of $\Phi_\mathrm{KZ}$ due to Le and Murakami~\cite{Le-Mu}*{Appendix A}. 

Define the KZ connection 
\[
\nabla_\mathrm{KZ}:=d-\frac{1}{2\pi i} \left( \frac{dz}{z}\ X +\frac{dz}{1-z}\ Y \right)
\]
on $\Conf_1(\bbC\smallsetminus\{0,1\})$ with values in $\Lie(X,Y)$.

Let $\Phi_\varepsilon$ be the monodromy of the connection $\nabla_\mathrm{KZ}$ between the points $z=\varepsilon$ and $z=1-\varepsilon$, for $\varepsilon>0$ small enough. 
Then, $\Phi_\varepsilon$ has an asymptotic expansion in $\log(\varepsilon)$ as $\varepsilon\to 0$, and furthermore
\[
\lim_{\varepsilon \to 0}
\varepsilon^{-Y/2\pi i}\ \Phi_\varepsilon\ \varepsilon^{X/2\pi i} = \Phi_\mathrm{KZ}.
\]
Alternatively, $\Phi_\mathrm{KZ}$ may be obtained from the asymptotic expansion of $\Phi_\varepsilon$ by formally setting $\varepsilon=\log(\varepsilon)=0$.

More precisely, for $k\geq 2$, we define the flat KZ connection on the trivial principal $\mathsf T_k$ -bundle over $\Conf_k$ {\em via}
\[
\nabla_{\mathrm{KZ},k}=d-\frac{1}{2\pi i}\sum_{1\leq i<j\leq k}\frac{dz_i-dz_j}{z_i-z_j}\ t_{ij}.
\]
The KZ connection is obviously basic with respect to the action of $\bbC$ by componentwise complex translations and $\bbC^\times$-invariant: in particular, let us consider the $\mathrm{KZ}_3$-system of PDE
\[
\nabla_{\mathrm{KZ},3}(f)=0,
\]
for $f$ a function on $\Conf_3$ with values in the Universal Enveloping Algebra $\mathrm U(\mathfrak t_3)$ of $\mathfrak t_3$.
By means of the fact that $\nabla_{\mathrm{KZ},k}$ is $\bbC$-basic and $\bbC^\times$-invariant, the above system can be equivalently re-written as 
\[
\frac{df}{dz}=\frac{1}{2\pi i}\left(\frac{dz}z\ t_{12}+\frac{dz}{z-1}\ t_{23}\right)f,
\]
for a function $f=f(z)$ on $\Conf_1(\bbC\smallsetminus\{0,1\})$ (observe that $f$ does not depend on $\overline z$).

The anti-KZ associator $\Phi_{\overline{\mathrm{KZ}}}$ is defined {\em via} 
\[
\Phi_{\overline{\mathrm{KZ}}}=\Phi_{\overline{\mathrm{KZ}}}(X,Y)=\Phi_\mathrm{KZ}(-X,-Y),
\]
for $X$, $Y$ as at the beginning.
In particular, the anti-KZ associator $\Phi_{\overline{\mathrm{KZ}}}$ is obtained from the equation 
\[
\frac{df}{dz}=-\frac{1}{2\pi i}\left(\frac{X}{z}f+\frac{Y}{1-z}f\right),
\]
for $f$ as above.
It is not difficult to verify, by computations similar to the ones above, that the previous equation is the reduction to an ordinary differential equation of the integrable system of three partial differential equations specified by the complex conjugate of the KZ connection $\nabla_{\mathrm{KZ},3}$, by the involution $z\mapsto \overline z$.
This involution does not alter the relevant computations for the associator, because it is constant, though it is defined by suitably multiplying two special asymptotic solutions of an ordinary differential equation of Fuchs type.

\section{The regularized Stokes' Theorem on \texorpdfstring{$\overline C_{k+n}^f$}{C(k+n)f}}
\label{app:stokesbundle}
The goal of this Appendix is to show that the bundle $\overline C_{k+n}^f\to \underline C_k$ from section \ref{ss-8-0} fits into the framework of the regularized Stokes Theorem \ref{thm:regstokes}, i.~e., that one has local charts and locally defined free torus actions.
Furthermore, we will calculate the regularizations of the singular differential forms appearing in the proof of Proposition \ref{p-flat-gauge-AT-bound}.

To distinguish the $k$ points determining the configuration in the base and the other $n$ points, we will call the former points "external" and the latter points "internal".

\subsection{Boundary strata, charts and local coordinates}
Recall that the bundle $\overline C_{k+n}^f\to \underline C_k$ (for $k\geq 2$) is defined as the pullback of the bundle $\overline C_{k+n}\to \overline C_k$. The fiber-wise boundary strata are hence labelled by trees with leaf set $[n+k]$, as are the boundary strata of $\overline C_{k+n}$. Note however that not all such trees will appear, since some trees denote boundary strata which lie in the fibers over the boundary of $\overline C_{k}$.
The trees that appear are such that any pair of external leaves has the same least common ancestor in the tree. This least common ancestor is hence distinguished, and we denote it by $*$.
We will furthermore call the ancestor vertices of $*$ (including $*$ itself) the \emph{infinite vertices}.

One may define local coordinates on $\overline C_{k+n}^f$ for each tree $i$.
Let $B$ be a vertex corresponding to some subset of points. Then we define its center of mass 
\[
\zeta_B = 
\begin{cases}
\frac 1 {|B|} \sum_{j\in B} z_j & \text{if $B^{ext}=\emptyset$} \\
\frac 1 {|B^{ext}|} \sum_{j\in B^{ext}} z_j & \text{otherwise} 
\end{cases}
\]
where $B^{ext}\subset B$ is the set of external points in $B$.
In other words, when computing the center of mass we assign the external points infinite mass.

Next we will introduce local coordinates, given a tree $i$. We will attach a parameter $r_B$ to each (non-leaf-)vertex $B\neq *$ of $i$. 
Suppose first that the tree $i$ has a single vertex $B$ (apart from the root and leaves).
There are three possible cases to consider: (i) $B$ does not contain external points, (ii) $B$ contains exactly one external point or (iii) $B$ contains all external points. Otherwise, not any pair of external points would have the same least common ancestor in the tree.

Suppose that $B$ does not contain any external points, and consider some $b\in B$.
As in \cite{ARTW} we will will write
\[
z_b = \zeta_B + r_B z_b^{(1)}
\]
where the new coordinates $z_b^{(1)}$ are normalized such that 
\begin{align*}
\sum_{b\in B} z_b^{(1)}=0\ 
\text{and}\
\sum_{b\in B} |z_b^{(1)}|^2=1.
\end{align*}
At $r_B=0$ one obtains the interior of the boundary component
\[
\overline C_{k+n-|B|}^f \times \overline C_{|B|}.
\]
Next suppose that $B$ contains a single external point $c$. Then by definition $\zeta_B=z_c$. Let again $b\in B$ be such that $b\neq c$. The we again use a parameterization
\[
z_b = \zeta_B + r_B z_b^{(1)}
\]
as above, but now the coordinates $z_\cdot^{(1)}$ are normalized such that
\begin{align*}
\sum_{b\in B\setminus \{c\}} |z_b^{(1)}|^2&=1
\end{align*}
without a second condition.

Finally assume that $B$ contains all external points, i.~e., $*=B$. Call the root vertex $R$, i.~e., $R=[n+k]$. 
In this case we attach a parameter $r_R$ to $R$. Suppose $c$ is a vertex not in $B$. Then we use coordinates such that
\[
z_c = \frac 1 {r_R} z_c^{(1)}
\]
where the $z_c^{(1)}$ are normalized such that 
\begin{align*}
\sum_{c\in R\setminus B} |z_c^{(1)}|^2&=1
\end{align*}
The boundary is reached at $r_R=0$.

In general, we assign
\begin{itemize}
\item To each non-leaf vertex $B_j\neq *$ a parameter $r_{B_j}$ and coordinates $z_B^{(j)}$ where $B$ ranges over the set of direct children $\mathfrak B_j$ of $B_j$ not containing external points. These coordinates are normalized such that $\sum_{B\in \mathfrak B_j}  |B| |z_B^{(j)}|^2=1$ and, if $B_j$ does not contain external vertices then in addition $\sum_{B\in \mathfrak B_j} |B| z_B^{(j)}=0$. If $B$ is a child of $B_j$ containing external points then we will set $z_{B}^{(j)}=0$ for notational convenience.
\item To the vertex $B_j=*$ we assign coordinates $z_B^{(j)}$ for $B$ ranging over the direct children normalized such that 
\[
\sum_{B\in \mathfrak B_j'} |B| z_B^{(j)}=0.
\]
where $\mathfrak B_j'$ is the set of direct children containing external points.
\end{itemize}

Next consider some point $b$, corresponding to a leaf of the tree. There is a unique path $B_1,B_2,\cdots,B_r$ from the vertex $X$ to $b$. We suppose that $B_k$ is the highest vertex, i.~e.,
\[
B_1\subset B_2\subset \cdots \subset B_k \supset B_{k+1} \supset \cdots \supset B_r.
\]
Then we use coordinates such that
\[
z_b = \frac 1 {r_{B_1} \cdots r_{B_{k}}}  \left( z^{(k)}_{B_{k+1}}  + r_{k+1}\left( z^{(k+1)}_{B_{k+2}}  + r_{B_{k+2}}\left(\cdots  + z_b^{r}\right) \right) \right).
\]

\subsection{The torus actions}
To each tree $i$ as above we will assign a local action of a torus $T_i\cong (S^1)^m$, where $m$ is the number of non-leaf vertices minus one. We will assign commuting $S^1$ actions to each non-leaf vertex $B_j\neq *$.
Concretely, the action is merely by rotating the coordinates $z_{B}^{(j)}$ assigned at that vertex. Clearly the normalization conditions are untouched and the various $S^1$ actions commute since they operate on disjoint sets of variables.

Note also that these actions are bundle actions, i.~e., the configuration of all external points is left unaltered.

\subsection{Charts and a partition of unity}
In \cite{ARTW} the next step is to define subsets $U_i$ satisfying the conditions of section \ref{sec:regstokes}, and an invariant partition of unity.
We leave it to the reader to verify that the construction of these data from \cite{ARTW} goes through in the present slightly modified setting, and hence the regularized Stokes' Theorem can be applied.

\subsection{Proof of Lemma \ref{lem:a12exists}}\label{app:thetaextendsproof}
In this subsection we will show Lemma \ref{lem:a12exists}, asserting that the fiberwise top degree parts of the differential forms 
$$
\widetilde\theta^{\frac 1 2}_\Gamma
=
\sum_{e\in E(\Gamma)} (-1)^{e-1} \frac i \pi 
\log|z_{s(e)}- z_{t(e)}|
\prod_{e'\neq e\in E(\Gamma)} \frac 1 {2\pi} d\arg(z_{s(e')}-z_{t(e')})
$$
extend to the boundary of the fiber of $\overline C_{k,n}^f$.
Note that this is true for all $d\arg(\dots)$-factors, as the only potentially singular term is contributed by the logarithm.
We claim that in fact the singularity is cancelled.
Consider a tree $i$, a chart $U_i$ and a vertex $B$ of the tree. We need to check that there is no singularity in the corresponding coordinate $r_B$ as $r_B\to 0$.
Consider first the case of $B$ being a non-infinite vertex, i.~e., it describes a subset of collapsing points.
Consider three types of edges:
\begin{itemize}
 \item Edges with both endpoints in $B$. They may contribute a singularity
 \[
  \log r_B +(\text{terms regular in $r_B$}).
 \]
 Edges of this type are the only ones that may contribute to the singularity, 
 \item Edges with one endpoint (say $w$) in $B$ and one (say $z$) in the complement contribute a term
 \[
  d\arg(z-w)= d\arg(z-\zeta_B) + r_B(\cdots).
 \]
 \item Edges with both endpoints in the complement of $B$ are not important for the discussion.
\end{itemize}
We call the above subsets of edges (in this order) $E_3(\Gamma)$, $E_2(\Gamma)$, $E_1(\Gamma)$.
We want to show that the fiberwise top degree part of the form $\widetilde\theta^{\frac 1 2}_\Gamma$ is regular, or equivalently that $\prod_{B'}\iota_{v_B'}\widetilde\theta^{\frac 1 2}_\Gamma$ is regular. Collecting potentially singular terms we compute
\begin{align*}
 \iota_{v_B}\widetilde\theta^{\frac 1 2}_\Gamma
 &=
 \iota_{v_B}\left( \sum_{e\in E(\Gamma)} (-1)^{e-1} \frac i \pi 
\log(r_B)
\prod_{e'\in E_1(\Gamma)} \frac 1 {2\pi} d\arg(z_{s(e')}-z_{t(e')})
\prod_{(z,w)=e'\in E_2(\Gamma)} \frac 1 {2\pi} d\arg(z-\zeta_B)
\right.\\&\quad\quad\quad\quad\quad\quad\left.
\prod_{e'\neq e\in E_3(\Gamma)} \frac 1 {2\pi} d\arg(z_{s(e')}-z_{t(e')})\right)
+(\text{terms regular in $r_B$})
\\
&=
 \sum_{\substack{e, \hat e\in E(\Gamma)\\ e\neq \hat e}} \psgn{e,\hat e} \frac i \pi 
\log(r_B)
\prod_{e'\in E_1(\Gamma)} \frac 1 {2\pi} d\arg(z_{s(e')}-z_{t(e')})
\prod_{(z,w)=e'\in E_2(\Gamma)} \frac 1 {2\pi} d\arg(z-\zeta_B)
 \\ & \quad\quad\quad\quad\quad\quad 
\prod_{e'\in E_3(\Gamma)\smallsetminus\{e,\hat{e}\}} \frac 1 {2\pi} d\arg(z_{s(e')}-z_{t(e')})
+(\text{terms regular in $r_B$})
\\
&= 0 + (\text{terms regular in $r_B$})
\end{align*}
where we used the antisymmetry of the summand under interchange of vertices $e$ and $\hat e$.
This shows that there is no singularity in $r_B$ for $B$ non-infinite.

Consider now $B$ infinite. The corresponding boundary stratum is obtained as the (necessarily internal) points in the complement of $B$ tend to infinity, while the points in $B$ stay at ``finite distance'' around their center of mass $\zeta_B$. The $S^1$ action corresponding to $B$ is by rotating the points in the complement of $B$ around $\zeta_B$.

Again, we distinguish three kinds of edges.
\begin{itemize}
 \item Edges with both endpoints in $B$. They cannot contribute a singularity as the locations of the endpoints are independent of $r_B$.
 \item Edges with one endpoint (say $w$) in $B$ and one (say $z=:\zeta_B+\frac 1 r_B Z$) in the complement contribute a term
 \[
  d\arg(z-w)=d\arg(\frac 1 {r_B}(Z  +r_B(\zeta_B-w) )) =d\arg(Z) + r_B(\cdots).
 \]
 \item Edges with both endpoints in the complement of $B$ are the only terms who can contribute a singularity of the formal
  \[
   -\log(r_B) +(\text{terms regular in $r_B$}).
  \]
\end{itemize}
One sees that the contributions are similar to those for non-infinite $B$, except that the roles of $E_1$ and $E_3$ are interchanged, and except for an unimportant minus sign.
Hence we see by the same argument as before that the form $\iota_{v_B}\widetilde\theta^{\frac 1 2}_\Gamma$ is regular in $r_B$.

We conclude that $\prod_{B'}\iota_{v_B'}\widetilde\theta^{\frac 1 2}_\Gamma$ is regular in all variables $r_B'$.
\hfill\qed

\subsection{Proof of Lemma \ref{lem:extendstoX}}\label{app:extendstoXproof}
Our next goal is to show Lemma \ref{lem:extendstoX}, i.~e., that for $\Gamma\in \sder_3$ a graph with $n\geq 1$ internal vertices the form $\hat\theta_\Gamma$ extends to the compactification $X_n$, using the notation of section \ref{ss-8-3}.
In fact, Lemma \ref{lem:extendstoX} is almost a special case of Lemma \ref{lem:a12exists} shown in the previous subsection. The only new feature is that the base space is now compactified, i.~e. the point $x$ may collapse to either 0 or 1.

Since $X_n\subset \overline C_{n+3}$ we may again use the familiar system of charts $U_i$ on $\overline C_{n+3}$ from section \ref{sec:3-4}. Fix some nested family $i$ and a subset $B\in i$. Our goal is to show that the differential form $\widetilde\theta_\Gamma^{\frac 1 2}$ is regular in the coordinate $r_B$. In fact, the only non-regular term in the definition of $\widetilde\theta_\Gamma^{\frac 1 2}$ is the function $\log|z-w|$ associated to some edge. As in the previous subsection we consider separately the cases of $B$ being an ``infinite'' vertex or not.
If $B$ is infinite or if $B$ contains at most one external point, then the argument of the previous subsection shows that there is no singularity in $r_B$. The only new case is that $B$ contains exactly two external vertices, necessarily the vertex at $x$ and the one at $0$ or $1$. Without loss of generality let us assume that $B$ contains the external vertices at $x$ and $0$.
We will distinguish three types of edges.

\begin{itemize}
  \item Edges with no endpoints in $B$. They do not contribute to the singularity.
  
  \item Edges with with exactly one endpoint in $B$ and one endpoint, say $z$ in the complement. Their contribution is 
  \[
   \frac 1 {2\pi} d\arg(z-\zeta_B) +r_B(\cdots).
  \]
  \item Edges with both endpoints in $B$. One of these edges can contribute a factor 
  \[
   \frac 1 {\pi i} \log(r_B) + (\text{terms regular in $r_B$})
  \]
  to the singularity.
 \end{itemize}
We call the corresponding subsets of edges $E_j$, $j=1,2,3$.
Let first collect the potentially singular terms in $\widetilde \theta_\Gamma$, omitting an unimportant prefactor:
\[
\widetilde \theta_\Gamma^{\frac 1 2} = 
 \log(r_B)\sum_{e'\in E_3} (-1)^{e'-1}
 \prod_{e\in E_1} d\arg(z_{s(e)}-z_{t(e)})
 \prod_{e\in E_2} d\arg(z_{s(e)}-\zeta_B)
 \prod_{e\in E_3\smallsetminus\{e'\}}d\arg(z_{s(e)}-z_{t(e)}) + (\text{terms regular in $r_B$}).
\]
Here we understand that for $e\in E_2$ the endpoint $s(e)$ is the one not in $B$.
Note that there are $|E_1|+|E_2|$ one-form factors which depend only on the configuration of the points not in $B$ and $\zeta_B=\frac x 2$.
Note furthermore that the forms associated to these one-form factors are basic under scaling. Hence by degree reasons the singular part vanishes if $|E_1|+|E_2|> 2k$, where $k$ is the number of internal vertices in the complement of $B$.
However, since the graph $\Gamma$ is an internally trivalent tree $|E_1|+|E_2|\geq 2k+1$, unless there are no internal vertices in the complement of $B$ and $E_1=E_2=\emptyset$. But this may happen only if the third external vertex in $\Gamma$ has valence 0.

Next let us analyze the potential singularity of $\hat \theta_\Gamma$ as defined in \eqref{equ:hattheta1}, \eqref{equ:hattheta2}.
If $\Gamma$ is such that both the first and third external vertex have valence $\geq 1$ then $\hat \theta_\Gamma=\widetilde \theta_\Gamma$, and by the above discussion there is no singularity. 

Suppose the first external vertex in $\Gamma$ has valence 0. Then, for $B$ as considered above, the additional term in \eqref{equ:hattheta2} does not contain any divergent factor, hence again $\hat \theta_\Gamma$ does not have a singularity in $r_B$.

Finally suppose that the third external vertex in $\Gamma$ has valence 0. Then the term subtracted in \eqref{equ:hattheta1} contributes to the singularity through the factor 
\[
\log(x)= \log(r_B)+(\text{terms regular in $r_B$}). 
\]
Concretely, the subtracted term has the following form
\begin{multline*}
\sum_{e'\in E} \frac{(-1)^{e'-1}}{\pi i} \log(x) \prod_{e\in E(\Gamma)\smallsetminus \{e'\}} \omega_e
\\=
 \log(r_B) \sum_{e'\in E(\Gamma)} (-1)^{e'-1}
 \prod_{e\in E_1\smallsetminus\{e'\}} d\arg(z_{s(e)}-z_{t(e)})
 \prod_{e\in E_2\smallsetminus\{e'\}} d\arg(z_{s(e)}-\zeta_B)
 \prod_{e\in E_3\smallsetminus\{e'\}}d\arg(z_{s(e)}-z_{t(e)}) 
 \\ + (\text{terms regular in $r_B$}).
\end{multline*}

Note that since $\Gamma$ is an internally trivalent tree with no edge connecting to the third external vertex, the product over $E_2$ must necessarliy contain at least one pair of identical forms, unless $E_2=\emptyset$. Hence the singular term vanishes, unless $E_2=\emptyset$. In that case $E_1$ is necessarily empty as well and $B$ contains all internal vertices.
But then the singular term just kills the singularity in $\widetilde \theta_\Gamma$ so that $\hat \theta_\Gamma$ has no singularity in $r_B$.
\hfill\qed

\subsection{Regularizability of the weight forms}\label{app:8-10proof}
Our final goal in this Appendix is to prove Proposition \ref{prop:8-10} about the regularizability of the weight forms occurring in the proof of Proposition \ref{p-flat-gauge-AT-bound}.
The arguments are mostly copies of those leading to Proposition \ref{prop:regularizable} and Theorem \ref{thm:factoring}. Some steps even simplify since in our case $\Gamma$ is an internally trivalent tree, as opposed to a general graph. First, as in Proposition \ref{prop:splitting} consider a chart corresponding to a tree $i$ and a vertex $B$ thereof. We denote by $v_B$ the vector field generating the $S^1$ action. Suppose first that $B$ is a non-infinite vertex, i.~e., not an ancestor of vertex $*$. Then, by the same arguments as in the proof of Proposition \ref{prop:splitting} the weight form $\tilde \theta_\Gamma^t$ has an expansion
\[
\tilde \theta_\Gamma^t =  \frac{dr_B}{r_B}\wedge \tilde\alpha^t +  \log(r_B) \hat\alpha^t + (\text{terms regular in $r_B$})
\]
for some forms $\tilde\alpha^t$ and  $ \hat\alpha^t $ which are $v_B$-basic. In fact, as in the proof of Proposition \ref{prop:splitting} one may give explicit formulas for $\tilde\alpha^t$ and  $ \hat\alpha^t $.  

If $B$ is an infinite vertex, the argument is similar.
 
Given a splitting formula for the weight forms as above, the rest of the arguments leading to Theorem \ref{thm:factoring}. go through unchanged, if applied to the form $\psi \theta_\Gamma^t$, for $\psi$ a compactly supported test $2k-3$-form on $\underline C_k$ as in the proof of Proposition \ref{p-flat-gauge-AT-bound}.
We find that the form $\psi \theta_\Gamma^t$ is regularizable, and that the regularizations at the boundary strata satisfy a formula analogus to that in Theorem \ref{thm:factoring}.
However, note that the graph $\Gamma$ is an internally trivalent tree.
Hence, a subgraph $\Gamma$ corresponding to some $B$ as above has to be an at most internally trivalent forest.
The associated form has too low degree unless the subgraph consists of exactly two vertices. Hence we have shown Proposition \eqref{prop:8-10}.

\section{Explicit computation of the simplest integral weight for the graph cocycle \texorpdfstring{$x^t$}{xt}}\label{app-coc}
It is well-known that the simplest, non-trivial cocycle of degree $0$ in $\GC$ is the tetrahedron graph depicted in Figure~\ref{fig:tetrahedron} (left).
We want to compute the corresponding integral weight~\eqref{equ:xtweightdef} directly as a consistency check.

By the computations in Subsection~\ref{ss-6-2}, we know that, up to polynomials in $t$, $x^t=x^{\frac 1 2}$. 
The result is the following:
\begin{Prop}\label{p-w-tetra}
The integral weight~\eqref{equ:xtweightdef} of the tetrahedron graph in $x^{\frac 1 2}$ is a non-zero rational multiple of 
\[
\frac{\zeta(3)}{\pi^3}.
\]
(The rational pre-factor can be recovered {\em e.~g.} from the calculation below.)
\end{Prop}
To compute the integral weight~\eqref{equ:xtweightdef} for the tetrahedron graph $\Gamma$, we use the section of $\mathrm{Conf}_4\to C_4/S^1$ which identifies $C_4/S^1$ with $\mathrm{Conf}_2(\mathbb C\smallsetminus\{0,1\})$ as before. Here the first two points in a configuration in $\mathrm{Conf}_4$ are set to be $0$ and $1$, and we denote by $z$ and $w$ the remaining 2 points.

Consider the integrand \eqref{equ:tildebetadef}. In the sum over edges $e$, the only contributing term is $e=(1,2)$, since we fixed the position of the first and second point. In the sum over edges $e'$ multiple terms can contribute.
In fact, there are 2 possible types of terms, which we call type I and II, see Figure \ref{fig:tetrahedron}.

First of all, let us consider the two involutions $(z,w)\mapsto (w,z)$ and $(z,w)\mapsto (1-z,1-w)$ of $\mathrm{Conf}_2(\mathbb C\smallsetminus\{0,1\})$: it is easy to see that the four contributions of type II are related to each other by means of these involution or their composition.

We leave it to the reader to show the following Lemma using Stokes' Theorem (all boundary contributions vanish).
\begin{Lem}\label{l-reduction}
The form 
\[
\alpha=\log(|w|)\log(|z-w|)d\log(|z|)d\log(|z-1|)d\log(|w-1|)
\]
satisfies
\[
\int_{\mathrm{Conf}_2(\mathbb C\smallsetminus\{0,1\})}d\alpha=0.
\]
\end{Lem}

Together with Leibniz' rule, we hence find that
\begin{multline*}
 \int_{\mathrm{Conf}_2(\mathbb C\smallsetminus\{0,1\})}\log(|z-w|)d\log(|w|)d\log(|z|)d\log(|z-1|)d\log(|w-1|)
 \\=
 \int_{\mathrm{Conf}_2(\mathbb C\smallsetminus\{0,1\})}\log(|w|)d\log(|z-w|)d\log(|z|)d\log(|z-1|)d\log(|w-1|).
\end{multline*}
But since $\mathrm{Conf}_2 (\mathbb C\smallsetminus\{0,1\})$ is a complex manifold and hence only terms with an equal number of holomorphic and antiholomorphic form components contribute (cf. \cite{K}*{section 6.6.1}),
\begin{align*}
\log(|z-w|)d\log(|w|)d\log(|z|)d\log(|z-1|)d\log(|w-1|) &= \log(|z-w|)d\arg(|w|)d\arg(|z|)d\arg(|z-1|)d\arg(|w-1|) \\
 \log(|w|)d\log(|z-w|)d\log(|z|)d\log(|z-1|)d\log(|w-1|) &= \log(|w|)d\arg(|z-w|)d\arg(|z|)d\arg(|z-1|)d\arg(|w-1|).
\end{align*}
The right hand sides are the integrands for the type I and for the fourth contribution of type II to $c_\Gamma^{\frac 1 2}$, respectively.
Hence the contribution of type I equals the fourth contribution of type II and we may just compute the contribution of type I and multiply it by 5 to get the desired result $c_\Gamma^{\frac 1 2}$.

\begin{figure}
\begin{tikzpicture}[scale=.5]
\tikzstyle{ext}=[circle, draw, minimum size=5, inner sep=0]
\tikzstyle{int}=[circle, draw, fill, minimum size=5, inner sep=0]

\node [int] (v4) at (-6,-0.5) {};
\node [int] (v1) at (-6,1.5) {};
\node [int] (v2) at (-4,1.5) {};
\node [int] (v3) at (-4,-0.5) {};
\node [int, label=-90:{$0$}] (v5) at (-2,-0.5) {};
\node [int, label=90:{$z$}] (v6) at (-2,1.5) {};
\node [int, label=90:{$w$}] (v8) at (0,1.5) {};
\node [int, label=-90:{$1$}] (v7) at (0,-0.5) {};
\draw  (v1) edge (v2);
\draw  (v2) edge (v3);
\draw  (v1) edge (v3);
\draw  (v4) edge (v1);
\draw  (v4) edge (v2);
\draw  (v4) edge (v3);
\draw  (v5) edge (v6);
\draw  (v6) edge (v7);
\draw  (v7) edge (v8);
\draw  (v8) edge (v6);
\draw  (v5) edge (v8);
\begin{scope}[shift={(1,0.5)}]
\node [int, label=-90:{$0$}] (v5_1) at (2,-1) {};
\node [int] (v6_1) at (2,1) {};
\node [int] (v8_1) at (4,1) {};
\node [int, label=-90:{$1$}] (v7_1) at (4,-1) {};
\draw  (v5_1) edge (v6_1);
\draw  (v6_1) edge (v7_1);
\draw  (v7_1) edge (v8_1);
\draw  (v8_1) edge (v5_1);
\draw  (v6_1) edge[dashed] (v8_1);
\end{scope}

\begin{scope}[shift={(7,2)}, scale=0.6]
\node [int, label=-90:{$0$}] (v5_1) at (2,-1) {};
\node [int] (v6_1) at (2,1) {};
\node [int] (v8_1) at (4,1) {};
\node [int, label=-90:{$1$}] (v7_1) at (4,-1) {};
\draw  (v5_1) edge (v6_1);
\draw  (v6_1) edge[dashed]  (v7_1);
\draw  (v7_1) edge (v8_1);
\draw  (v8_1) edge (v5_1);
\draw  (v6_1) edge(v8_1);
\end{scope}

\begin{scope}[shift={(10,2)}, scale=0.6]
\node [int, label=-90:{$0$}] (v5_1) at (2,-1) {};
\node [int] (v6_1) at (2,1) {};
\node [int] (v8_1) at (4,1) {};
\node [int, label=-90:{$1$}] (v7_1) at (4,-1) {};
\draw  (v5_1) edge (v6_1);
\draw  (v6_1) edge (v7_1);
\draw  (v7_1) edge[dashed] (v8_1);
\draw  (v8_1) edge (v5_1);
\draw  (v6_1) edge (v8_1);
\end{scope}

\begin{scope}[shift={(10,-1)}, scale=0.6]
\node [int, label=-90:{$0$}] (v5_1) at (2,-1) {};
\node [int] (v6_1) at (2,1) {};
\node [int] (v8_1) at (4,1) {};
\node [int, label=-90:{$1$}] (v7_1) at (4,-1) {};
\draw  (v5_1) edge (v6_1);
\draw  (v6_1) edge (v7_1);
\draw  (v7_1) edge (v8_1);
\draw  (v8_1) edge[dashed] (v5_1);
\draw  (v6_1) edge (v8_1);
\end{scope}

\begin{scope}[shift={(7,-1)}, scale=0.6]
\node [int, label=-90:{$0$}] (v5_1) at (2,-1) {};
\node [int] (v6_1) at (2,1) {};
\node [int] (v8_1) at (4,1) {};
\node [int, label=-90:{$1$}] (v7_1) at (4,-1) {};
\draw  (v5_1) edge[dashed] (v6_1);
\draw  (v6_1) edge (v7_1);
\draw  (v7_1) edge (v8_1);
\draw  (v8_1) edge (v5_1);
\draw  (v6_1) edge (v8_1);
\end{scope}

\node at (-5,-3) {(a)};
\node at (-1,-3) {(b)};
\node at (4,-3) {Type I};
\node at (10.5,-4) {Type II};
\end{tikzpicture}
\caption{\label{fig:tetrahedron} (a) The tetrahedron graph $\Gamma$, (b) after contracting by the vector field generating the infinitesimal $S^1$-action and fixing two vertices at 0,1, all contributions come from this graph. 
The replacement of each of the remaining edges by the function $\log|\cdot - \cdot|$, represented by the dashed edge, yields 5 possible graphs. 
The type II contributions are all equal by symmetry; it is furthermore shown in Lemma~\ref{l-reduction} that the contribution of the type I graph is the same as that of any graph of type II.}
\end{figure}
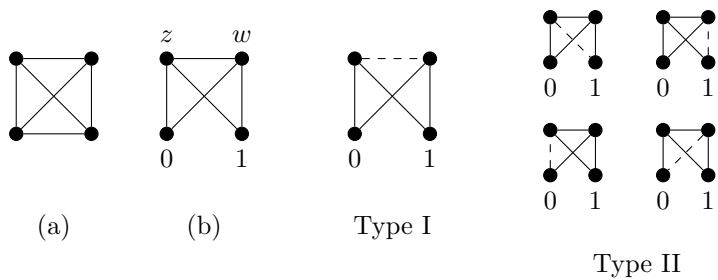

\begin{Prop}\label{p-polylog}
The function $F(w)$ defined by the following graph:\footnote{The three labeled vertices are to be kept fixed and the unlabeled vertex is to be integrated over.}
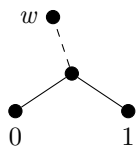
\[
\begin{tikzpicture}[scale=.5]
\tikzstyle{ext}=[circle, draw, minimum size=5, inner sep=0]
\tikzstyle{int}=[circle, draw, fill, minimum size=5, inner sep=0]
\node [int, label=-90:{$0$}] (v1) at (-2,0) {};
\node [int, label=-90:{$1$}] (v3) at (1,0) {};
\node [int] (v2) at (-0.5,1) {};
\node [int, label=180:w] (v4) at (-1,2.5) {};
\draw  (v1) edge (v2);
\draw  (v2) edge (v3);
\draw  (v2) edge[dashed] (v4);
\end{tikzpicture}
\]
is given by the following explicit formula:
\[
F(w) = \frac{2}{\pi^2}\Im\left(\Li_2(w) + \log(|w|)\log(1-w)\right),
\]
where $\mathfrak I$ denotes the imaginary part of a complex number.
\end{Prop}
\begin{proof}
Up to signs and constant factors, we have to compute explicitly the following integral:
\[
\int_z\log(|z-w|)d\mathrm{arg}(z)d\mathrm{arg}(z-1),
\]
where the integration domain is $\mathbb C\smallsetminus\{0,1,w\}$, for $w\neq 0,1$.
Standard arguments imply that we may actually integrate over $\mathbb C$. 
We split the integration domain
\[
\int_\mathbb C\log(|z-w|)d\mathrm{arg}(z)d\mathrm{arg}(z-1)=\int_{|z|< 1}\log(|z-w|)d\mathrm{arg}(z)d\mathrm{arg}(z-1)+\int_{|z|> 1}\log(|z-w|)d\mathrm{arg}(z)d\mathrm{arg}(z-1).
\]
Using the involution $z\mapsto 1/z$ we reduce the previous integral to 
\begin{equation}\label{eq-int-log}
\begin{aligned}
\int_\mathbb C\log(|z-w|)d\mathrm{arg}(z)d\mathrm{arg}(z-1)=&\int_{|z|< 1}\left(\log(|z-\overline w|)-\log(|1-z\overline w|)\right)d\mathrm{arg}(z)d\mathrm{arg}(z-1)+\\
&+\int_{|z|< 1}\log(|z|))d\mathrm{arg}(z)d\mathrm{arg}(z-1).
\end{aligned}
\end{equation}

The second term on the right-hand side of~\eqref{eq-int-log} vanishes, because
\[
\int_{|z|< 1}\log(|z|)d\mathrm{arg}(z)d\mathrm{arg}(z-1)=-\frac{1}{2 i}\sum_{n\geq 0}\int_{|z|< 1}\log(|z|)d\arg(z)z^ndz+\frac{1}{2 i}\sum_{n\geq 0}\int_{|z|< 1}\log(|z|)d\arg(z)\overline z^nd\overline z,
\]
using the Taylor expansion for the complex logarithm. Then it is immediate to verify the vanishing of the integral by resorting to polar coordinates on $\{|z|<1\}$.

We first consider $|w|\geq 1$: then, $|zw|\geq 1$, if $|z|\geq 1$; further, let us write $\{|z|\geq 1\}=\{|z|\geq |w|\}\cup \{|w|\geq |z|\geq 1\}$.
Let us also write 
\[
\log(|z-w|)=\log(|w|)+\log\left(\left|1-\frac{z}w\right|\right),\ \log(|z-w|)=\log(|z|)+\log\left(\left|1-\frac{w}z\right|\right),
\]
on $\{|z|<|w|\}$ and $\{|w|<|z|<1\}$ respectively.

Recalling the convergent power series expansion of the complex logarithm function and the (orientation-reversing) involution $z\mapsto \overline z$ of $\mathbb C$, we find for $|w|<1$,
\[
\begin{aligned}
&\int_{|z|<1}\left(\log(|z-\overline w|)-\log(|1-z\overline w|)\right)d\mathrm{arg}(z)d\mathrm{arg}(z-1)=-\frac{1}2\sum_{n\geq 1}\frac{1}m\left(\frac{1}{w^m}+\frac{1}{\overline w^m}\right)\left(\int_{|z|<|w|}z^m d\mathrm{arg}(z)d\mathrm{arg}(z-1)\right)-\\
&-\frac{1}2\sum_{n\geq 1}\frac{1}m\left(w^m+\overline w^m\right)\left(\int_{|w|<|z|< 1}\frac{1}{z^m}d\mathrm{arg}(z)d\mathrm{arg}(z-1)\right)+\frac{1}2\sum_{n\geq 1}\frac{1}m\left(w^m+\overline w^m\right)\left(\int_{|z|<1}z^m d\mathrm{arg}(z)d\mathrm{arg}(z-1)\right).
\end{aligned}
\]
Using polar coordinates, it is easy to verify, for $|w|<1$, 
\[
\begin{aligned}
\int_{|z|<|w|}z^m d\mathrm{arg}(z)d\mathrm{arg}(z-1)&=-i\pi\ \frac{w^m \overline w^m}{2m},\ &\ \int_{|w|<|z|<1}\frac{1}{z^m}d\mathrm{arg}(z)d\mathrm{arg}(z-1)&=i\pi\ \log(|w|),\\
\int_{|z|<1}z^m d\mathrm{arg}(z)d\mathrm{arg}(z-1)&=-i\pi\ \frac{1}{2m},\ m\geq 1,
\end{aligned}
\]
whence
\[
\int_{|z|<1}\left(\log(|z-\overline w|)-\log(|1-z\overline w|)\right)d\mathrm{arg}(z)d\mathrm{arg}(z-1)=\pi\mathfrak I\left(-\Li_2\left(\frac{1}w\right)+\log(|w|)\log\left(1-\frac{1}w\right)\right).
\]
Similar computations hold true for $|w|\geq 1$.

The final result is consequence of the formul\ae\ relating $\Li_n(w)$ and $\Li_n(1/w)$, for $n$ a positive integer.
\end{proof}
One can easily see that the function $F$ computed in Proposition~\ref{p-polylog} has the following properties:
\begin{enumerate}
\item[$i)$] $F(\overline w) = -F(w)$;
\item[$ii)$] $F$ is (strictly) positive for $\Im w>0$ and strictly negative for $\Im w<0$.
\end{enumerate}

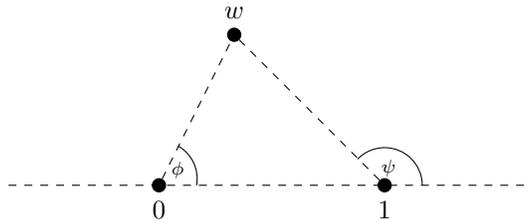
\begin{figure}
\centering
\begin{tikzpicture}
\tikzstyle{ext}=[circle, draw, minimum size=5, inner sep=0]
\tikzstyle{int}=[circle, draw, fill, minimum size=5, inner sep=0]

\draw[dashed] (-4,-1) -- (3,-1);
\node [int, label=90:{$w$}] (v2) at (-1,1) {};
\node [int, label=-90:1] (v3) at (1,-1) {};
\node [int, label=-90:0] (v1) at (-2,-1) {};
\draw  (v1) edge[dashed] (v2);
\draw  (v2) edge[dashed] (v3);
\draw (v1) +(.5,0) arc (-9.9974:60:0.5);
\draw (v3) +(.5,0) arc (0:136:0.5);

\node at (-1.7493,-0.8021) {\scriptsize  $\phi$};
\node at (1.0552,-0.7651) {\scriptsize $\psi$};
\end{tikzpicture}
\caption{\label{fig:angles} The integral to be computed in order to determine the coefficient of the tetrahedron graph is $\int F(w) d\phi(w) d\psi(w)$. }
\end{figure}
The contribution of the type I graph is an integral 
\[
\int_{w\in\mathbb C\smallsetminus\{0,1\}} F(w)d\phi(w)d\psi(w)
\]
see Figure \ref{fig:angles}. The form $d\phi(w)d\psi(w)$ is a positive volume form on the lower half-plane, and a negative volume form on the upper half-plane. 
By Properties $i)$ and $ii)$ of $F$ above we immediately see that the corresponding integral is positive.

\begin{proof}[proof of Proposition \ref{p-w-tetra}]
We have to compute the integral 
 \[
\int_w F(w) d\phi(w) d\psi(w) = 
\int_{|w|<1} F(w) d\phi(w) d\psi(w)
+
\int_{|w|>1} F(w) d\phi(w) d\psi(w).
\]
We compute the first term by using a suitable Taylor expansion of the terms in the integrand. The second term we write as 
\[
\int_{|w|>1} F(w) d\phi(w) d\psi(w)=\int_{|u|<1} F(1/u) d\phi(1/u) d\psi(1/u)
\]
by using the map $u=\frac 1 z$. 
Then, we may again write a Taylor expansion of the terms on the right-hand side and compute the integral explicitly. 

First define 
\[
f(w):=\Li_2(w)+\log|w|\log(1-w)   
\]
so that 
\[
F(w) = \frac{2}{\pi^2}\Im\left( f(w) \right).
\]
The function $f(w)$ has the following expansion for $|w|<1$:
\[
f(w) = \sum_{n\geq 1} w^n \left( \frac{1}{n^2} - \log|w| \frac{1}{n}   \right).
\]

By using the following well-known identity for the dilogarithm:\footnote{Here $B_2$ is the second Bernoulli polynomial.}
\[
\Li_2(w)+\Li_2(1/w) 
= 2\pi^2 B_2\!\left(\frac 1 2 -\frac{\log (-1/w)}{2\pi i}\right)
=-\frac{\pi^2}{6} - \frac{\log^2(-1/u)}{2}
\]
we deduce that 
\[
f(1/u) = -f(u) -\frac{\pi^2}{6} + \frac{1}{2} \log(-1/u) \log(-\bar u).
\]
The two terms on the right are real, and since we take the imaginary part of $f$, they can be omitted, i.~e., 
\[
f(1/u) = -f(u) +( \text{irrelevant}).
\]


Similarly the forms $d\phi(w) d\psi(w)$ and $d\phi(1/u) d\psi(1/u)$ can be expanded as follows for $|z|<1$ (resp. $|u|<1$):
\begin{align*}
d\phi(w) d\psi(w) &=  \frac{d w d\bar w}{16 \pi^2}\left(
\frac{1}{w(1-\bar w)} 
-
\frac{1}{\bar w(1-w)} 
\right)
=
 \frac{dwd\bar w}{16 \pi^2}
 \sum_{m\geq 0} \left( \frac{\bar w^m}{w}-\frac{w^m}{\bar w}\right)
\\
d\phi(1/u) d\psi(1/u) &= 
\frac{d u d\bar u}{16 \pi^2 u \bar u}\left(
\frac{1}{1-u} 
-
\frac{1}{1-\bar u}
\right)
=
\frac{d u d\bar u}{16 \pi^2 u \bar u}
\left(
u^m
-
\bar u^m
\right).
\end{align*}
Inserting, we have to calculate the following integral over the unit disk:
\begin{align*}
&\frac{2}{\pi^2}
\Im \int_{|z|<1}\frac{dzd\bar z}{16 \pi^2 z\bar z}
\sum_{m\geq 0}\left((1+\bar z)\bar z^m  -(1+z)z^m \right)
\sum_{n\geq 1} z^n \left( \frac{1}{n^2} - \frac{1}{n}\log|z|    \right)
\\ &=
\frac{2}{\pi^2}
\Im \int_{|z|<1}\frac{dzd\bar z}{16 \pi^2 z\bar z}
\sum_{m\geq 0}\sum_{n\geq 1}(1+\bar z)\bar z^m
 z^n \left( \frac{1}{n^2} - \frac{1}{n}\log|z|    \right)
 \\ &=
\frac{2}{\pi^2}
\Im 
(-4\pi i)
\int_{0}^1\frac{r dr}{16 \pi^2 r^2}
\sum_{n\geq 1}2 r^{2n} \left( \frac{1}{n^2} - \frac{1}{n}\log r    \right)
 \\ &=
-\frac{1}{\pi^3}
\Im 
 i
\sum_{n\geq 1} \left( \frac{1}{2 n^3} - \frac{1}{n(2n)^2} \right)
 \\ &=
-\frac{1}{4 \pi^3}
\zeta(3)
\end{align*}

This yields the desired claim.
\end{proof}

\section{Singularities of the AT connection}\label{app:ATSing}

The goal of this Appendix  is to show that the connection one-forms of the Alekseev-Torossian connection on configuration space have at most logarithmic singularities. The result is used in the proof of Proposition \ref{prop:fastdecay} used in section \ref{sec:prooffastdecay}, but may as well be of interest in its own right.

\subsection{Introduction and basic definitions}
We will decompose the Alekseev-Torossian connection \eqref{eq-AT-bound-1} on $\underline C_n(\bbC)$ as follows 
\[
 \nabla^{\frac 1 2}_n = d - \omega_{{\rm AT},k}^{\frac 1 2} = d - \sum_{j=1}^n A_j dz_j + \bar A_j d\bar z_j
\]
where $A_j$ is a smooth function on $\Conf_n(\bbC)$ with values in $\sder_n$.
Recall that $\omega_{{\rm AT},k}^{\frac 1 2}$ is given by the sum-of-graphs formula \eqref{eq-AT-bound-1}.
We will furthermore decompose the coefficients as
\[
 \vartheta^t_\Gamma  = \sum_{j=0}^n A_{\Gamma,j} dz_j + \bar A_{\Gamma,j} d\bar z_j.
\]

\begin{Thm}\label{thm:logsing}
 For any connected $\sder_n$-graph $\Gamma$ with at least one internal vertex and any $j$ there are constants $C\in \bbR$ and $N\in \bbN_0$ such that
\[
 |A_{\Gamma,j}(z)| \leq C \frac{ ( 1+\sum_{i<j} |\log|z_i-z_j||)^N }{\sum_{i<j}|z_i-z_j|}.
\]
In other words, the connection forms of the AT connection have at most logarithmic singularities as points collide.
\end{Thm}

\begin{Rem}
 Connectedness is important since otherwise there may be an additional pole at configurations where the points corresponding to a connected component collapse. (In fact, by definition of $\sder$, there will be only one connected component with edges, plus possible multiple external vertices of valence 0.)
\end{Rem}

\subsection{Forms with logarithmic singularities}
\begin{Def}
 We say that a 1-form $\alpha$ on $\Conf_n(\mathbb C)$ has logarithmic singularities if 
\[
 \alpha = \sum_j \alpha_j dz_j + \beta_j d \bar z_j
\]
and there are constants $C$ and $N$ such that for all $k$:
\begin{equation}\label{equ:logsingdef}
  |\alpha_{k}|+|\beta_k| \leq C \frac{ ( 1+\sum_{i<j} |\log|z_i-z_j||)^N}{\sum_{i<j}|z_i-z_j|}.
\end{equation}
\end{Def}

\subsection{Coordinate systems and a criterion for logarithmic singularities}

We call a rooted tree with leaf set $\{1,\dots,n\}$, such that all internal vertices have at least two children and such that for each internal vertex two children are marked by labels $0$ and $1$ an \emph{admissible tree}.
We distinguish here between internal vertices and leaves: observe that the root is considered an internal vertex.
We consider a partial order on the vertices of $T$ by declaring the root to lie at level $0$, and, inductively, vertices to lie at level $k\geq 1$, if they are directly connected to vertices lying at level $k-1$.
Observe that, with this partial order, every vertex at $k\geq 1$ is connected to exactly one internal vertex at level $k-1$ and, if it is not a leaf, to at least two distinct vertices at level $k+1$.
For any internal vertex $\nu$, say at level $k$, we consider the family of edges $\{e(\nu)\}$ connecting $\nu$ to the vertices of level $k+1$.
To each admissible tree $T$ we may associate a system of local coordinates. 
It involves variables $\rho_e=\rho_{e(\nu)}\in \bbC^\times$, one for each internal vertex $\mu$, and variables $Z_\nu$, one for each vertex $\nu$ not labelled by 0 or 1. In fact, we will set $Z_\nu=0$ or $Z_\nu=1$ if the vertex $\nu$ is labelled by 0 or 1 for notational convenience.
The mapping from this set of coordinates to a configuration is realized by the formula
\begin{equation}\label{equ:zcoord}
 z_j = \sum_{\nu \vartriangleright j} Z_\nu \prod_{e \blacktriangleright \nu} \rho_e
\end{equation}
where the notation $\nu \vartriangleright j$ means that we sum over all internal vertices $\nu$ which are ancestors of the leaf labeled $j$, including the leaf $j$ ({\em i.e.} the vertices lying at a level lower or equal than the one of $j$), and the notation $e \blacktriangleright \nu$ means that we consider all internal edges $e$ connecting $\nu$ to the root of $T$.

\begin{Rem}
 Note that the above trees have nothing to do with the (internal) trees occurring in the definition of $\sder_n$. If there is a risk of confusion, we will call the trees above \emph{coordinate-trees}.
\end{Rem}

Furthermore, for each $R>1$ we associate a subset $U^T_R\subset \overline C_n$, for which the coordinates can vary within the following bounds:
\begin{align*}
0 < |\rho_e| < \frac 1 R \\
|Z_\mu|<1 \\
|Z_\mu - Z_\nu| > \frac {10} R
\end{align*}
where the last condition is imposed only to vertices $\mu$, $\nu$ having a common parent. \footnote{The 10 here is a relatively arbitrary "sufficiently large" constant.}

Note that the sets $U^T_R$ trivially cover the interior $C_n$ of $\overline C_n$, since for example the sets $U^{T_0}_R$ do, for $T_0$ the trivial tree with all leaves attached to one vertex and $R>1$. Less trivially, one has the following result.

\begin{Lem}\label{lem:finite cover}
There is a finite set $S$ of pairs $(T,R)$ such that $C_n$ is covered by the corresponding $U^T_R$, i.~e., 
$$\bigcup_{(T,R)\in S} U^T_R = C_n.$$
\end{Lem}
\begin{proof}
Write $\rho_e= r_e e^{i\phi_e}$ for the coordinates above.
Consider slightly larger subsets $\tilde U^T_R\subset \overline C_n$ obtained by allowing the $r_e$ to approach zero, and in that case assign the corresponding boundary point on $\overline C_n$. We claim that we can find a finite set $S$ of pairs $(T,R)$ such that  $\bigcup_{(T,R)\in S}\tilde U^T_R = \overline C_n$.
Indeed, consider an arbitrary boundary point of $\overline C_n$. The corresponding boundary stratum is labelled by a tree $T_1$, and the above coordinates are defined such that the set $\tilde U^{T_1}_R$ for some $R$ contains a neighborhood of that boundary point. Hence $\overline C_n$ is covered by the collection of all $\tilde U^T_R$. Hence, by compactness of $\overline C_n$ it is covered by a finite subset, corresponding to a set $S$ of pairs $(T,R)$. But since the $\tilde U^T_R$ differs from $U_R^T$ only by adding boundary points of $\overline C_n$ we arrive at the desired conclusion $\cup_{(T,R)\in S} U^T_R = C_n$.
\end{proof}

We will denote the coordinates $\rho_e$ and $Z_\nu$ collectively by $x_1,x_2,\dots$ for simplicity of notation.
The-one forms $\alpha$ on $\Conf_n$ iwe are interested in are $\bbC^\times\ltimes \bbC$-basic. Hence, in particular they are $\bbR^+\ltimes \bbC$-basic and descend to $C_n\cong \Conf_n/\bbR^+\ltimes \bbC$.
For such one-forms the condition of having logarithmic singularities can be translated in a set of conditions that can be checked on the charts $U^T_R$.

\begin{Lem}\label{lem:altcondition}
 Suppose the 1-form $\alpha$ on $\Conf_n$ is $\bbC^\times\ltimes \bbC$-basic. Suppose further that the restriction of $\alpha$ to each $U^T_R$, for $R>1$ and $T$ ranging over admissible trees, satisfies the following estimate. 
If $\alpha$ is 
\[
 \alpha = \sum_j \alpha_j(x_1,x_2,\dots) dx_j
\]
for $x_j$ the coordinates on $U^T_R$ as above, then there are numbers $C$ and $N$, possibly depending on $R$ and $T$ such that 
\be{equ:altestimate}
 |\alpha_k(x_1,x_2,\dots)| \leq C (1+\sum_{j}|\log|\rho_j|| )^N
\ee
for all $k$ on $U^T_R$. 
Then $\alpha$ has only logarithmic singularities.

Conversely, if $\alpha$ has only logarithmic singularities, then the above estimate is satisfied.
\end{Lem}
\begin{proof}
First note that the converse direction is trivial, one just needs to insert the expression \eqref{equ:zcoord} for $z_j$ in terms of the $x_i$.

For the forward direction choose a finite set of pairs $(T,R)$ such that the corresponding $U^T_R$ cover $C_n$. The existence of such a finite set $S$ is guaranteed by Lemma \ref{lem:finite cover} above. It will be sufficient to show that on each $U_R^T$ the defining conditions \eqref{equ:logsingdef} hold for some $C$, $N$ depending on $(T,R)$ provided the estimate in the Lemma holds. Then, since we consider only a finite number of such $(T,R)$, the condition \eqref{equ:logsingdef} holds for $C$ and $N$ large enough.

Now the $x_j$ may be expressed as rational functions in the $z_j$. These rational functions have poles when certain points collide. However, by definition of $U_R^T$ the poles lie all outside the closure of $U_R^T$. Hence inserting the expression for $x_j$ in terms of the $z_i$ in the above estimate yields (essentially) the defining estimate \eqref{equ:logsingdef} for forms with logarithmic singularities.
\end{proof}

\subsection{Fiber integrals, and the proof of Theorem \ref{thm:logsing}}
We will show the main Theorem by an induction on the number of internal vertices.
For one internal vertex the form $A_\Gamma$ may be explicitly computed, see Proposition \ref{p-AT-3}, and one verifies easily that this form has only logarithmic singularities.

Suppose $\Gamma$ is an $\sder_n$-graph with at least two internal vertices. We pick some internal vertex and make it external with label $n+1$, obtaining a ``join'' of three graphs $\Gamma_1, \Gamma_2, \Gamma_3\in \sder_{n+1}$.
By definition, the form $A_\Gamma$ satisfies
\be{equ:AGamma}
 A_\Gamma = \int_{z_{n+1}\in \bbC}  A_{\Gamma_1}A_{\Gamma_2}A_{\Gamma_3}.
\ee
Our goal is to show that $A_\Gamma$ has only logarithmic singularities. By induction we may assume that all $A_{\Gamma_j}$ either have a single edge, or $A_{\Gamma_j^0}$ has only logarithmic singularities, where $\Gamma_j^0$ is the connected component. 

We will in fact check the conditions of Lemma \ref{lem:altcondition} above.
So fix some admissible (coordinate-)tree $T$, and $R>0$. Our goal is to check the estimates \eqref{equ:altestimate} on $U^T_R$.
For fixed $(z_1,\dots , z_n)\in U^T_R$ we will decompose the integration region into a union of subsets $V^{T'}\subset \bbC$ (depending on $z_1,\dots , z_n$), where $T'$ is a graph obtained from $T$ by adding one leaf labelled $n+1$ at either of the following positions:
\begin{itemize}
 \item We may add the leaf at the root, creating a new internal vertex. Pictorially:  
 \[
  T'= 
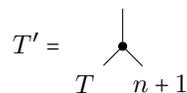
  \begin{tikzpicture}[baseline=-.65ex]
   \node[int] (v) at (0,0) {};
   \node (T) at (-.5,-.5) {$T$};
   \node (w) at (.5,-.5) {$n+1$};
   \draw (v) edge (w) edge (T) edge +(0,.5);
  \end{tikzpicture}
 \]
\item We may add the leaf at an existing internal vertex.
Pictorially: 
 \[
  T'= 
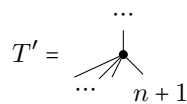
  \begin{tikzpicture}[baseline=-.65ex]
   \node[int] (v) at (0,0) {};
   \node (T) at (-.5,-.5) {$\cdots$};
   \node (w) at (.5,-.5) {$n+1$};
   \node (u) at (0,.5) {$\cdots$};
   \draw (v) edge (w) edge (T) edge (u) edge +(-.15,-.3) edge +(-.65,-.3);
  \end{tikzpicture}
 \]
\item We may add the leaf at an existing internal edge, creating a new internal vertex. Pictorially:
 
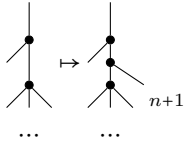
\[
  \begin{tikzpicture}[baseline=-.65ex]
   \node[int] (v) at (0,.3) {};
   \node[int] (w) at (0,-.3) {};
   \node (T) at (0,-1) {$\cdots$};
   \draw (v) edge (w) edge +(-.3,-.3) edge +(0,.5) (w) edge +(-.3,-.3) edge +(0,-.3) edge +(.3,-.3);
  \end{tikzpicture}
  \mapsto
    \begin{tikzpicture}[baseline=-.65ex]
   \node[int] (v) at (0,.3) {};
   \node[int] (vv) at (0,0) {};
   \node (ww) at (.75,-.5) {$\scriptstyle n+1$};
   \node[int] (w) at (0,-.3) {};
   \node (T) at (0,-1) {$\cdots$};
   \draw (v) edge (vv) edge +(-.3,-.3) edge +(0,.5) (w) edge +(-.3,-.3) edge +(0,-.3) edge +(.3,-.3)
         (vv) edge (w) edge (ww);
  \end{tikzpicture}
 \]
\item We may add the leaf at an existing leaf, creating a new internal vertex. Pictorially:
 
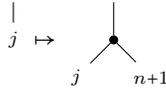
\[
  \begin{tikzpicture}[baseline=-.65ex]
   \node (v) at (0,0) {$\scriptstyle j$};
   \draw (v) edge +(0,.5);
  \end{tikzpicture}
  \mapsto
    \begin{tikzpicture}[baseline=-.65ex]
   \node[int] (v) at (0,0) {};
   \node (w) at (-.5,-.5) {$\scriptstyle j$};
   \node (ww) at (.5,-.5) {$\scriptstyle n+1$};
   \draw (v) edge +(0,.5) edge (w) edge (ww);
  \end{tikzpicture}
 \]
\end{itemize}

For each such tree $T'$ we may extend our coordinate system on (a subset of) $\overline C_n$ defined by $T$ to a coordinate system on (a subset of) $\overline C_{n+1}$ defined by $T'$. Concretely, the coordinate system changes as follows. 
\begin{itemize}
 \item If $T'$ is obtained by adding the additional leaf at an existing internal vertex, we just add another variable $Z_{n+1}$ that is related to $z_{n+1}$ by the formula
\[
 z_{n+1}= \sum_{\nu \vartriangleright n+1} Z_\nu \prod_{e \blacktriangleright \nu} \rho_e .
\]
\item If we add the leaf at an existing leaf, one variable $Z_{n+1}:=\rho_e$ is added for the newly created internal edge and the formula for $z_{n+1}$ is similar.
\item If we add the leaf at an internal edge, some coordinate $\rho_e$ is replaced by two coordinates $\rho_{e'}, \rho_{e''}$ such that $\rho_{e}=\rho_{e'} \rho_{e''}$. We will denote the additional coordinate $Z_{n+1}:=\rho_{e'}$, the formula for $z_{n+1}$ is again that from above.
\end{itemize}
Now the subset $V^{T'}$ of the fiber over some fixed point $(z_1,\dots, z_n)\in U^T_R$ is defined as follows. We require that $|Z_{n+1}-Z_j|>\frac{1}{2R}$ for vertices on the same level. Furthermore we require that $|Z_{n+1}|\leq 2R$. Finally we require that $|z_{n+1}-z_j|\geq R^{-d(j)} \prod_{e \blacktriangleright n+1,j} |\rho_e|$ (explain notation).
If the leaf is added at the root, we just require $|z_{n+1}|>2R$.

The subsets $V^{T'}$ cover the fiber over $(z_1,\dots, z_n)$. Note that they are not disjoint, but have overlaps (we don't care, at the end this just means that we will overestimate the left hand side of \eqref{equ:AGamma}).
Furthermore, all $V^{T'}$ are bounded, except the one for which $n+1$ is added at the root.

Clearly it suffices to show that for each $T'$ there are constants $C$ and $N$ such that
\[
\left| \int_{z_{n+1}\in V^{T'}} A_{\Gamma_1}A_{\Gamma_2}A_{\Gamma_3} 
\right|
\leq 
C (1+\sum_{e}|\log|\rho_e|| )^N
\]
for all $(z_1,\dots, z_n)\in U^T_R$. Here we define $|\alpha|$ for a form $\alpha$ to be the sum of the absolute values of the coefficient functions in the coordinate system provided; in our case in the coordinate system given by the $Z_\nu$ and $\rho_e$ as above.

We will show the estimnate by using the induction hypothesis to bound the absolute values of the coefficient functions occurring in $A_{\Gamma_1},A_{\Gamma_2},A_{\Gamma_3}$ by powers of logarithms. Expanding a power if necessary and ignoring irrelevant constants, it suffices to bound integrals of the form
\[
 \int_{V^{T'}}
\frac{|\prod_{i<j\leq n+1} \log|z_i-z_j||^N }{\prod_{k=1}^3 \sum_{i,j\in V_{ext}(\Gamma_k^0)} |z_i-z_j|} d^2z_{n+1}
\]
where $V_{ext}(\Gamma_k^0)$ denotes the set of external vertices in $\Gamma_k$ that have valence $\geq 1$, necessarily including $n+1$.

Let us express the potentially singular terms in the integrand in our coordinate system and find bounds if possible.
We will treat first the case of bounded $V^{T'}$.
\begin{itemize}
 \item First $|\log|z_i-z_{n+1}||\leq (const) \sum_{e \blacktriangleright i,j} | \log|\rho_e||$.
So these terms contribute powers of $\log|\rho_e|$ that may be taken out of the integral, or a power of $\log|Z_{n+1}|$.
\item There may be one $\frac 1 {|z_i-z_{n+1}|}\leq (const) \prod_{e \blacktriangleright i, n+1} \frac 1 {|\rho_e|}$ for each subgraph $\Gamma_k$ that is a single edge. 
\item By similar considerations, a subgraph $\Gamma_k$ may produce a divergent factor $(const)\prod_{e \blacktriangleright V_{ext}(\Gamma_k^0), n+1} \frac 1 {|\rho_e|}$, where the product is over internal edges $e$ of $T'$ that lie above all vertices in $T'$ corresponding to valence $\geq 1$-vertices of $\Gamma_k$ and above the added leaf $n+1$.
\item So in total, there may be one, two or three such terms with poles, contributing three products $\prod \frac 1 {|\rho_e|}$. However, note that any $|\rho_e|$ may occur in at most two such products, because there are no internal edges $e$ of $T$ that lie above all vertices of $\Gamma$ by the assumption of connectedness of $\Gamma$.
Note furthermore that $d^2z_{n+1}=\prod_{e\blacktriangleright n+1} |\rho_e|^2 d^2 Z_{n+1}$, where the product is over all internal edges above $n+1$. Hence all $\rho_j$ involved in the potentially divergent terms occur twice in the numerator, so no pole in $\rho_j$ is produced in effect.
\end{itemize}
Taking constants and the $|\rho_e|$ and $\log|\rho_e|$-terms out of the integral and disregrading them, it eventually suffices to show that the integrals 
\[
 \int_{z_{n+1}\in V^{T'}}
|\log|Z_{n+1}| |^N d^2 Z_{n+1}
\]
are finite and bounded uniformly in $(z_1,\dots, z_n)$. Note that the integrand does not involve $(z_1,\dots, z_n)$ any more, but the integration domain does. However, since $|\log|Z_{n+1}| |^N$ is integrable on all bounded domains, we may just enlarge the integration domain to, say $|Z_{n+1}|\leq 10 R$ and bound the integral by a constant, independent of $(z_1,\dots, z_n)$.

Next consider the unbounded case, i. e., the case when the leaf $n+1$ is added at the root of $T$ to produce $T'$. Here all factors $\log|z_i-z_{n+1}|$ may be estimated by $\log|z_{n+1}|+(const)$, and all factors $\frac 1 {|z_i-z_{n+1}|}$ may be estimated by $\frac {(const)}{|z_{n+1}|}$. Finally, again after expanding powers where necessary and taking out and disregrading constants and factors $\log|z_i-z_j|$ for $i,j<n+1$, one obtains integrals of the form
\[
 \int_{|z_{n+1}|\geq 2R} \frac{\log|z_{n+1}|^{N'}}{|z_{n+1}|^3}d^2z_{n+1}.
\]
These integrals are clearly finite constants, independent of the configuration $(z_1,\dots, z_n)$.
Hence we are done. 
\hfill\qed

\end{document}